\documentclass[11pt]{amsart}
\usepackage{amssymb}
\usepackage{amsmath}
\allowdisplaybreaks
\usepackage{amsfonts}
\usepackage{mathrsfs}
\usepackage{psfrag}
\usepackage{verbatim}
\usepackage{color}
\usepackage{enumitem}
\usepackage{bm}
\usepackage[usenames,dvipsnames]{xcolor}
\usepackage{dsfont}
\usepackage{overpic}
\usepackage{subcaption}
\usepackage{multirow}
\usepackage{float}
\usepackage{tikz}
\usetikzlibrary{trees}
\usepackage[section]{placeins}

\usepackage[colorlinks=true]{hyperref}

\hypersetup{urlcolor=blue, citecolor=MidnightBlue}

\numberwithin{equation}{section} 

\DeclareGraphicsExtensions{}
\theoremstyle{plain}
\newtheorem{theorem}{Theorem}[section]

\newtheorem{lemma}[theorem]{Lemma}
\newtheorem{proposition}[theorem]{Proposition}

\theoremstyle{definition}
\newtheorem{assumption}[theorem]{Assumption}

\newtheorem{example}[theorem]{Example}

\theoremstyle{remark}
\newtheorem{remark}[theorem]{Remark}

\newcommand{\BR}{\mathbb{R}}
\newcommand{\BP}{\mathbb{P}}
\newcommand{\BE}{\mathbb{E}}

\newcommand{\BN}{\mathbb{N}}

\newcommand{\ind}{\mathds{1}}

\begin{document}
\title[Quantitative Fluctuation Analysis for Continuous-Time SGD via Malliavin Calculus]{Quantitative Fluctuation Analysis for Continuous-Time Stochastic Gradient Descent via Malliavin Calculus}

\author[S. Bourguin]{}
\address{LAMA, Universit\'e Gustave Eiffel and Department of Mathematics and Statistics\\ Boston University}
\email{solesne.bourguin@univ-eiffel.fr, bourguin@math.bu.edu}
\thanks{S.B was partially supported by Simons Foundation Award
  635136.}

\author[S. S. Dhama]{}
\address{Department of Mathematics and Statistics\\ Boston University}
\email{ssdhama@bu.edu, shivamsd.maths@gmail.com}
\thanks{S.S.D. was partially supported by NSF DMS-2311500.}

\author[K. Spiliopoulos]{}
\address{Department of Mathematics and Statistics\\ Boston University}
 \email{kspiliop@math.bu.edu}
\thanks{K.S was partially supported by NSF DMS-2311500 and NSF DMS-2107856.}


\maketitle

\centerline{\scshape Solesne Bourguin, Shivam Singh Dhama, Konstantinos Spiliopoulos}

\begin{abstract}
In this paper, we establish a Quantitative Central Limit Theorem ({\sc qclt}) for the Stochastic Gradient Descent in Continuous Time ({\sc sgdct}) algorithm, whose parameter updates are governed by a stochastic differential equation. We derive an explicit rate at which the {\sc sgdct} iterates converge, in the Wasserstein metric, to a critical point of the objective function. This rate is driven primarily by the magnitude of the learning rate: for a fixed convexity constant of the objective function, smaller learning rates lead to slower convergence.
Our approach relies on tools from Malliavin calculus. In particular, we apply a second-order Poincar\'e inequality and obtain explicit bounds by estimating the first- and second-order Malliavin derivatives separately. Controlling the second-order derivative requires several delicate calculations and a careful sequence of decompositions in order to achieve sharp estimates. We complement the theoretical results with several numerical experiments that illustrate the predicted convergence behavior.
\end{abstract}

\section{Introduction}
Training models on large-scale, continuously evolving datasets poses a
fundamental challenge in both machine learning and deep learning. One
effective way to address this challenge is Stochastic Gradient Descent
in Continuous Time ({\sc sgdct}), an optimization algorithm designed
to efficiently handle streaming data -- see, for example, \cite{KS17, siri_spilio_2020, sharrock2023online,pavliotis2025filtered}. Unlike traditional batch optimization methods, which require access to the full dataset in order to compute updates, {\sc sgdct} follows a noisy descent direction informed by incoming data, updating parameters incrementally using gradients computed from mini-batches. This can substantially reduce computational cost and allows the model to adapt in real time. In addition, {\sc sgdct} is used to estimate unknown parameters or functions in stochastic differential equations ({\sc sde}), making it a natural tool for dynamical systems in which data arrive continuously.

To formulate the question of primary interest, we consider a diffusion process $X_t \in \mathscr{X} \subseteq \BR$ governed by the {\sc sde}:
\begin{equation}\label{E:Process-X}
dX_t = f^{*}(X_t) \thinspace dt + \sigma \thinspace dW_t,
\end{equation}
where \(f^*(x)\) is an unknown function, \(\sigma\) is a known constant, and the process \(W_t\) denotes a standard Brownian motion. Given observations of the process $X_t$, we estimate a model \(f(x,\theta)\) for the unknown function \(f^*(x)\) using {\sc sgdct}. To derive the {\sc sde} governing the parameter updates $\theta,$ we assume that the process $X_t$ admits a unique invariant measure $\mu$. Under this assumption, we define the objective function and the distance function as follows:
\begin{equation}\label{E:Objective-distance-function}
\bar{g}(\theta) \triangleq \int_{\mathscr{X}}g(x,\theta)\mu(dx), \qquad g(x,\theta) \triangleq  \frac{1}{2}|f(x,\theta)-f^*(x)|^2_{\sigma^2} = \frac{1}{2}\sigma^{-2}|f(x,\theta)-f^*(x)|^2.
\end{equation}
The function $\bar{g}(\theta)$ decreases when the parameter $\theta$ is updated in the descent direction \( -\bar{g}_\theta(\theta) \), leading to the algorithm:
\begin{equation*}
{d \theta_t} = \alpha_t f_\theta(X_t, \theta_t){\sigma}^{-2}[dX_t - f(X_t, \theta_t)] \thinspace dt,
\end{equation*}
where $\alpha_t$ denotes the learning rate. A more detailed formulation of this algorithm is provided in Section \ref{S:Problem-Set-up}.

The asymptotic behavior of the {\sc sgdct} updates \(\theta_t\) as \(t\) approaches infinity has been a central topic in the literature. In particular, the authors of \cite{siri_spilio_2020} study the convergence of {\sc sgdct} updates by establishing an \(L^p\) convergence rate and analyzing the fluctuations of \(\theta_t\) under certain assumptions. For the fluctuation analysis, they identify the limiting behavior \textit{in distribution} of the rescaled process
\[
\mathsf{F}_t \triangleq \sqrt{t} \thinspace (\theta_t - \theta^*)
\]
as time increases to infinity, where \(\theta^*\) is a critical point of the objective function \(\bar{g}\) defined in \eqref{E:Objective-distance-function}. Notably, the fluctuation result in \cite{siri_spilio_2020} is qualitative and can be interpreted as a qualitative central limit theorem ({\sc clt}). A natural and important extension of this qualitative {\sc clt} is therefore to investigate the quantitative fluctuation behavior of the process \(\theta_t\).

Accordingly, the primary goal of the present work is to establish a Quantitative Central Limit Theorem ({\sc qclt}) for the process $\theta_t$, in the sense of deriving an explicit rate at which the rescaled fluctuation process $\mathsf{F}_t$ converges to its limiting distribution in the Wasserstein metric (Theorem \ref{T:Main-theorem}). To obtain this rate, we employ tools from Malliavin calculus. In particular, we use a second-order Poincar\'e inequality that yields quantitative convergence rates in terms of bounds on the first- and second-order Malliavin derivatives of the pre-limit process, with Gaussian limiting behavior \cite[Theorem 2.1]{vido_2020}. A key difficulty is therefore to derive explicit bounds for these Malliavin derivatives. To address this challenge, we construct appropriate Poisson equations to control the fluctuation terms that arise, and we apply H\"older's inequality repeatedly throughout the analysis.

This work contributes in several respects. First, we provide an explicit convergence rate for the fluctuations of the {\sc sgdct} algorithm, thereby making the result quantitative rather than merely qualitative. The rate depends on the interplay between the learning rate magnitude $C_\alpha$ and the strong convexity constant $C_{\bar{g}}$ of the objective function $\bar{g}$. For a fixed convexity constant, a larger learning rate yields a faster convergence rate. Second, our analysis allows the {\sc sgdct} updates to depend on the dynamics of the underlying stochastic process $X_t$, thereby introducing temporal correlation into the data stream. This stands in contrast to standard discrete-time stochastic gradient descent ({\sc sgd}), where the data are typically assumed to be independent and identically distributed. The presence of correlated data substantially complicates the analysis. Furthermore, we allow the model $f(x,\theta)$ (see Assumption \ref{A:Growth-f-g}) to grow polynomially in \(x\) and quadratically in \(\theta\).

The most technical part of the present work concerns the estimation of the first- and, especially, the second-order Malliavin derivatives. Controlling the second-order derivatives requires intricate decompositions and careful bookkeeping to obtain sufficiently sharp bounds. In particular, the derivation in Lemma \ref{L:Second-der-Gamma-f-3} is delicate due to the structure of the first-order Malliavin derivatives appearing in the definition of the relevant process (see $\Gamma_3^f$ in Lemma \ref{L:Second-der-Gamma-f-3}). Another subtle point is the control of the fluctuation term \(\int_0^t \alpha_s \left[ (\bar{g}(\theta_s) - g(X_s, \theta_s)) \right]ds\), which appears repeatedly in the estimation of both derivatives and pre-limit expectations and variances. Intuitively, this term should be small because the process $X_t$ is assumed to be ergodic; however, in the present setting we derive an explicit rate at which it vanishes by constructing several Poisson equations.

We now discuss related work on {\sc sgd}, in both discrete and continuous time, and compare it with the present study. Our work is related to \cite{siri_spilio_2020}, where the authors establish an $L^p$ convergence rate and characterize qualitative fluctuations for {\sc sgdct}. Their analysis relies on Duhamel's principle, It\^o’s formula, the construction of several Poisson equations \cite{pardoux2003poisson}, and martingale moment inequalities \cite[Proposition 3.26]{KS91}. In contrast, we focus on the {\sc qclt} and employ a different set of tools—most notably, a second-order Poincar\'e inequality from Malliavin calculus—to obtain explicit convergence rates. Earlier contributions to the study of {\sc sgdct} include, for example, \cite{KS17}, which investigates convergence of {\sc sgd} updates $\theta_t$ to a critical point of the objective function and discusses applications in deep learning, American options, and partial differential equations. For discrete-time {\sc sgd} with independent and identically distributed (i.i.d.) data, foundational results can be found in \cite{benveniste2012adaptive}, while studies such as \cite{sakrison1964continuous} and \cite{yin1993continuous} examine convergence and {\sc clt} results for continuous-time {\sc sgd} in settings without data dynamics for $X$. The works \cite{KS17, siri_spilio_2020}, however, do incorporate stochastic dynamics for $X$, as we do here. For convergence analysis of {\sc sgdct} updates using filtered data, we refer to \cite{pavliotis2025filtered} and the references therein. For parameter estimation in McKean-Vlasov {\sc sde} and interacting particle systems, see, for example, \cite{sharrock2023online}.

In recent decades, Malliavin calculus has emerged as a powerful framework for establishing {\sc qclt} results across a wide range of stochastic models. Several recent works use this approach to derive explicit rates of convergence in terms of Malliavin-derivative bounds. For example, \cite{bour_spilio_2025} studies a fully coupled multiscale system and derives quantitative fluctuation bounds by controlling the first- and second-order Malliavin derivatives of both the fast and slow components. The authors of \cite{rockner2021averaging} also investigate quantitative fluctuations in multiscale diffusion processes, but follow a different route via the analysis of associated Cauchy problems. Our methodology is closest in spirit to that of \cite{bour_spilio_2025}, while our setting introduces additional challenges due to the structure of the {\sc sgdct} updates, the dependence on the dynamics of $X$, and the growth of the model $f(x,\theta)$. For further developments in quantitative error analysis for stochastic iterative algorithms, see \cite{wang2025quantitative}. In addition, for {\sc qclt} results in the context of stochastic partial differential equations ({\sc spde}), we refer to \cite{nualart2022quantitative,nualart2022central,balan2024hyperbolic} and the references therein.

The paper is organized as follows. In Section \ref{S:Problem-Set-up}, we present our problem set-up ({\sc sgdct}) and the core equations used throughout the manuscript. This section also states the main assumptions and presents our principal result, Theorem \ref{T:Main-theorem}. In Section \ref{S:theorem-proof}, we prove the main result by combining several key ingredients, including Propositions \ref{P:First-der-product} \ref{P:Second-der-product}, \ref{P:Pre-limit-Expectation-P-1}, and \ref{P:Prelimit-Sec-2-P1}. In Section \ref{S:simulation}, we numerically illustrate our theoretical findings. Sections \ref{S:First-order-derivative} and \ref{S:Second-order-derivative} are devoted to the analysis of first- and second-order Malliavin derivatives, respectively; the main conclusions are summarized in Propositions \ref{P:First-der-product} and \ref{P:Second-der-product}. Section \ref{S:Slower-rates} focuses on cases exhibiting slower convergence rates for the quantity of interest. Next, in Section \ref{S:Pre-limit-Exp-Var}, we present the calculations associated with rescaling, pre-limit expectations, and variances. Finally, in Section \ref{S:Appendix}, we provide preliminary material on Poisson equations and Malliavin calculus for the convenience of the reader, along with proofs of several auxiliary results from the previous sections.

\subsection*{Notation and conventions}
We list some of the notation and conventions used throughout this paper. The symbol $\triangleq$ is read as ``is defined to equal''. We denote the sets of positive integers and real numbers by $\BN$ and $\BR,$ respectively. $\BE$ denotes the expectation operator with respect to the probability measure $\BP.$ For a random variable $Y$ with a normal distribution with parameters $\mu$ and $\xi^2,$ we write $Y \sim \mathscr{N}(\mu, \xi^2).$ For two random variables $U$ and $V$, we define
$$d_W(U,V) \triangleq \sup_{\mathsf{f}:|\mathsf{f}|_{\text{Lip}}\le 1}\left|\BE \left[\mathsf{f}(U)\right] - \BE \left[\mathsf{f}(V)\right]\right|,$$
which represents the Wasserstein distance between the laws of $U$ and $V.$ The partial derivatives of the function \(\mathsf{j}:\mathbb{R}^n \to \mathbb{R}\) with respect to the \(i^{\text{th}}\) variable are denoted using subscript notation: \(\mathsf{j}_{x_i} \triangleq \frac{\partial \mathsf{j}}{\partial x_i}\). The same convention applies to higher-order derivatives. For \(0 \leq r \leq t\), the first- and second-order Malliavin derivatives of the process \(U_t\) are denoted by \(D_r U_t\) and \(D_r^2 U_t\), respectively. Throughout the paper, the letter $K$ denotes a positive constant that may depend on various parameters \textit{except} for the time parameter $t$; its value may change from line to line. We will repeatedly use the following versions of the triangle inequality and H\"older's inequality: for any $n \in \BN,$ $p>0$, positive real numbers $a_1, \cdots, a_n,$ and random variables $U_1, \cdots, U_n$,
\begin{equation*}
\begin{aligned}
(a_1 + \cdots + a_n)^p  \le K (a_1^p + \cdots + a_n^p), \qquad
\BE \left[ \prod_{i=1}^n U_i \right] \le \prod_{i=1}^n \left[ \BE (U_i^n) \right]^{\frac{1}{n}}.
\end{aligned}
\end{equation*}

\section{Problem Statement, Assumptions and Main result}\label{S:Problem-Set-up}
Our problem formulation follows the work of \cite{siri_spilio_2020}. Given observations of the data process $X_t$ satisfying the {\sc sde} \eqref{E:Process-X} and a model $f(x,\theta)$ for the unknown function $f^*(x)$, the continuous-time stochastic gradient descent updates for the parameter $\theta$ are governed by the {\sc sde}
\begin{equation}\label{E:SGDCT-theta-updates}
d\theta_t = \alpha_t \left[f_\theta(X_t,\theta_t)\sigma ^{-2}dX_t - f_\theta(X_t,\theta_t)\sigma^{-2}f(X_t, \theta_t)\thinspace dt \right],
\end{equation}
where \(\alpha_t\) denotes the learning rate, and throughout this
paper we assume $\alpha_t = \frac{C_\alpha}{C_0 + t}.$ To better
understand the parameter estimation problem for \(\theta\), we note
from \eqref{E:Process-X} that $dX_t$ provides a noisy estimate of
$f^*(X_t)dt,$ which motivates the derivation of
\eqref{E:SGDCT-theta-updates}. Following this, we can interpret {\sc sgdct} as the sum of a descent direction and a noise term:
\begin{equation}\label{E:SGDCT-theta-updates-V-2}
\begin{aligned}
d\theta_t &= \alpha_t \left[ f_\theta(X_t, \theta_t)\sigma^{-2} dX_t - f_\theta(X_t, \theta_t)\sigma^{-2} f(X_t, \theta_t) \thinspace dt \right] \\
&= \alpha_t f_\theta(X_t, \theta_t)\sigma^{-2} (f^*(X_t) - f(X_t, \theta_t)) \thinspace dt + \alpha_t f_\theta(X_t, \theta_t)\sigma^{-1} \thinspace dW_t \\
&= -\alpha_t g_{\theta}(X_t, \theta_t) \thinspace dt + \alpha_t f_\theta(X_t, \theta_t)\sigma^{-1} dW_t.
\end{aligned}
\end{equation}
Next, to make the progress of $\theta$ toward its minimum more
transparent -- since the descent term depends on the dynamics of
$X_t$ -- we rewrite the above equation as the sum of a descent term, a fluctuation term, and a noise term:
\begin{equation}\label{E:Process-theta}
d\theta_t = \underbrace{-\alpha_t \bar{g}_\theta(\theta_t) \thinspace dt}_{\text{Descent term}} + \underbrace{\alpha_t \Big(
  \bar{g}_{\theta}(\theta_t) - g_\theta(X_t,\theta_t)  \Big) \thinspace dt}_{\text{Fluctuation term}} + \underbrace{\alpha_t f_\theta(X_t,\theta_t) \sigma^{-1} \thinspace dW_t}_{\text{Noise term}}.
\end{equation}
We expect the process \(\theta\) to converge toward its minimum for the following heuristic reason. With our choice of learning rate \(\alpha_t\), which decreases over time, and in view of the definition of the distance function \(g(x, \theta)\), the descent term $\alpha_t \bar{g}_\theta(\theta_t)$ dominates both the fluctuation term $\alpha_t (
  \bar{g}_{\theta}(\theta_t) - g_\theta(X_t,\theta_t))$ and the noise term $\alpha_t f_\theta(X_t,\theta_t) \sigma^{-1}$ for large $t$. A detailed and rigorous analysis is given in \cite{KS17}, where the authors show that \(\theta\) converges to a critical point of \(\bar{g}(\theta)\) by proving that \(|\nabla \bar{g}(\theta_t)| \to 0\) almost surely as \(t\) approaches infinity.

Before stating our main result (Theorem \ref{T:Main-theorem}), we introduce the key assumptions on which Theorem \ref{T:Main-theorem} relies.

\begin{assumption}[Ergodicity]\label{A:f*-growth}
There exists a positive number $C^*$ such that
$$f_x^* \le -C^*< 0.$$
\end{assumption}
As a consequence of this assumption, we have $\lim_{|x|\to \infty}f^*(x)\cdot x = - \infty.$ This condition guarantees the existence of a unique invariant measure for the process $X_t.$ Next, Assumption \ref{A:Well-Posedness} (part 2) ensures that \eqref{E:Process-X} is well posed.

\begin{assumption}[Well posedness]\cite{siri_spilio_2020}\label{A:Well-Posedness}
\begin{enumerate}
\item For all $\theta\in \BR$, $x \in \mathcal{S} \subseteq \BR $ and for some $\gamma \in (0,1),$ we assume that $g_\theta(x, \cdot) \in \mathscr{C}(\BR), \thinspace g_{xx \theta} \in \mathscr{C}(\mathcal{S},\BR),$ and $g_\theta(\cdot, \theta) \in \mathscr{C}^{\gamma}(\mathcal{S})$ uniformly in $\theta.$
\item For $\gamma \in (0,1),$ the function $f^*(x)$ has two derivatives in $x$ with all partial derivatives being H\"older continuous with exponent $\gamma,$ i.e., $f^*(x) \in \mathscr{C}^{2+\gamma}(\mathcal{S}).$
\item For all $t \ge 0,$ Equation \eqref{E:SGDCT-theta-updates} has a unique strong solution.
\end{enumerate}
\end{assumption}

In our analysis, we allow the function \( f(x, \theta) \) and its derivatives with respect to \( \theta \) to grow polynomially in both variables \( x \) and \( \theta \) (Assumption \ref{A:Growth-f-g}). Consequently, we require uniform moment bounds for the processes \( \theta_t \) and \( X_t \). The bounds for $\theta$ and $X$ are ensured by the following assumption and by the conditions in \cite{pardoux2001poisson}, respectively.

\begin{assumption}[Uniform-in-time moment bounds for process $\theta$]\cite{siri_spilio_2020}\label{A:Moment-bounds}

\begin{enumerate}
\item For some finite number $R$ such that $|\theta| \ge R$ and the function $\kappa(x)> \gamma >0$, we have:
$- \theta \cdot g_\theta(x,\theta) \le -\kappa(x)|\theta|^2, \thinspace \text{for all} \thinspace \thinspace x \in \BR.$
\item Let $\tau(x,\theta) \triangleq |f_\theta(x,\theta)|.$ Then, there exists a function $\lambda(x)$ growing not faster than polynomially in $|x|,$ such that for all $\theta_1, \theta_2 \in \BR$ and $x \in \BR$, we have
$|\tau(x, \theta_1)- \tau(x, \theta_2)| \le |\lambda(x)| \rho(|\theta_1 - \theta_2|),$
where for $s \ge 0,$ $\rho(s)$ is an increasing function with $\rho(0)=0$ and $\int_0^\infty \rho^{-2}(s) \thinspace ds = \infty.$
\end{enumerate}
 \end{assumption}

\begin{assumption}[Growth of the derivatives of the functions $f(x,\theta)$ and $\bar{g}(\theta)$]\label{A:Growth-f-g}
The functions $f(x,\theta)$ and $\bar{g}(\theta)$ are assumed to satisfy the following conditions for some positive constants $K$ and $q$:
\begin{enumerate}
\item $\left|\frac{ \partial^i f (x, \theta)}{\partial \theta^i}\right| \le K(1+|x|^q +|\theta|^j),$ for $i=0,1,2,3$ and $j \triangleq \text{max}\{2-i,0\}.$
\item $\bar{g}(\theta)$ is strongly convex, i.e., there exists a positive constant $C_{\bar{g}}$ such that
$\bar{g}_{\theta \theta}(\theta) \ge C_{\bar{g}},$ as $\bar{g}$ is differentiable.
\item $\left|\frac{ \partial^i \bar{g}(\theta)}{\partial \theta^i} \right| \le K (1+ |\theta|^{4-i}),$ for $i=0,1,2,3.$
\end{enumerate}
\end{assumption}

\begin{assumption}[Learning rate] \label{A:Learning-rate}
The learning rate $\alpha_t$ satisfies the condition:
\begin{equation*}
\int_0^\infty \alpha_s \thinspace ds = \infty, \qquad \int_0^\infty \alpha_s^2 \thinspace ds < \infty.
\end{equation*}
\end{assumption}
In this paper, we focus on a specific form of the learning rate\footnote{We conveniently express $\alpha_t$ as $\frac{C_\alpha}{t}$ for $t \ge r > 0.$} defined as $\alpha_t \triangleq \frac{C_\alpha}{C_0 + t}, \thinspace C_\alpha > 0$. This choice is made for two main reasons: first, it simplifies the presentation, and second, it reflects the typical structure of learning rates used in practice, see, for example, \cite{sharrock2023online,siri_spilio_2020}.

For completeness, we recall the qualitative {\sc clt} for the process \( \theta_t \), which is one of the main results in \cite{siri_spilio_2020}. In contrast, the primary focus of the present work is to provide a quantitative counterpart.

\begin{proposition}[Qualitative {\sc clt} \cite{siri_spilio_2020}]\label{P:SS20-Main-Result}
Let $\theta_t$ be the solution of Equation \eqref{E:Process-theta} and Assumptions \ref{A:f*-growth} through \ref{A:Learning-rate} are satisfied. Then, as $t \to \infty$, we have
\begin{equation*}
\sqrt{t}(\theta_t - \theta^*) \Longrightarrow \mathscr{N}(0, \bar{\Sigma}),
\end{equation*}
where, for the solution $\Psi$ of Poisson Equation \eqref{E:Poisson-equation-prelimit} and the functions $\bar{h}(\theta) = \int h(x,\theta)\mu(dx)$, $h(x,\theta)  \triangleq \sigma^2 \left[ f_\theta(x, \theta) \sigma^{-2} - \Psi_x(x, \theta) \right]^2$,  the limiting variance\footnote{We would like to point out a minor typo in \cite[Theorem 2]{siri_spilio_2020} regarding the expression for the limiting variance \(\bar{\Sigma}\). Specifically, the expression is missing the number \(\frac{1}{2}\) in the exponent; the correct expression should read \(\left(C_\alpha \bar{g}_{\theta\theta}(\theta^*) - \frac{1}{2}\right)\) instead of \(\left(C_\alpha \bar{g}_{\theta\theta}(\theta^*) - 1\right)\). This typo appears \emph{only} in the integral version of the variance expression. However, the summation counterpart, as given in \cite[Equation 30]{siri_spilio_2020}, is correct. Therefore, the calculations remain unaffected.} is defined to be
\begin{equation}\label{E:Limiting-Variance}
\bar{\Sigma} \triangleq C_\alpha^2 \int_0^\infty e^{-2s \left(C_\alpha \bar{g}_{\theta\theta}(\theta^*)- \frac{1}{2} \right)}\bar{h}(\theta^*) \thinspace ds.
\end{equation}
\end{proposition}

Before stating our main result, Theorem \ref{T:Main-theorem}, we specify a technical condition involving the learning-rate magnitude $C_\alpha.$ In this condition, $K_{g_{\theta \theta}}^*$ is the upper bound of the quantity $\sup_{t \ge t^*}K_{g_{\theta \theta}}(t) \triangleq \sup_{t \ge t^*} \left[ \BE|g_{\theta \theta}(X_t, \theta_t)|^p\right]^{\frac{1}{p}}$, $p \in \BN,$ where $t^*$ is a sufficiently large number and its choice is discussed in Lemma \ref{L:moment-bound} and Remark \ref{R:Rem-t*-choice}. For all $t \ge t^*$, the function $K_{g_{\theta \theta}}(t)$ is uniformly bounded by a constant independent of $C_\alpha$, and hence $K_{g_{\theta \theta}}^* \triangleq \sup_{t \ge t^*}K_{g_{\theta \theta}}(t)$ is independent of $C_\alpha.$

\begin{assumption}\label{A:Tech-Cond}
For a fixed convexity constant $C_{\bar{g}}$, the learning rate magnitude $C_\alpha$ satisfies the conditions:
\begin{equation*}
C_{\bar{g}} C_\alpha > \frac{\sigma^2}{2}, \quad \text{and} \quad K_{g_{\theta \theta}}^* < \frac{\sigma^2}{2C_\alpha} + 2 C_{\bar{g}}.
\end{equation*}
\end{assumption}
\color{black}
We now state the main result of the paper.

\begin{theorem}[Quantitative {\sc clt}]\label{T:Main-theorem}
Let the process $\theta_t$ be the solution of Equation \eqref{E:Process-theta}, the rescaled process $\mathsf{F}_t \triangleq \sqrt{t}(\theta_t-\theta^*)$ and $N \sim
\mathscr{N}(0, \bar{\Sigma})$. Assume that Assumptions \ref{A:f*-growth} through \ref{A:Tech-Cond} hold. Then, for a sufficiently large $t$, i.e., $t \ge t^*$, there exists a time-independent positive constant $K$ such that
$$d_W(\mathsf{F}_t, N) \le \begin{cases} \frac{K \log t}{{t}^{\frac{1}{4}}},  & {C_{\bar{g}} C_\alpha \ge \frac{3}{4}\sigma^2}\\  \frac{K}{t^{C_{\bar{g}} C_\alpha \sigma^{-2}- \frac{1}{2}}}, & \frac{\sigma^2}{2} < C_{\bar{g}}C_\alpha < \frac{3}{4} \sigma^2,\end{cases}$$
where the choice of $t^*$ is specified in Lemma \ref{L:moment-bound} and Remark \ref{R:Rem-t*-choice}.
\end{theorem}

\begin{remark}
In this manuscript, the positive constant $ K $ is always time-independent; however, it may depend on various other parameters such as the noise size $ \sigma$, the magnitude of the learning rate $C_\alpha$, and the growth-rate parameters of the functions $ f(x, \theta)$ and $\bar{g}(\theta)$. For simplicity, we set $ \sigma = 1$, as it appears in our calculations merely as a scalar.
\end{remark}

\begin{remark}[Comments on Assumption \ref{A:Tech-Cond}]\label{R:Tech-Cond}
The inequality \( K_{g_{\theta \theta}}^* < \frac{\sigma^2}{2C_\alpha}
+ 2 C_{\bar{g}} \) stated in Assumption \ref{A:Tech-Cond} is a purely
technical condition that arises from Lemma
\ref{L:Second-der-Gamma-f-3}. As a consequence of this assumption, we
observe that \( K_{g_{\theta \theta}}^* < 3 C_{\bar{g}} \), which is
consistent with the stability of the algorithm. The condition is also
expected to hold for $t>t^*$ when $t^*$ is sufficiently large if
convergence is to occur; indeed, around a local minimum the loss
function should behave approximately like a quadratic function. A careful inspection of the proof of Lemma \ref{L:Second-der-Gamma-f-3} shows that the condition $K_{g_{\theta \theta}}^* < \frac{1}{2C_\alpha} + 2 C_{\bar{g}}$ can be weakened to $K_{g_{\theta \theta}}^* < \frac{1}{2C_\alpha} + 3 C_{\bar{g}}$, which would, however, yield a slower convergence rate in Theorem \ref{T:Main-theorem}. We also note that in Lemma \ref{L:Second-der-Gamma-f-3}, if we assume $f_{\theta \theta \theta} = 0$, then the condition $K_{g_{\theta \theta}}^* < \frac{\sigma^2}{2C_\alpha} + 2 C_{\bar{g}}$ is not required for our calculations.
\end{remark}

\begin{remark}[Comments on the restriction of the domain  of integration in Section \ref{S:theorem-proof}]\label{R:Large-time-choice}
It is important to note that in the proof of Theorem \ref{T:Main-theorem}, we restrict the domain of integration from the interval \([0, t]\) to \([t^*, t]\), where $t^{*}$ is large enough. This restriction arises from the need to compute the second-order derivatives specifically within the interval \([t^*, t]\), as discussed in Proposition \ref{P:Second-der-product}, see also Lemma
\ref{L:Second-der-Gamma-f-3} and Remark \ref{R:Tech-Cond}. As discussed in Proposition \ref{P:Second-der-product}, $t^*$ is a sufficiently large number and its choice is specified in Lemma \ref{L:moment-bound} and Remark \ref{R:Rem-t*-choice}. Specifically, it reveals a natural and interesting interplay between $\gamma$, the constant from Assumption \ref{A:Moment-bounds}, and $C_{\alpha}$, the learning rate magnitude. 
\end{remark}

\begin{remark}[Comments on the multidimensional case]
For readability, and to prevent the paper from being unnecessarily
long, we present our results, proofs, and calculations in dimension
one, since extending them to the multidimensional case requires no
additional ideas -- only more involved computations, which amount to summing over the additional dimensions. Indeed, the key ingredient is Proposition \ref{P:vidotto}, which corresponds to \cite[Theorem 2.1]{vido_2020}. The multidimensional case can be obtained with the same methodology by using \cite[Theorem 2.2 and Remark 2.1]{vido_2020} instead of \cite[Theorem 2.1]{vido_2020}, and by summing over the dimensions, as suggested by the structure of the bound in \cite[Theorem 2.2]{vido_2020}.
\end{remark}

\begin{remark}[Uniform-in-time moments]
We note that, in many places in the proof of Theorem \ref{T:Main-theorem}, we construct several Poisson equations whose solutions grow polynomially in the variables $\theta$ and $x$. For relevant results concerning Poisson equations, we refer readers to Section \ref{S:Poisson-Equation}. It is therefore necessary to control uniform-in-time moments of the processes $\theta_t$ and $X_t$, specifically, $\sup_{t \ge 0} \BE\left[ |\theta_t|^{\mathsf{n}} \right]$ and $\sup_{t \ge 0} \BE\left[ |X_t|^{\mathsf{m}} \right]$ for some positive integers ${\mathsf{n}}$ and ${\mathsf{m}}$. Under the stated assumptions, \cite{pardoux2001poisson} shows that $\sup_{t \ge 0} \BE\left[ |X_t|^{\mathsf{m}} \right] < \infty$, and, \cite{siri_spilio_2020}  and subsection \ref{ss:MomentBoundTheta} prove that $\sup_{t \ge 0} \BE\left[ |\theta_t|^{\mathsf{n}} \right]< \infty$.
\end{remark}

\section{Proof of our Main Result: Theorem \ref{T:Main-theorem}}\label{S:theorem-proof}
In this section, we prove Theorem \ref{T:Main-theorem}. The main ingredient of the proof is the improved second-order Poincar\'e inequality of \cite[Theorem 2.1]{vido_2020}, which we recall below. After a suitable normalization of the pre-limit process to satisfy the assumptions of this theorem, the resulting Wasserstein distance bound can be expressed in terms of the first- and second-order Malliavin derivatives of the pre-limit process together with its expectation and variance. Consequently, the proof of the {\sc qclt} reduces to estimating these four quantities. Propositions \ref{P:First-der-product} and \ref{P:Second-der-product} provide the bounds required to control the first- and second-order Malliavin derivatives, respectively, while Propositions \ref{P:Pre-limit-Expectation-P-1} and \ref{P:Prelimit-Sec-2-P1} characterize the asymptotic behavior of the expectation and variance of the pre-limit process. Combining these estimates through the Poincar\'e inequality yields the convergence rate stated in Theorem \ref{T:Main-theorem}. 

\begin{proposition}\cite[Theorem 2.1]{vido_2020}\label{P:vidotto}
Let ${F} \in \mathbb{D}^{2,4}$ be such that $\mathbb{E}(F) = 0$ and
$\mathbb{E}(F^2) = \tilde{\sigma}^2$, and let $N \sim
\mathscr{N}(0,\tilde{\sigma}^2)$. Then,
\begin{align*}
d_W(F,N) &\leq \sqrt{\frac{8}{ \tilde{\sigma}^2 \pi}}\left[
  \int_{\BR_+^2}^{}\sqrt{\mathbb{E} \left[ \left( D^2F \otimes_1 D^2F
           \right)(x,y)^2 \right]\mathbb{E} \left[ DF(x)^2 \cdot DF(y)^2 \right]}\thinspace dx \thinspace dy\right]^{\frac{1}{2}},
\end{align*}
where the Malliavin operators $D$ and $D^2$ are those defined in
Subsection \ref{PrelimsOnMalliavin}, and $\otimes_1$ is the contraction operator defined as follows: for any two functions $\mathsf{f}$ and $\mathsf{g}$ in $L^2([0,t]),$
$$[\mathsf{f} \otimes_1 \mathsf{g}] \triangleq \int_0^t \mathsf{f}(x,z)\mathsf{g}(y,z)\thinspace dz.$$
\end{proposition}

To apply Proposition \ref{P:vidotto}, we define
$\tilde{\mathsf{F}}_t \triangleq \sqrt{\frac{\bar{\Sigma}}{\operatorname{Var}(\mathsf{F}_t)}}(\mathsf{F}_t - \BE(\mathsf{F}_t))$,
where $\mathsf{F}_t = \sqrt{t}(\theta_t-\theta^*)$ and $\bar{\Sigma}$ is the limiting variance defined in \eqref{E:Limiting-Variance}.
\begin{remark}
As we prepare to apply Proposition \ref{P:vidotto} to prove our result, it is essential for the process \(\tilde{\mathsf{F}}_t\) to belong to \(\mathbb{D}^{2,4}\). This requirement, in turn, necessitates that \(\theta_t\) also belongs to \(\mathbb{D}^{2,4}\). According to the definition of the Sobolev norm in Equation \eqref{E:Sobolev-Norm}, it suffices to demonstrate that \(\theta_t \in L^4(\Omega)\), \(D \theta_t \in L^4(\Omega; \mathfrak{H})\), and \(D^2 \theta_t \in L^4(\Omega; \mathfrak{H}^{\otimes 2})\). These conditions follow from the uniform moment bounds for \(\theta_t\) established in \cite{siri_spilio_2020}, as well as from Lemmas \ref{L:1-der-moments} and \ref{L:2-der-moments} for the first- and second-order Malliavin derivatives with $p=2$, respectively. Therefore, we conclude that \(\theta_t \in \mathbb{D}^{2,4}\).
\end{remark}

\begin{proof}[Proof of Theorem \ref{T:Main-theorem}]
We start noting $\mathbb{E}(\mathsf{F}_t) \neq 0$, $\mathbb{E}(\mathsf{F}_t^2) \neq \bar{\Sigma}$, and $N \sim \mathscr{N}(0,\bar\Sigma)$, and then applying triangle inequality to get
\begin{equation}\label{E:Rearranged-Equation}
\begin{aligned}
d_W \left( \mathsf{F}_t,N \right) &\leq d_W \left( \mathsf{F}_t,\tilde{\mathsf{F}}_t \right)  +
d_W \left( \tilde{\mathsf{F}}_t,N \right)\\
& \leq \mathbb{E} \left( \left|{\mathsf{F}_t - \tilde{\mathsf{F}}_t}\right| \right) + d_W \left(
  \tilde{\mathsf{F}}_t,N \right)\\
  &\leq \mathbb{E} \left( |{\mathsf{F}_t}|
    \right)\left|{1-\sqrt{\frac{\bar{\Sigma}}{\operatorname{Var}(\mathsf{F}_t)}}}\right|
    + \sqrt{\frac{\bar{\Sigma}}{\operatorname{Var}(\mathsf{F}_t)}}
    |{\mathbb{E}(\mathsf{F}_t)}| + d_W \left(
  \tilde{\mathsf{F}}_t,N \right),
\end{aligned}
\end{equation}
where the second inequality follows from the fact that the Wasserstein distance is bounded above by the $L^1(\Omega)$-norm, since its test functions are $1$-Lipschitz. Since $\mathbb{E}(\tilde{\mathsf{F}}_t) = 0$ and $\mathbb{E}(\tilde{\mathsf{F}}_t^2) =
\bar{\Sigma}$, Proposition \ref{P:vidotto} applies to the last term $d_W \left(
  \tilde{\mathsf{F}}_t,N \right)$ in \eqref{E:Rearranged-Equation}. Moreover, since
$D\tilde{\mathsf{F}}_t =
\sqrt{\frac{\bar{\Sigma}}{\operatorname{Var}(\theta_t)}}D\theta_t$, we obtain
\begin{equation}\label{E:Tilde-F-N}
\begin{aligned}
d_W \left(
  \tilde{\mathsf{F}}_t,N \right) \leq \frac{\sqrt{8}\bar{\Sigma}}{\sqrt{\pi}\operatorname{Var}(\theta_t)}\sqrt{\int_{\BR_+^2}^{}\sqrt{\mathbb{E} \left[ \left( D^2\theta_t \otimes_1 D^2\theta_t
           \right)(r,s)^2 \right]\mathbb{E} \left[ D\theta_t(r)^2 \cdot D\theta_t(s)^2 \right]} \thinspace dr \thinspace ds}.
\end{aligned}
\end{equation}
We now repeatedly apply H\"older's inequality to bound the right-hand side of \eqref{E:Tilde-F-N}, while restricting the domain of integration from $[0,t]$ to $[t^*,t]$ (see, Remark \ref{R:Large-time-choice}), obtaining\footnote{For convenience, we change the notation for the Malliavin derivative from $D\theta_t(r)$ to $D_r \theta_t.$}
\begin{equation}\label{E:Main-Eq-Holder}
\begin{aligned}
&\sqrt{\int_{r=t^*}^t \int_{s=t^*}^t  \sqrt{ \BE \left\{(D^2 \theta_t \otimes_1 D^2 \theta_t) (r, s)^2 \right\} \cdot \BE \left\{D_r \theta_t^2 \cdot D_s \theta_t^2 \right\}} \thinspace dr \thinspace ds} \\
& \qquad  \le \left[\int_{[t^*,t]^2} \sqrt{\BE [D_r \theta_t^4]} \sqrt{\BE [D_s \theta_t^4]} \thinspace dr \thinspace ds \right]^{\frac{1}{4}}  \\
& \qquad \qquad      \times  \left[ \int_{[t^*,t]^4} \left( \BE |D^2_{u,r} \theta_t|^4\right)^{\frac{1}{4}} \left( \BE |D^2_{u,s} \theta_t|^4\right)^{\frac{1}{4}} \left( \BE |D^2_{w,r} \theta_t|^4\right)^{\frac{1}{4}} \left( \BE |D^2_{w,s} \theta_t|^4\right)^{\frac{1}{4}} \thinspace du \thinspace  ds \thinspace dw \thinspace dr\right]^{\frac{1}{4}}.
\end{aligned}
\end{equation}
Combining \eqref{E:Tilde-F-N} and \eqref{E:Main-Eq-Holder} yields\footnote{We note that the limiting variance $\bar{\Sigma}$ and other constants are included in the generic constant $K$ in this section and elsewhere in the manuscript.}
 \begin{equation}\label{E:Main-Eq-Holder-2}
 \begin{aligned}
 d_W \left(
  \tilde{\mathsf{F}}_t,N \right) & \leq \frac{K}{\operatorname{Var}(\theta_t)}\left[\int_{[t^*,t]^2} \sqrt{\BE [D_r \theta_t^4]} \sqrt{\BE [D_s \theta_t^4]} \thinspace dr \thinspace ds \right]^{\frac{1}{4}}  \left[ \int_{[t^*,t]^4} \left( \BE |D^2_{u,r} \theta_t|^4\right)^{\frac{1}{4}} \times \right.\\
  & \qquad \qquad \qquad \qquad \quad \left. \left( \BE |D^2_{u,s} \theta_t|^4\right)^{\frac{1}{4}} \left( \BE |D^2_{w,r} \theta_t|^4\right)^{\frac{1}{4}} \left( \BE |D^2_{w,s} \theta_t|^4\right)^{\frac{1}{4}} \thinspace du \thinspace  ds \thinspace dw \thinspace dr\right]^{\frac{1}{4}}.
 \end{aligned}
 \end{equation}
From \eqref{E:Rearranged-Equation} and \eqref{E:Main-Eq-Holder-2}, we obtain
 \begin{equation}\label{E:Main-result-Proof-Eq-1}
 \begin{aligned}
d_W \left( \mathsf{F}_t,N \right) &\leq \mathbb{E} \left( |{\mathsf{F}_t}|
    \right)\left|{1-\sqrt{\frac{\bar{\Sigma}}{\operatorname{Var}(\mathsf{F}_t)}}}\right|
    + \sqrt{\frac{\bar{\Sigma}}{\operatorname{Var}(\mathsf{F}_t)}}
    |{\mathbb{E}(\mathsf{F}_t)}| \\
   & \qquad \quad + \frac{K}{\operatorname{Var}(\theta_t)}\left[\int_{[t^*,t]^2} \sqrt{\BE [D_r \theta_t^4]} \sqrt{\BE [D_s \theta_t^4]} \thinspace dr \thinspace ds \right]^{\frac{1}{4}}  \left[ \int_{[t^*,t]^4} \left( \BE |D^2_{u,r} \theta_t|^4\right)^{\frac{1}{4}} \times \right.\\
  & \qquad \qquad \qquad \qquad \qquad \quad \left. \left( \BE |D^2_{u,s} \theta_t|^4\right)^{\frac{1}{4}} \left( \BE |D^2_{w,r} \theta_t|^4\right)^{\frac{1}{4}} \left( \BE |D^2_{w,s} \theta_t|^4\right)^{\frac{1}{4}} \thinspace du \thinspace  ds \thinspace dw \thinspace dr\right]^{\frac{1}{4}}.
 \end{aligned}
 \end{equation}
We now apply Propositions \ref{P:Pre-limit-Expectation-P-1} and \ref{P:Prelimit-Sec-2-P1} to bound
 $\mathbb{E} \left( |{\mathsf{F}_t}|
    \right)\left|{1-\sqrt{\frac{\bar{\Sigma}}{\operatorname{Var}(\mathsf{F}_t)}}}\right|$ and $\sqrt{\frac{\bar{\Sigma}}{\operatorname{Var}(\mathsf{F}_t)}}
    |{\mathbb{E}(\mathsf{F}_t)}|$, respectively, obtaining
    \begin{equation}\label{E:Main-result-Proof-Eq-2}
    \begin{aligned}
    \mathbb{E} \left( |{\mathsf{F}_t}|
    \right)\left|{1-\sqrt{\frac{\bar{\Sigma}}{\operatorname{Var}(\mathsf{F}_t)}}}\right| + \sqrt{\frac{\bar{\Sigma}}{\operatorname{Var}(\mathsf{F}_t)}}|{\mathbb{E}(\mathsf{F}_t)}| & \le \frac{\mathbb{E} \left( |{\theta_t-\theta^*}| \right)}{\sqrt{{\operatorname{Var}(\theta_t-\theta^*)}}} \times \begin{cases} \frac{K+K \log t}{{t}^{\frac{1}{4}}},  & {C_{\bar{g}} C_\alpha \ge \frac{3}{4}}\\  \frac{K}{t^{C_{\bar{g}} C_\alpha- \frac{1}{2}}}, & \frac{1}{2} < C_{\bar{g}}C_\alpha < \frac{3}{4}\end{cases} \\
    & \qquad  + \frac{1}{\sqrt{\operatorname{Var}(\theta_t)}}\times \begin{cases} \frac{K}{t},  & {C_{\bar{g}}C_\alpha > 1}\\   \frac{K + K \log t }{t}, & {C_{\bar{g}}C_\alpha = 1} \\ \frac{K}{t^{C_{\bar{g}} C_\alpha}}, & \frac{1}{2} < C_{\bar{g}} C_\alpha < 1,\end{cases}
    \end{aligned}
    \end{equation}
whereas the term
    \begin{equation*}
    \begin{aligned}
   &  \frac{K}{\operatorname{Var}(\theta_t)}\left[\int_{[t^*,t]^2} \sqrt{\BE [D_r \theta_t^4]} \sqrt{\BE [D_s \theta_t^4]} \thinspace dr \thinspace ds \right]^{\frac{1}{4}}  \left[ \int_{[t^*,t]^4} \left( \BE |D^2_{u,r} \theta_t|^4\right)^{\frac{1}{4}} \times \right.\\
  & \qquad \qquad \qquad \qquad \qquad \quad \qquad \qquad \qquad \left. \left( \BE |D^2_{u,s} \theta_t|^4\right)^{\frac{1}{4}} \left( \BE |D^2_{w,r} \theta_t|^4\right)^{\frac{1}{4}} \left( \BE |D^2_{w,s} \theta_t|^4\right)^{\frac{1}{4}} \thinspace du \thinspace  ds \thinspace dw \thinspace dr\right]^{\frac{1}{4}}
    \end{aligned}
    \end{equation*}
in \eqref{E:Main-result-Proof-Eq-1} can be bounded by combining Propositions \ref{P:First-der-product} and \ref{P:Second-der-product}, which yields
    \begin{equation}\label{E:Main-result-Proof-Eq-3}
    \begin{aligned}
    &  \frac{K}{\operatorname{Var}(\theta_t)}\left[\int_{[t^*,t]^2} \sqrt{\BE [D_r \theta_t^4]} \sqrt{\BE [D_s \theta_t^4]} \thinspace dr \thinspace ds \right]^{\frac{1}{4}}  \left[ \int_{[t^*,t]^4} \left( \BE |D^2_{u,r} \theta_t|^4\right)^{\frac{1}{4}} \times \right.\\
  & \qquad \qquad \qquad \qquad \qquad \quad \qquad   \left. \left( \BE |D^2_{u,s} \theta_t|^4\right)^{\frac{1}{4}} \left( \BE |D^2_{w,r} \theta_t|^4\right)^{\frac{1}{4}} \left( \BE |D^2_{w,s} \theta_t|^4\right)^{\frac{1}{4}} \thinspace du \thinspace  ds \thinspace dw \thinspace dr\right]^{\frac{1}{4}}\\
  & \qquad \qquad \qquad \qquad \qquad \qquad \qquad \qquad \qquad \qquad   \le \frac{1}{\operatorname{Var}(\theta_t)}\times \begin{cases} \frac{K}{{t}^{\frac{5}{4}}},  & {C_{\bar{g}} C_\alpha > \frac{5}{4}}\\
\frac{K + K(\log t)^{\frac{1}{4}}}{{t}^{\frac{5}{4}}},  & {\frac{3}{4} \le C_{\bar{g}} C_\alpha \le \frac{5}{4}} \\
  \frac{K}{t^{C_{\bar{g}} C_\alpha+ \frac{1}{2}}}, & \frac{1}{2} < C_{\bar{g}}C_\alpha < \frac{3}{4}.\end{cases}
    \end{aligned}
    \end{equation}

Putting \eqref{E:Main-result-Proof-Eq-1}, \eqref{E:Main-result-Proof-Eq-2}, and \eqref{E:Main-result-Proof-Eq-3} together, we obtain
\begin{equation*}
\begin{aligned}
d_W \left( \mathsf{F}_t,N \right) &\leq  \frac{\mathbb{E} \left( |{\theta_t-\theta^*}| \right)}{\sqrt{{\operatorname{Var}(\theta_t-\theta^*)}}} \times \begin{cases} \frac{K+K \log t}{{t}^{\frac{1}{4}}},  & {C_{\bar{g}} C_\alpha \ge \frac{3}{4}}\\  \frac{K}{t^{C_{\bar{g}} C_\alpha- \frac{1}{2}}}, & \frac{1}{2} < C_{\bar{g}}C_\alpha < \frac{3}{4}\end{cases} \\
    & \qquad \qquad \qquad \qquad \qquad \quad  + \frac{1}{\sqrt{\operatorname{Var}(\theta_t)}}\times \begin{cases} \frac{K}{t},  & {C_{\bar{g}}C_\alpha > 1}\\   \frac{K + K \log t }{t}, & {C_{\bar{g}}C_\alpha = 1} \\ \frac{K}{t^{C_{\bar{g}} C_\alpha}}, & \frac{1}{2} < C_{\bar{g}} C_\alpha < 1\end{cases} \\
    & \qquad \qquad \qquad \qquad \qquad \quad  + \frac{1}{\operatorname{Var}(\theta_t)}\times \begin{cases} \frac{K}{{t}^{\frac{5}{4}}},  & {C_{\bar{g}} C_\alpha > \frac{5}{4}}\\
\frac{K + K(\log t)^{\frac{1}{4}}}{{t}^{\frac{5}{4}}},  & {\frac{3}{4} \le C_{\bar{g}} C_\alpha \le \frac{5}{4}} \\
  \frac{K}{t^{C_{\bar{g}} C_\alpha+ \frac{1}{2}}}, & \frac{1}{2} < C_{\bar{g}}C_\alpha < \frac{3}{4}.\end{cases}
\end{aligned}
\end{equation*}
Finally, using the bound $\BE \left[ |\theta_t - \theta^*| \right] \le \frac{K}{\sqrt{t}}$ \cite[Theorem 1]{siri_spilio_2020} and noting that $\operatorname{Var}(\theta_t)  = \frac{\bar{\Sigma}}{t}$ for sufficiently large $t$ \cite[Theorem 2]{siri_spilio_2020}, we obtain
$$d_W \left( \mathsf{F}_t,N \right) \leq \begin{cases} \frac{K \log t}{{t}^{\frac{1}{4}}},  & {C_{\bar{g}} C_\alpha \ge \frac{3}{4}}\\  \frac{K}{t^{C_{\bar{g}} C_\alpha- \frac{1}{2}}}, & \frac{1}{2} < C_{\bar{g}}C_\alpha < \frac{3}{4}\end{cases}$$
which completes the proof.
\end{proof}

\section{Numerical Examples and Simulation}\label{S:simulation}
This section provides numerical illustrations of our theoretical findings, in particular Theorem \ref{T:Main-theorem}, through a few simple examples. In Example \ref{Ex:XInd}, we consider the case of \(X\)-independent dynamics, for which an explicit expression for the limiting variance \(\bar{\Sigma}\) can be derived. In Example \ref{Ex:OU}, we study the Ornstein--Uhlenbeck ({\sc ou}) process, and in Example \ref{Ex:Cubic}, we investigate a nonlinear setting with a cubic drift in \(X\). For Examples \ref{Ex:OU} and \ref{Ex:Cubic}, we provide numerical estimates of the limiting variances.

Before presenting the examples, we recall that in dimension one the Wasserstein distance between two probability measures $\mu$ and $\nu$ on $\BR$, with distribution functions $\mathscr{D}_\mu$ and $\mathscr{D}_\nu$, can be written as
\begin{align*}
W_1(\mu,\nu) \triangleq  \int_0^1 |\mathscr{D}_\mu^{-1}(u)- \mathscr{D}_\nu^{-1}(u)| \thinspace du.
\end{align*}
In particular, for empirical measures $\mu = \frac{1}{N}\sum_{i=1}^N \delta_{x_i}$ and $\nu = \frac{1}{N}\sum_{i=1}^N \delta_{y_i}$ with ordered statistics $x_{(1)} \le \cdots \le x_{(N)}$ and $y_{(1)} \le \cdots \le y_{(N)}$, this simplifies to
\begin{align*}
W_1(\mu,\nu) = {\frac{1}{N}\sum_{i=1}^N \left|x_{(i)} - y_{(i)} \right|},
\end{align*}
which is the formula used in the numerical examples below.

\begin{example}\label{Ex:XInd}
We consider the dynamics of the process $\theta_t$ governed by the {\sc sde}
$$d \theta_t = -\alpha_t g_{\theta}(\theta_t) \thinspace dt + \alpha_t \thinspace dW_t,$$
where \(W_t\) is a one-dimensional Brownian motion and $f_{\theta}(x,\theta)=1, \thinspace f^*(x)=f(x, \theta^*)=\theta^*$ is the unknown function, through the unknown parameter $\theta^*$. In this case,
$g(x,\theta) = \frac{1}{2}\left[ f(x,\theta)- f^*(x) \right]^2 = \frac{1}{2}(\theta - \theta^*)^2$,
and the convexity constant satisfies $C_{\bar{g}} =1.$ We use {\sc sgdct} to estimate the unknown parameter \(\theta^*\). In the estimated model \(f(x, \theta)\), the parameter \(\theta\) evolves according to the {\sc sde} above, where \(\alpha_t = \frac{C_\alpha}{1 + t}\) is the learning rate, with \(C_\alpha\) denoting the learning-rate magnitude and \(f^*(x) = f(x, \theta^*)\).

We fix $\theta^* =2.3$ and apply the Euler--Maruyama method to generate $1100$ sample paths of $\theta_t$. In this simulation, we use
$t= 5000, \thinspace dt =0.1,$ and $N = t/dt$. We consider four learning-rate magnitudes,
$C_\alpha =0.43, \thinspace 0.72, \thinspace 0.78$ and $1.0.$
\begin{figure}[!htbp]
\centering
\begin{subfigure}{.5\textwidth}
  \centering
  \includegraphics[width=.97\linewidth]{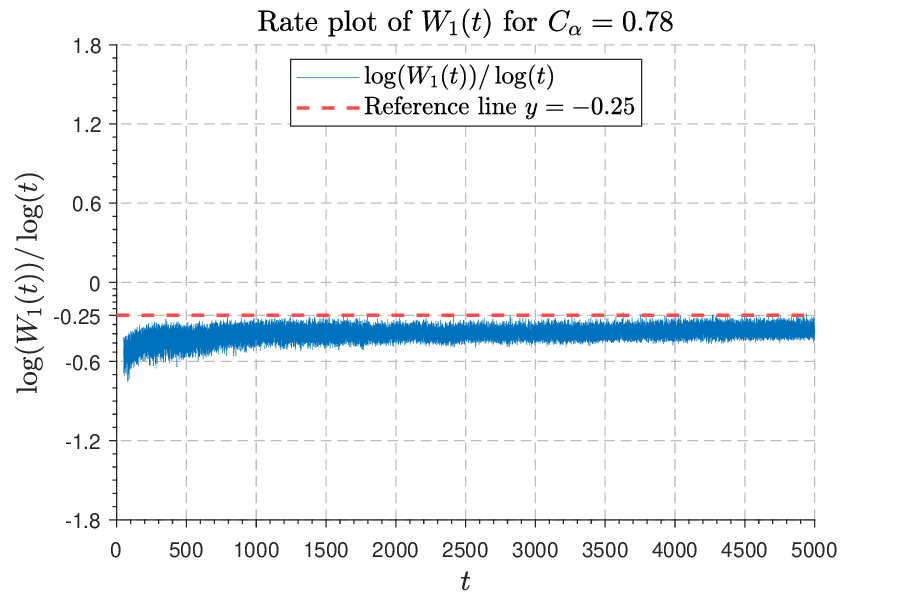}
  \caption{$\log(W_1(t))/\log(t)$}
  \label{fig:sub1}
\end{subfigure}%
\begin{subfigure}{.5\textwidth}
  \centering
  \includegraphics[width=.88\linewidth]{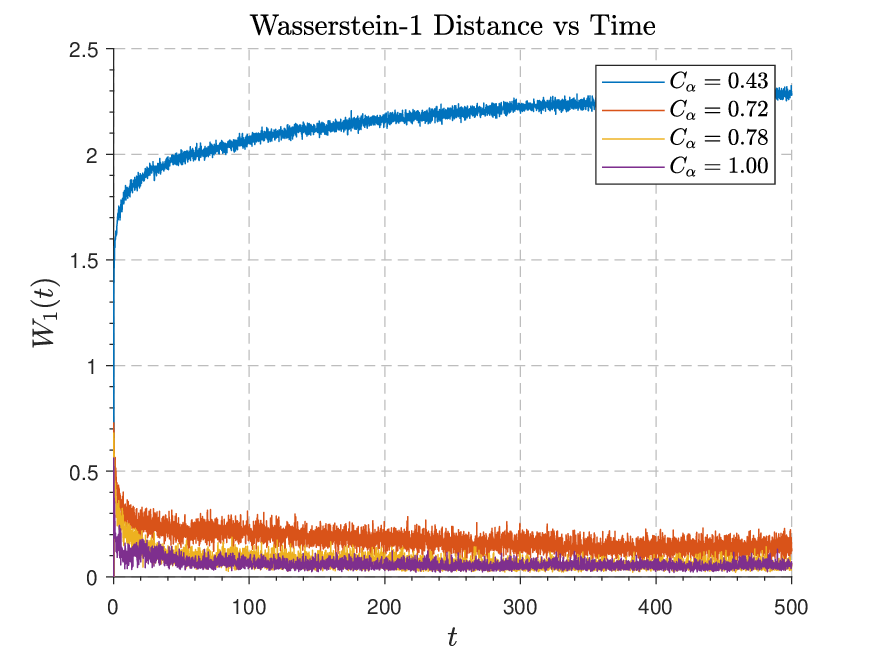}
  \caption{$W_1(t)$}
  \label{fig:sub2}
\end{subfigure}
\caption{X-independent dynamics: The quantities $\frac{\log(d_W(\mathsf{F}_t, N))}{\log(t)}$ and \( d_W(\mathsf{F}_t, N) \) are examined over $1100$ sample paths with $t = 5000.$ For notational convenience, we denote the Wasserstein distance by \(W_1(t)\) in all figures. Since \(C_{\bar{g}} = 1\), the values of \(C_\alpha C_{\bar{g}}\) are \(0.43, 0.72, 0.78,\) and \(1.0\). For visualization, in Figure \ref{fig:sub2} we display trajectories only up to \(t = 500\).}
\label{fig:OU-1}
\end{figure}


We analyze the Wasserstein distance, i.e., the error \( d_W(\mathsf{F}_t, N) \), where \( N \sim \mathscr{N}(0, \bar{\Sigma}) \) and \( \mathsf{F}_t = \sqrt{t}(\theta_t - \theta^*) \). We also examine the quantity \( \frac{\log(d_W(\mathsf{F}_t, N))}{\log(t)} \), which is particularly informative for assessing the predicted convergence rates. In this setting, the limiting variance admits the explicit expression
$\bar{\Sigma} = C_\alpha^2 \int_0^\infty e^{-2s \left[C_\alpha g_{\theta \theta}(\theta^*)- \frac{1}{2} \right]}ds = C_\alpha^2 \int_0^\infty e^{-2s (C_\alpha - \frac{1}{2})}ds = \frac{C_\alpha^2}{2 (C_\alpha - 0.5)}$,
provided $C_\alpha - 0.5> 0.$ In Figure \ref{fig:sub1}, we observe that in the regime \( C_\alpha C_{\bar{g}} = 0.78 > 0.75 \), the quantity \( \frac{\log(d_W(\mathsf{F}_t, N))}{\log(t)} \) falls below \(-\frac{1}{4}\), in agreement with the theoretical prediction
\( \log(d_W(\mathsf{F}_t, N)) \le \log (K) + \log \log(t) - \frac{1}{4} \log(t) \).
In Figure \ref{fig:sub2}, we compare the behavior of \( d_W(\mathsf{F}_t, N) \) for the four learning-rate magnitudes $C_\alpha =0.43, \thinspace 0.72, \thinspace 0.78$ and $1.0.$ We note that the trajectory corresponding to $C_\alpha = 0.43$ does not converge, since $C_\alpha C_{\bar{g}}< 0.5.$
\end{example}

\begin{example}\label{Ex:OU}
We next consider the {\sc ou} process governed by the {\sc sde}
$$dX_t = -c^* X_t \thinspace dt + dW_t,$$
where \(W_t\) is a one-dimensional Brownian motion and \(f^*(x) \triangleq -c^*x\) is the unknown function through the unknown parameter $c^*$. We use {\sc sgdct} to estimate the parameter \(\theta^* \triangleq c^*\). In the estimated model \(f(x, \theta)\), the parameter \(\theta\) evolves according to
$d\theta_t = -\alpha_t X_t \left[dX_t + \theta_t X_t \thinspace dt\right],$
where \(\alpha_t = \frac{C_\alpha}{1 + t}\) is the learning rate, with \(C_\alpha\) denoting the learning-rate magnitude and \(f^*(x) = f(x, \theta^*)\).
\begin{figure}[H]
\begin{center}
\includegraphics[height=5.5cm,width=9cm]{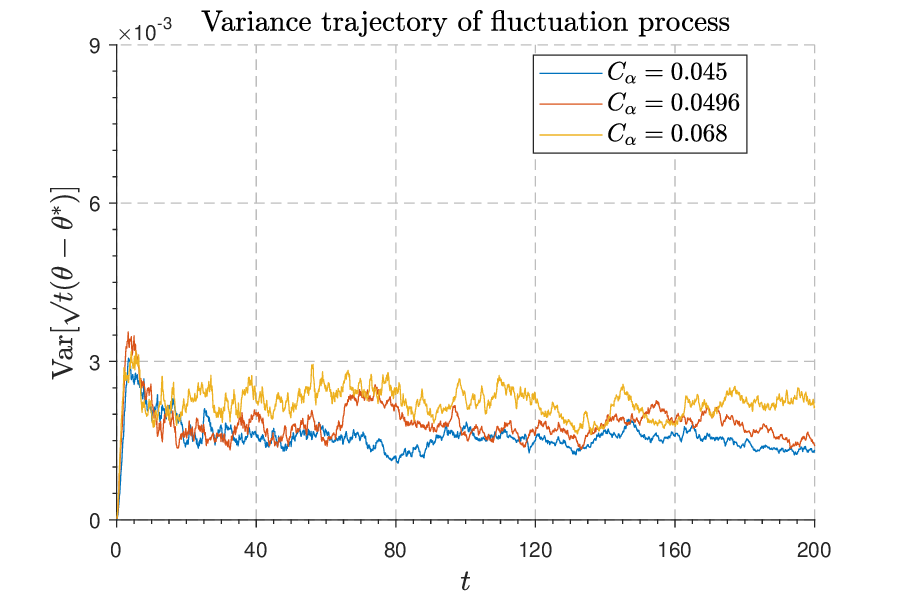}
\caption{{\sc ou} process: We numerically estimate the limiting variance $\bar{\Sigma}$ for three values of \(C_\alpha \): \(0.045\), \(0.0496\), and \(0.068\). The remaining parameters are $t= 7000, \thinspace dt = 0.1, \thinspace \theta^* =  0.031$. Since \(C_{\bar{g}} = 1/2\theta^* = 1/0.062 \), the corresponding values of \(C_\alpha C_{\bar{g}}\) are $0.72$, $0.8$, and $1.1$. At \(t = 6500\), we obtain the estimates \(\bar{\Sigma} \approx 0.0016\), \(0.002\), and \(0.0028\), respectively. For visualization, we display trajectories only up to \(t = 200\).}\label{F:SGD-W-1}
\end{center}
\end{figure}

We fix $\theta^*=0.031$ and apply the Euler--Maruyama method to generate $150$ sample paths of the processes $X_t$ and $\theta_t$. We use
$t= 7000,$ $dt =0.1$, $N = t/dt$, and $150$ Monte Carlo runs. We consider three learning-rate magnitudes $C_\alpha = 0.045, \thinspace 0.0496, \thinspace 0.068$. Since $C_{\bar{g}} =1/2\theta^*$, the corresponding values of $C_\alpha C_{\bar{g}}$ are $0.72$, $0.8$, and $1.1$.

We analyze the Wasserstein distance, i.e., the error \( d_W(\mathsf{F}_t, N) \), where \( N \sim \mathscr{N}(0, \bar{\Sigma}) \) and \( \mathsf{F}_t = \sqrt{t}(\theta_t - \theta^*) \). We also examine the quantity \( \frac{\log(d_W(\mathsf{F}_t, N))}{\log(t)} \), which is central for validating the predicted rates. Since the definition of the limiting variance involves the solution of a Poisson equation, we estimate \(\bar{\Sigma}\) numerically for the three learning-rate magnitudes. As shown in Figure \ref{F:SGD-W-1}, at $t= 6500$ we obtain
\(\bar{\Sigma} \approx 0.0016\), \(0.002\), and \(0.0028\) for \(C_\alpha = 0.045\), \(0.0496\), and \(0.068\), respectively. In Figure \ref{fig:sub11}, we observe that in the regime \( C_\alpha C_{\bar{g}} = 1.1 > 0.75 \), the quantity \( \frac{\log(d_W(\mathsf{F}_t, N))}{\log(t)} \) falls below \(-\frac{1}{4}\), consistent with the prediction
\( \log(d_W(\mathsf{F}_t, N)) \le \log (K) + \log \log(t) - \frac{1}{4} \log(t) \).
In Figure \ref{fig:sub22}, we compare the behavior of \( d_W(\mathsf{F}_t, N) \) for the three learning-rate magnitudes $C_\alpha =0.045, \thinspace 0.0496,$ and $0.068.$
\begin{figure}[!htbp]
\centering
\begin{subfigure}{.5\textwidth}
  \centering
  \includegraphics[width=.97\linewidth]{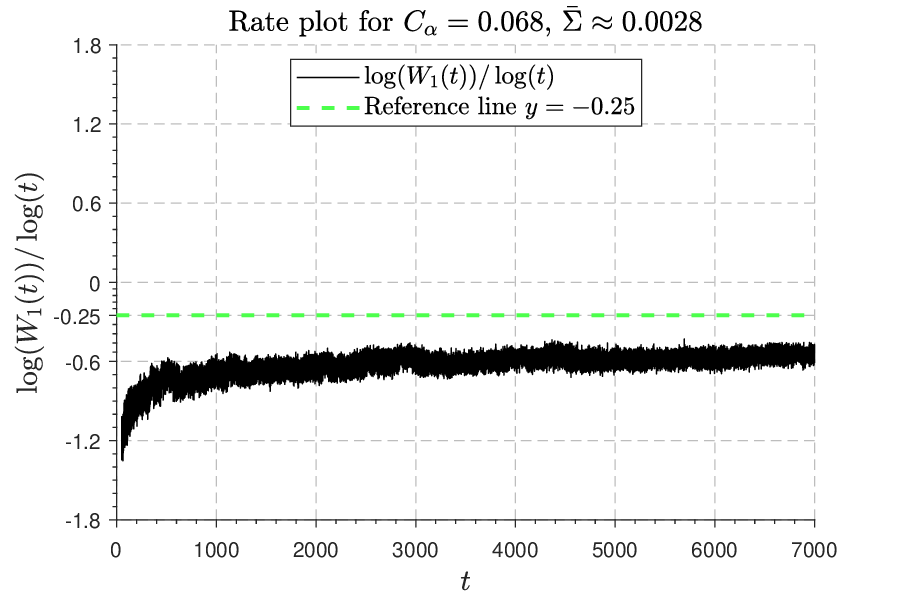}
  \caption{$\log(W_1(t))/\log(t)$}
  \label{fig:sub11}
\end{subfigure}%
\begin{subfigure}{.5\textwidth}
  \centering
  \includegraphics[height=5.3cm,width=8.5cm]{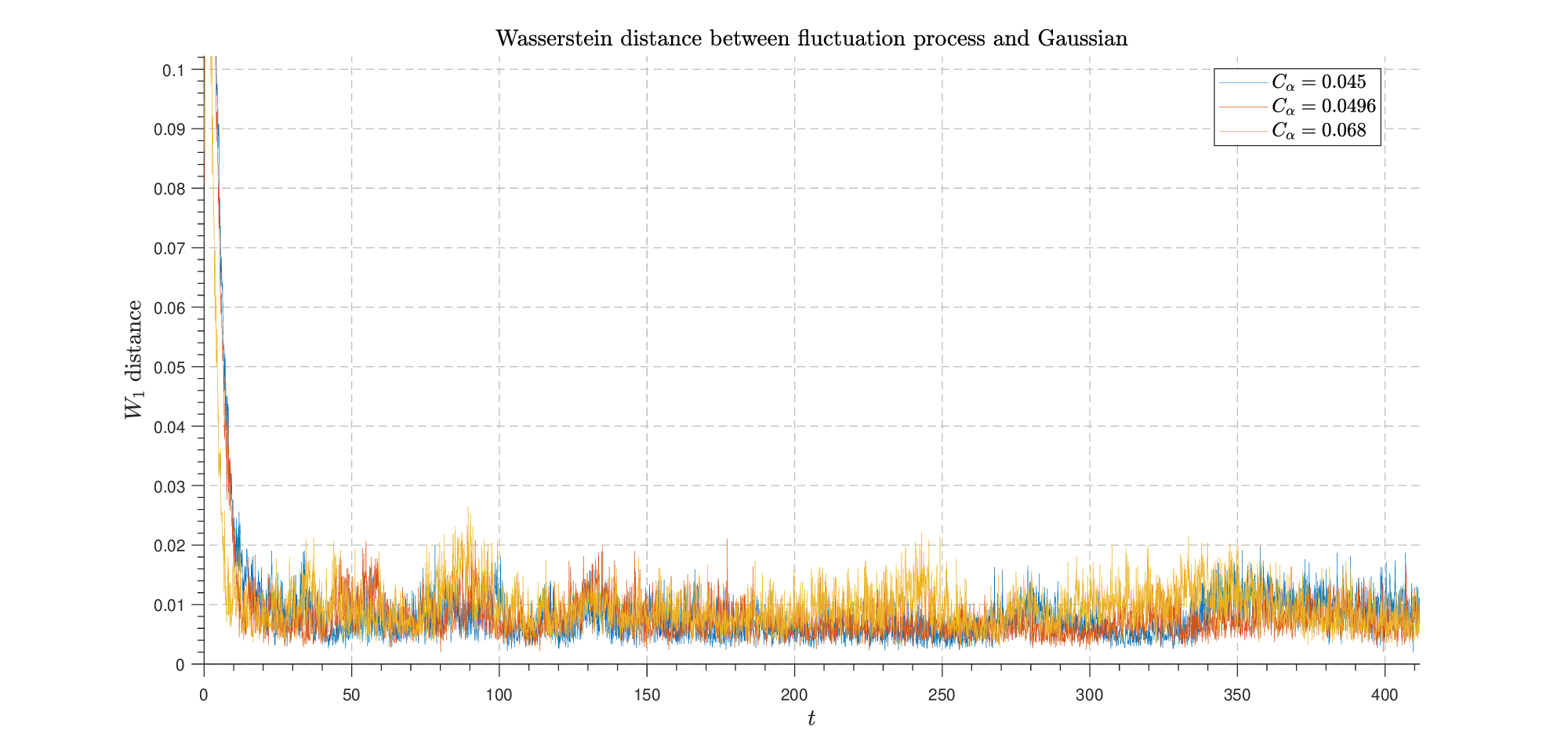}
  \caption{$W_1(t)$}
  \label{fig:sub22}
\end{subfigure}
\caption{{\sc ou} process: The quantities $\frac{\log(d_W(\mathsf{F}_t, N))}{\log(t)}$ and \( d_W(\mathsf{F}_t, N) \) are examined over $150$ sample paths and $150$ Monte Carlo runs with $t = 7000.$ For notational convenience, we denote the Wasserstein distance by \(W_1(t)\) in all figures. For visualization, in Figure \ref{fig:sub22} we display trajectories only up to \(t = 400\).}
\label{fig:OU-1}
\end{figure}

\end{example}

\begin{example}\label{Ex:Cubic}
We now illustrate Theorem \ref{T:Main-theorem} for the process $X_t$ governed by the {\sc sde}
$$dX_t = -c^* X_t^3 dt + dW_t,$$
where \(W_t\) is a one-dimensional Brownian motion and \(f^*(x) \triangleq -c^*x^3\) is the unknown function through the unknown parameter $c^*$. We use {\sc sgdct} to estimate the parameter \(\theta^* \triangleq c^*\). In the estimated model \(f(x, \theta)\), the parameter \(\theta\) evolves according to
$d\theta_t = -\alpha_t X_t^3 \left[dX_t + \theta_t X_t^3 \thinspace dt\right],$
where \(\alpha_t = \frac{C_\alpha}{1 + t}\) is the learning rate, with \(C_\alpha\) denoting the learning-rate magnitude and \(f^*(x) = f(x, \theta^*)\).

\begin{figure}[H]
\begin{center}
\includegraphics[height=5.3cm,width=9cm]{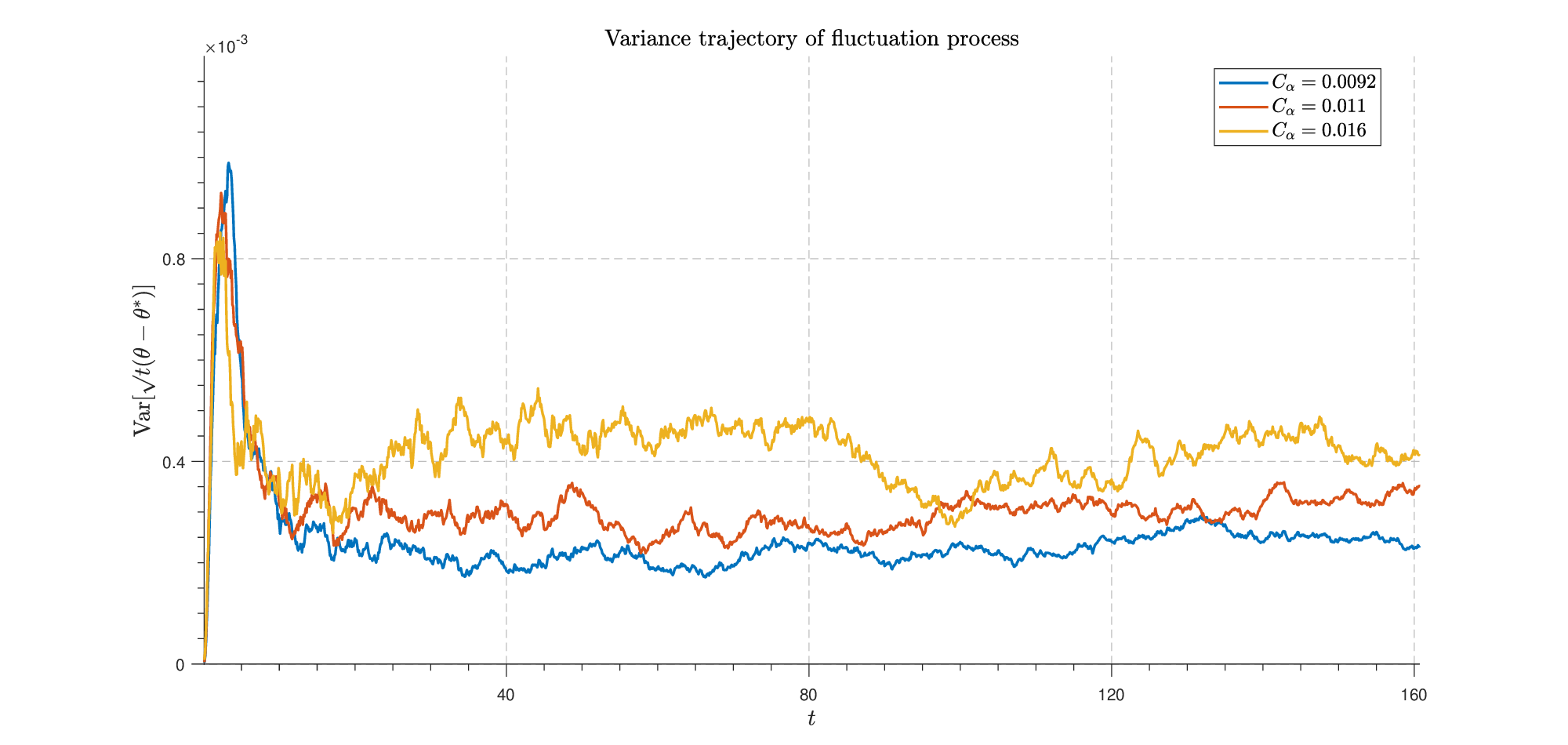}
\caption{Cubic drift: We numerically estimate the limiting variance $\bar{\Sigma}$ for three values of \(C_\alpha \): \(0.0092\), \(0.011\), and \(0.016\). The remaining parameters are $t= 2000, \thinspace dt = 0.1, \thinspace \theta^* =  0.035$. Since $C_{\bar{g}} \approx 0.253 \left( \frac{2}{\theta^*} \right)^{\frac{3}{2}}$, the corresponding values of $C_{\bar{g}} C_\alpha$ are $1.01$, $1.21$ and $1.7$. At \(t = 1600\), we obtain the estimates \(\bar{\Sigma} \approx  0.0003\), \(0.00034\), and \(0.00038\), respectively. For visualization, we display trajectories only up to \(t = 160\).}\label{fig:sub5}
\end{center}
\end{figure}

We next derive an explicit expression for the convexity constant $C_{\bar{g}}$. For the distance function $g(x,\theta) = \frac{1}{2}x^6 (\theta- \theta^*)^2$ and the objective function $\bar{g}(\theta) = \int_{\BR} g(x,\theta) \mu(dx),$ the density of the invariant measure $\mu$ is given by $\frac{e^{\frac{-c^*}{2}x^4}}{Z_{c^*}}$, where $Z_{c^*} \triangleq \int_{\BR} e^{\frac{-c^*}{2}x^4} dx$; see \cite[Section 4.2.2]{dawson1983critical}. Hence,
$$\bar{g}(\theta) = \frac{(\theta - c^*)^2}{2 Z_{c^*}} \int_{\BR} x^6 e^{\frac{-c^*}{2}x^4} dx.$$
Using elementary integration, we find
$Z_{c^*} = \frac{1}{c^*} \left( \frac{2}{c^*} \right)^{-\frac{3}{4}}\Gamma(\frac{1}{4})$,
and we compute $\int_{\BR} x^6 e^{\frac{-c^*}{2}x^4} dx$ by integration by parts. It follows that the strong convexity constant is
$$C_{\bar{g}} = \frac{1}{ Z_{c^*}} \int_{\BR} x^6 e^{\frac{-c^*}{2}x^4} dx = \frac{\Gamma(\frac{7}{4})}{\Gamma(\frac{1}{4})} \left( \frac{2}{c^*} \right)^{\frac{3}{2}} \approx 0.253 \left( \frac{2}{c^*} \right)^{\frac{3}{2}} .$$

We fix $\theta^*=0.035$ and apply the Euler--Maruyama method to generate $100$ sample paths of the processes $X_t$ and $\theta_t$. We use
$t= 10000,$ $dt=0.1$, $N = t/dt$, and $100$ Monte Carlo runs. We consider three learning-rate magnitudes
$C_\alpha = 0.0092, \thinspace 0.011, \thinspace 0.016,$
for which the corresponding values of $C_{\bar{g}} C_\alpha$ are $1.01$, $1.21$ and $1.7$.

We analyze the Wasserstein distance, i.e., the error \( d_W(\mathsf{F}_t, N) \), where \( N \sim \mathscr{N}(0, \bar{\Sigma}) \) and \( \mathsf{F}_t = \sqrt{t}(\theta_t - \theta^*) \). We also examine the quantity \( \frac{\log(d_W(\mathsf{F}_t, N))}{\log(t)} \), which is used to assess the convergence rate predicted by Theorem \ref{T:Main-theorem}. Since the limiting variance involves a Poisson equation, we estimate \(\bar{\Sigma}\) numerically for the three learning-rate magnitudes. As shown in Figure \ref{fig:sub5}, at $t= 1700$ we obtain
\(\bar{\Sigma} \approx 0.0003\), \(0.0002\), and \(0.00034\) for \(C_\alpha = 0.0092\), \(0.011\), and \(0.016\), respectively. In Figure \ref{fig:sub15}, we observe that in the regime \( C_\alpha C_{\bar{g}} = 1.21 > 0.75 \), the quantity \( \frac{\log(d_W(\mathsf{F}_t, N))}{\log(t)} \) falls below \(-\frac{1}{4}\), consistent with the bound
\( \log(d_W(\mathsf{F}_t, N)) \le \log (K) + \log \log(t) - \frac{1}{4} \log(t) \).
In Figure \ref{fig:sub16}, we compare the behavior of \( d_W(\mathsf{F}_t, N) \) for the three learning-rate magnitudes $C_\alpha =0.0092, \thinspace 0.011,$ and $0.016.$
\begin{figure}[!htbp]
\centering
\begin{subfigure}{.5\textwidth}
  \centering
  \includegraphics[width=.97\linewidth]{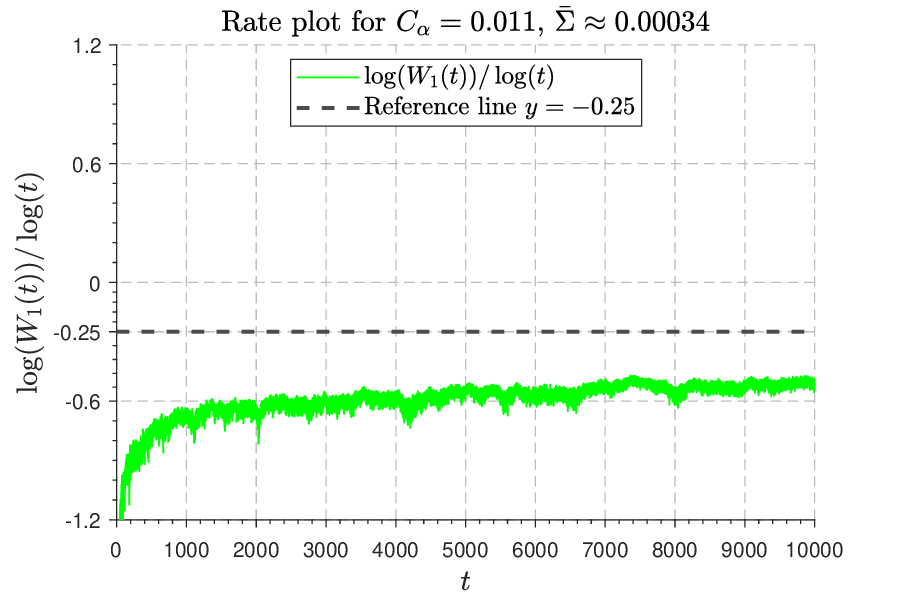}
  \caption{$ \log(W_1(t))/\log(t) $}
  \label{fig:sub15}
\end{subfigure}%
\begin{subfigure}{.5\textwidth}
  \centering
  \includegraphics[height=5.4cm,width=8.5cm]{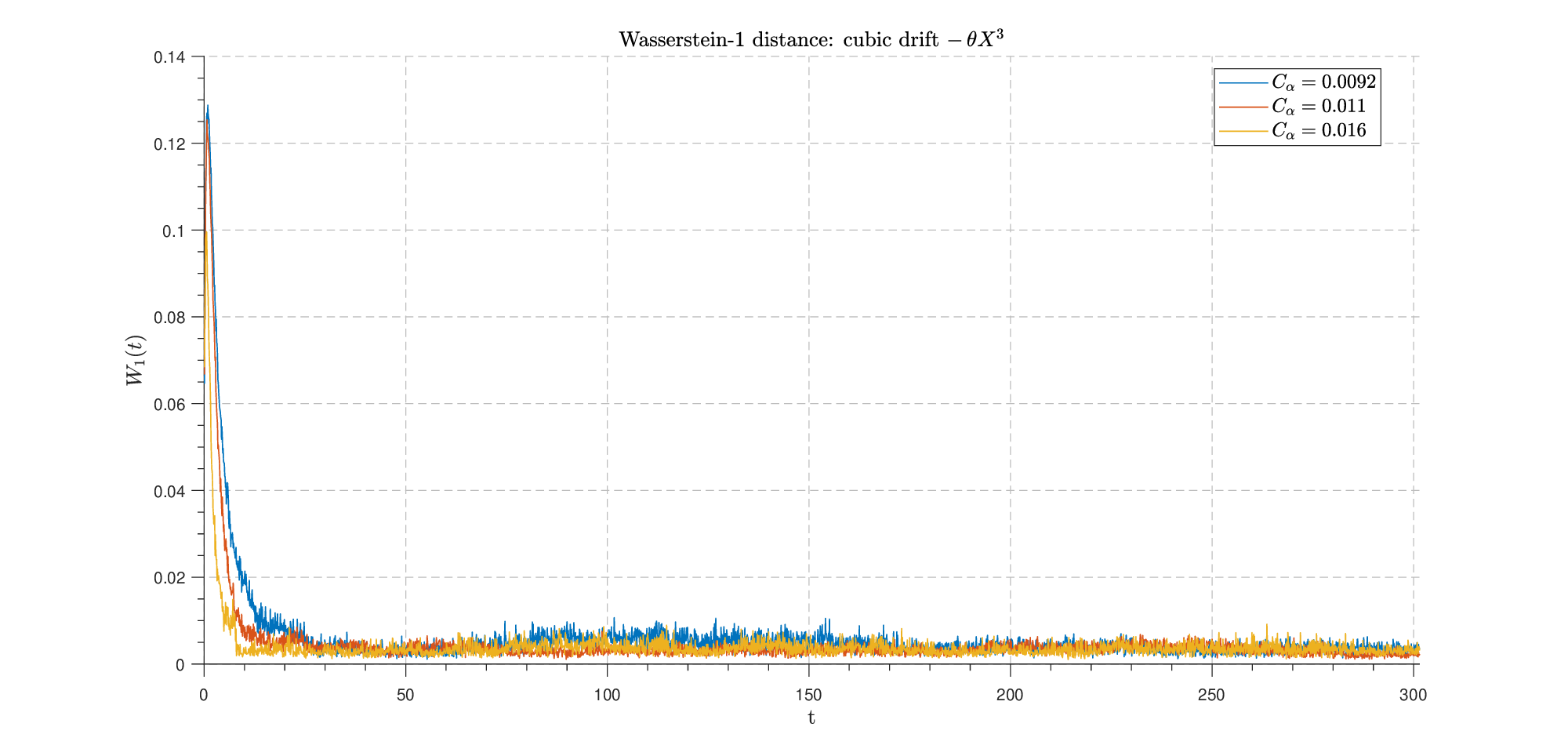}
  \caption{$W_1(t)$}
  \label{fig:sub16}
\end{subfigure}
\caption{Cubic drift: The quantities $\frac{\log(d_W(\mathsf{F}_t, N))}{\log(t)}$ and \( d_W(\mathsf{F}_t, N) \) are examined over $100$ sample paths and $100$ Monte Carlo runs with $t = 10000.$ For notational convenience, we denote the Wasserstein distance by \(W_1(t)\) in all figures. For visualization, in Figure \ref{fig:sub16} we display trajectories only up to \(t = 300\).}
\label{fig:OU}
\end{figure}

\end{example}

\begin{remark}
As these simulations suggest, our theoretical bound on the actual error is likely suboptimal. Indeed, by applying the Cauchy--Schwarz inequality to the terms $\mathbb{E} \left[ \left( D^2F \otimes_1 D^2F
           \right)(x,y)^2 \right]$ and $\mathbb{E} \left[
           DF(x)^2 \cdot DF(y)^2 \right]$ in Proposition
         \ref{P:vidotto}, one recovers \cite[Corollary 4.2]{NPR09}, which is known to yield suboptimal rates; see \cite[Remarks 4.3 and 6.2]{NPR09} and \cite[Remark
         4.1]{vido_2020}. Avoiding Cauchy--Schwarz in Proposition \ref{P:vidotto} comes at a significant cost: one must compute the terms $\mathbb{E} \left[ \left( D^2F \otimes_1 D^2F
           \right)(x,y)^2 \right]$ and $\mathbb{E} \left[
           DF(x)^2 \cdot DF(y)^2 \right]$ directly, which is substantially more difficult and appears intractable in the present setting of stochastic differential equations. For this reason, we ultimately apply Cauchy--Schwarz so as not to introduce additional complexity into an already long and technical proof. Obtaining sharper bounds through a direct treatment of the terms in Proposition \ref{P:vidotto} therefore remains an open problem.
\end{remark}

\color{black}

\section{First-order Malliavin Derivatives}\label{S:First-order-derivative}
In this section, we bound the quantity $\int_{[t^*,t]^2} \sqrt{\BE [D_r \theta_t^4]} \sqrt{\BE [D_s \theta_t^4]} \thinspace ds \thinspace dr$ by means of Proposition \ref{P:First-der-product} below. This result is required in the proof of Theorem \ref{T:Main-theorem}. The proof of Proposition \ref{P:First-der-product} is built upon two principal ingredients: Lemmas \ref{L:Integrating-Factor-first-der} and \ref{L:1-der-moments}. These lemmas establish uniform moment estimates for the integrating factor process $\eta_{t,r}^*$, defined in \eqref{E:Integrating-factor-1-der}, and the first-order Malliavin derivative $D_r \theta_t$, which satisfies \eqref{E:first-order-Mal-Der-theta}, respectively. Their proofs are provided separately in Section \ref{S:Proof-Lemma-3.2-3.3}. Further, the proof of Lemma \ref{L:Integrating-Factor-first-der} relies on a series of auxiliary Lemmas \ref{L:L-1-deri-1}--\ref{L:L-5-der-1}, which establish the uniform-in-time bounds needed to control the fluctuation term appearing in the integrating factor process through the solution of an associated Poisson equation. For $i \in \{1,2,3\}$ and the process $\zeta_t \triangleq \zeta(X_t, \theta_t)$, the typical quantities considered in these lemmas are of the form $\BE[ e^{ \int_r^t \alpha_u^i \zeta_u \thinspace du}]$ and $\BE[ e^{ \int_r^t \alpha_u^i \zeta_u \thinspace dW_u}]$.

On the other hand, the proof of Lemma \ref{L:1-der-moments} is based on Lemmas \ref{L:L-6-initial-cond-der-1}--\ref{L:L-8-deri-1}, where we derive quantitative moment estimates for the terms involving the Malliavin derivative of the state process. In particular, these lemmas show that the quantities of interest decay at the rate ${r^{ 2p C_{\bar{g}} C_\alpha-2p}}/{t^{ 2p C_{\bar{g}} C_\alpha}}$ as $t \to \infty.$ For a generic process $\xi_t \triangleq \xi(X_t, \theta_t)$, the principal quantities to be estimated are of the form $\BE [( \int_r^t  \alpha_u \xi_u D_r X_u \thinspace du)^{2p}]$ and $\BE[( \int_r^t  \alpha_u \xi_uD_r X_u \thinspace dW_u)^{2p}].$ The proofs of these auxiliary lemmas are presented in the remainder of this section. The above discussion is summarized in Figure \ref{F:Fig-Proof-Dec-1}.

\begin{figure}[!htbp]
\begin{equation*}
\begin{tikzpicture}[
  level 1/.style={sibling distance=45mm},
  level 2/.style={sibling distance=37mm},
  every node/.style={draw, rounded corners, align=center}
]
\node {Proposition \ref{P:First-der-product}}
  child {node {Lemma \ref{L:Integrating-Factor-first-der}}
    child {node {Lemmas \ref{L:L-1-deri-1}--\ref{L:L-5-der-1}}
    }
  }
  child {node {Lemma \ref{L:1-der-moments}}
    child {node {Lemmas \ref{L:L-6-initial-cond-der-1}--\ref{L:L-8-deri-1}}
    }
  };
\end{tikzpicture}
\end{equation*}
\caption{Schematic overview of the proof of Proposition \ref{P:First-der-product}.}\label{F:Fig-Proof-Dec-1}
\end{figure}


The first-order Malliavin derivatives of the processes $X_t$ and $\theta_t$, which satisfy \eqref{E:Process-X} and \eqref{E:Process-theta}, respectively, are given by
\begin{equation}\label{E:first-order-Mal-Der-X}
D_r X_t  = 1+ \int_r^t f_x^* (X_u) D_rX_u \thinspace du.
\end{equation}
\begin{multline}\label{E:first-order-Mal-Der-theta}
\begin{aligned}
D_r \theta_t& = \alpha_r f_\theta(X_r,\theta_r)- \int_r^t \left[\alpha_u \bar{g}_{\theta \theta}(\theta_u)D_r \theta_u \right]du + \int_r^t  \alpha_u \left[ \bar{g}_{\theta \theta}(\theta_u)-{g}_{\theta \theta}(X_u,\theta_u)\right]D_r \theta_u \thinspace du \\
& \qquad \qquad \qquad \qquad \qquad  - \int_r^t  \alpha_u {g}_{x \theta}(X_u,\theta_u)D_r X_u \thinspace du
+ \int_r^t \left[\alpha_u {f}_{\theta \theta}(X_u,\theta_u)D_r \theta_u \right]dW_u \\
& \qquad \qquad \qquad \qquad \qquad \qquad \qquad \qquad \qquad \qquad \qquad \qquad   + \int_r^t \left[\alpha_u {f}_{x \theta}(X_u,\theta_u)D_r X_u \right]dW_u.
\end{aligned}
\end{multline}

To obtain an explicit representation for $D_r \theta_t$, we introduce the integrating factor
\begin{equation}\label{E:Integrating-factor-1-der}
\eta_{t,r}^* \triangleq e^{- \int_r^t \alpha_u \bar{g}_{\theta \theta}(\theta_u) \thinspace du + \int_r^t \alpha_u \left[ \bar{g}_{\theta \theta}(\theta_u)-{g}_{\theta \theta}(X_u,\theta_u)\right] \thinspace du + \int_r^t \alpha_u f_{\theta \theta}(X_u, \theta_u)\thinspace dW_u -\frac{1}{2} \int_r^t \alpha_u^2 f_{\theta \theta}(X_u, \theta_u)^2 du};
\end{equation}
hence, the solution of \eqref{E:first-order-Mal-Der-theta} can be written as
\begin{equation}\label{E:First-der-Solution}
\begin{aligned}
D_r \theta_t& = \eta_{t,r}^* \alpha_r f_\theta(X_r,\theta_r)
 - \int_r^t  \alpha_u \eta_{t,u}^* {g}_{x \theta}(X_u,\theta_u)D_r X_u \thinspace du \\
 & \qquad \qquad   - \int_r^t \alpha_u^2  \eta_{t,u}^* {f}_{x \theta}(X_u,\theta_u) {f}_{\theta \theta}(X_u,\theta_u) D_r X_u \thinspace du + \int_r^t \left[\alpha_u \eta_{t,u}^* {f}_{x \theta}(X_u,\theta_u)D_r X_u \right]dW_u.
\end{aligned}
\end{equation}

\begin{proposition}\label{P:First-der-product}
Let $D_r \theta_t$ be the solution of \eqref{E:first-order-Mal-Der-theta}, represented in \eqref{E:First-der-Solution}. Then, for any $t > r \ge t^*,$ and  $C_{\bar{g}} C_\alpha> \frac{1}{2}$, there exists a time-independent positive constant $K$ such that
$$\int_{[t^*,t]^2} \sqrt{\BE [D_r \theta_t^4]} \sqrt{\BE [D_s \theta_t^4]} \thinspace ds \thinspace dr \le \frac{K}{t^2}.$$
\end{proposition}

Before proving Proposition \ref{P:First-der-product}, we state two auxiliary results, Lemmas \ref{L:Integrating-Factor-first-der} and \ref{L:1-der-moments}. Their proofs are given in Section \ref{S:Proof-Lemma-3.2-3.3}.

\begin{lemma}\label{L:Integrating-Factor-first-der}
Let $\eta_{t,r}^*$ be defined in \eqref{E:Integrating-factor-1-der}. Then, for any $p \in \BN,$, $t > r,$ and  $C_{\bar{g}} C_\alpha> \frac{1}{2}$, there exists a time-independent positive constant $K$ such that
$$\BE\left[ \left(\eta_{t,r}^*\right)^{2p} \right] \le K \left( \frac{r}{t}\right)^{ 2p C_{\bar{g}}C_\alpha}.$$
\end{lemma}

\begin{lemma}\label{L:1-der-moments}
Let $D_r \theta_t$ be the solution of \eqref{E:first-order-Mal-Der-theta}, given in \eqref{E:First-der-Solution}. Then, for any $t > r,$ $p \in \BN,$ and  $C_{\bar{g}} C_\alpha> \frac{1}{2}$, there exists a time-independent positive constant $K$ such that
$$\BE \left[ \left(D_r \theta_t \right)^{2p} \right] \le K \frac{r^{ 2p C_{\bar{g}}C_\alpha-2p}}{t^{ 2p C_{\bar{g}}C_\alpha}}. $$
\end{lemma}

We now prove Proposition \ref{P:First-der-product}.
\begin{proof}[Proof of Proposition \ref{P:First-der-product}]
Applying Lemma \ref{L:1-der-moments} with $p=2$ and $C_{\bar{g}}C_\alpha> \frac{1}{2}$ yields
\begin{equation*}
\begin{aligned}
\int_{[t^*,t]^2} \sqrt{\BE [D_r \theta_t^4]} \sqrt{\BE [D_s \theta_t^4]} \thinspace ds \thinspace dr & \le \frac{K}{t^{4C_{\bar{g}}C_\alpha}} \int_{[t^*,t]^2} \sqrt{r^{4C_{\bar{g}}C_\alpha-4}} \sqrt{s^{4C_{\bar{g}} C_\alpha-4}} \thinspace ds \thinspace dr \le \frac{K}{t^2}.
\end{aligned}
\end{equation*}
\end{proof}

\subsection{Proofs of Lemmas \ref{L:Integrating-Factor-first-der} and \ref{L:1-der-moments}}\label{S:Proof-Lemma-3.2-3.3}
Before proving Lemma \ref{L:Integrating-Factor-first-der}, we construct a Poisson equation in order to handle the fluctuation term $\int_r^t \alpha_u \left[ \bar{g}_{\theta \theta}(\theta_u)-{g}_{\theta \theta}(X_u,\theta_u)\right] du$ appearing in the definition of $\eta_{t,r}^*$ in \eqref{E:Integrating-factor-1-der}. Let $\mathscr{L}_x \triangleq f^*(x) \frac{d}{dx} + \frac{1}{2} \frac{d^2}{dx^2}$ denote the infinitesimal generator of the process $X$ solving \eqref{E:Process-X}, and let $\Phi(x,\theta)$ be the solution of the Poisson equation
\begin{equation}\label{E:Poisson-equation}
\mathscr{L}_x \Phi (x, \theta)= \mathsf{G}(x,\theta), \quad {\text{where}} \quad \mathsf{G}(x,\theta) \triangleq \bar{g}_{\theta \theta}(\theta)-{g}_{\theta \theta}(x,\theta).
\end{equation}
We further define
\begin{equation}\label{E:Zeta-intermediate-terms-deri-1}
\begin{aligned}
\zeta_1(X_t, \theta_t) & \triangleq f_{\theta \theta}(X_t, \theta_t) - \Phi_x(X_t, \theta_t), \\
\zeta_2(X_t, \theta_t) &\triangleq \frac{1}{C_\alpha}\Phi(X_t, \theta_t) + g_\theta(X_t, \theta_t)\Phi_\theta(X_t, \theta_t)- \frac{1}{2}\Phi_{x \theta}(X_t,\theta_t)f_\theta(X_t,\theta_t)- \frac{1}{2}f_{\theta \theta}(X_t, \theta_t)^2,\\
 \zeta_3(X_t, \theta_t)  & \triangleq -\frac{1}{2}f_{\theta}(X_t, \theta_t)^2 \Phi_{\theta \theta}(X_t, \theta_t),\quad \text{and} \quad
\zeta_4(X_t, \theta_t)  \triangleq -f_{\theta}(X_t, \theta_t) \Phi_{\theta}(X_t, \theta_t).
\end{aligned}
\end{equation}

To prepare the proof of Lemma \ref{L:Integrating-Factor-first-der}, we first state a collection of auxiliary lemmas, whose proofs are deferred to Section \ref{S:Proofs-lemmas-3.4-3.11}.
\begin{lemma}\label{L:L-1-deri-1}
Let $\Phi$ be the solution of \eqref{E:Poisson-equation}, and let $\mathsf{G}(x,\theta)$ be defined in \eqref{E:Poisson-equation}. Then, for any $t > r,$ we have
\begin{equation*}
\begin{aligned}
\int_r^t \alpha_u \mathsf{G}(X_u, \theta_u)\thinspace du
& \triangleq [\alpha_t \Phi(X_t, \theta_t) - \alpha_r \Phi (X_r, \theta_r) ] -   \int_r^t  \frac{d \alpha_u}{du}  \Phi(X_u, \theta_u) \thinspace du\\
& + \int_r^t \alpha_u^2 g_\theta(X_u, \theta_u) \Phi_\theta (X_u, \theta_u) \thinspace du
 - \frac{1}{2}\int_r^t \alpha_u^3 [f_\theta (X_u, \theta_u)]^2 \Phi_{\theta \theta}(X_u, \theta_u) \thinspace du  \\
 & - \frac{1}{2}\int_r^t \alpha_u^2 f_\theta (X_u, \theta_u) \Phi_{x \theta}(X_u, \theta_u) \thinspace du \\
& - \int_r^t \alpha_u \Phi_{x}(X_u, \theta_u) \thinspace dW_u -  \int_r^t \alpha_u^2 \Phi_{\theta}(X_u, \theta_u)f_\theta (X_u, \theta_u) \thinspace dW_u.
\end{aligned}
\end{equation*}
\end{lemma}

\begin{lemma}\label{L:L-2-deri-1}
Let $\Phi$ be the solution of \eqref{E:Poisson-equation}. Then, for any $p \in \BN,$ there exists a time-independent positive constant $K$ such that
\begin{equation*}
\BE \left[e^{p\{\alpha_t \Phi(X_t, \theta_t) - \alpha_r \Phi (X_r, \theta_r) \}} \right] \le K.
\end{equation*}
\end{lemma}

\begin{lemma}\label{L:L-3-deri-1}
Let $\Phi$ be the solution of \eqref{E:Poisson-equation}, and let $\zeta_2$ be defined in \eqref{E:Zeta-intermediate-terms-deri-1}. Then, for any $p \in \BN,$ there exists a time-independent positive constant $K$ such that
\begin{equation*}
\BE \left[ e^{p \int_r^t \alpha_u^2 \zeta_2 (X_u, \theta_u) \thinspace du} \right] \le K.
\end{equation*}
\end{lemma}

\begin{lemma}\label{L:L-4-der-1}
Let $\Phi$ be the solution of \eqref{E:Poisson-equation}, and let $\zeta_1$ be defined in \eqref{E:Zeta-intermediate-terms-deri-1}. Then, for any $p \in \BN,$ there exists a time-independent positive constant $K$ such that
\begin{equation*}
\BE \left[ e^{p \int_r^t \alpha_u \zeta_1 (X_u, \theta_u) \thinspace dW_u}\right] \le K.
\end{equation*}
\end{lemma}

\begin{lemma}\label{L:L-5-der-1}
Let $\Phi$ be the solution of \eqref{E:Poisson-equation}, and let $\zeta_3, \thinspace \zeta_4$ be defined in \eqref{E:Zeta-intermediate-terms-deri-1}. Then, for any $p \in \BN,$ there exists a time-independent positive constant $K$ such that
\begin{equation*}
\left[\BE e^{p\int_r^t \alpha_u^3 \zeta_3 (X_u, \theta_u) \thinspace du}\right] \left[\BE e^{p \int_r^t \alpha_u^2 \zeta_4 (X_u, \theta_u) \thinspace dW_u}  \right]  \le K.
\end{equation*}
\end{lemma}

We now prove Lemma \ref{L:Integrating-Factor-first-der}.
\begin{proof}[Proof of Lemma \ref{L:Integrating-Factor-first-der}]
Combining \eqref{E:Integrating-factor-1-der}, \eqref{E:Zeta-intermediate-terms-deri-1}, and Lemma \ref{L:L-1-deri-1}, we obtain
\begin{multline*}
\eta_{t,r}^*  = e^{- \int_r^t \alpha_u \bar{g}_{\theta \theta}(\theta_u) \thinspace du} \cdot e^{[\alpha_t \Phi(X_t, \theta_t) - \alpha_r \Phi (X_r, \theta_r) ]} \cdot
 e^{\int_r^t \alpha_u \zeta_1 (X_u, \theta_u) \thinspace dW_u}\cdot e^{\int_r^t \alpha_u^2 \zeta_2 (X_u, \theta_u) \thinspace du}\\
 \times e^{\int_r^t \alpha_u^3 \zeta_3 (X_u, \theta_u) \thinspace du} \cdot  e^{\int_r^t \alpha_u^2 \zeta_4 (X_u, \theta_u) \thinspace dW_u}.
\end{multline*}
Raising both sides to the power $2p$, using the strong convexity of $\bar{g}$ (Assumption \ref{A:Growth-f-g}), i.e., $\bar{g}_{\theta \theta}> C_{\bar{g}}$ for some $C_{\bar{g}}>0$, and then taking expectations, we obtain
\begin{multline*}
\BE\left[ \left(\eta_{t,r}^*\right)^{2p} \right] \le e^{- 2p C_{\bar{g}} \int_r^t \alpha_u du} \thinspace \BE \left[e^{2p[\alpha_t \Phi(X_t, \theta_t) - \alpha_r \Phi (X_r, \theta_r) ]}\cdot e^{2p \int_r^t \alpha_u \zeta_1 (X_u, \theta_u) \thinspace dW_u} \times   \right. \\
\left.  e^{2p \int_r^t \alpha_u^2 \zeta_2 (X_u, \theta_u) \thinspace du} \cdot e^{2p\int_r^t \alpha_u^3 \zeta_3 (X_u, \theta_u) \thinspace du} \cdot  e^{\int_r^t \alpha_u^2 \zeta_4 (X_u, \theta_u) \thinspace dW_u}  \right].
\end{multline*}
Applying H{\"o}lder's inequality and using the learning rate\footnote{In the definition of the learning rate, we take $C_0=0$ as $t \ge r>0.$} $\alpha_t = \frac{C_\alpha}{t}$ yields
\begin{multline*}
\BE\left[ \left(\eta_{t,r}^*\right)^{2p} \right] \le \left(\frac{r}{t} \right)^{ 2p C_{\bar{g}} C_\alpha} \left[\BE e^{10p[\alpha_t \Phi(X_t, \theta_t) - \alpha_r \Phi (X_r, \theta_r) ]} \right]^{\frac{1}{5}} \left[\BE e^{10 p \int_r^t \alpha_u \zeta_1 (X_u, \theta_u) \thinspace dW_u}  \right]^{\frac{1}{5}} \times \\
 \left[\BE e^{10 p \int_r^t \alpha_u^2 \zeta_2 (X_u, \theta_u) \thinspace du} \right]^{\frac{1}{5}}
\left[\BE e^{10 p\int_r^t \alpha_u^3 \zeta_3 (X_u, \theta_u) \thinspace du}   \right]^{\frac{1}{5}} \left[\BE e^{10 p \int_r^t \alpha_u^2 \zeta_4 (X_u, \theta_u) \thinspace dW_u}  \right]^{\frac{1}{5}}.
\end{multline*}
Finally, Lemmas \ref{L:L-2-deri-1}, \ref{L:L-3-deri-1}, \ref{L:L-4-der-1}, and \ref{L:L-5-der-1} provide uniform bounds for the five expectation terms on the right-hand side, which yields the desired estimate.
\end{proof}

We now begin the proof of Lemma \ref{L:1-der-moments}. We first state three auxiliary results, Lemmas \ref{L:L-6-initial-cond-der-1}, \ref{L:L-7-deri-1}, and \ref{L:L-8-deri-1}, whose proofs are given in Section \ref{S:Proofs-lemmas-3.4-3.11}.
\begin{lemma}\label{L:L-6-initial-cond-der-1}
Let $\eta_{t,r}^*$ be defined in Equation \eqref{E:Integrating-factor-1-der}. Then for any $p \in \BN,$ there exists a time-independent positive constant $K$ such that
\begin{equation*}
\BE \left[ (\eta_{t,r}^*)^{2p} \alpha_r^{2p} f_\theta(X_r,\theta_r)^{2p} \right] \le K \frac{r^{ 2p C_{\bar{g}} C_\alpha-2p}}{t^{ 2p C_{\bar{g}} C_\alpha}}.
\end{equation*}
\end{lemma}

\begin{lemma}\label{L:L-7-deri-1}
Let $\eta_{t,r}^*$ be defined in Equation \eqref{E:Integrating-factor-1-der}. Then for any $p \in \BN,$ there exists a time-independent positive constant $K$ such that
\begin{multline*}
\BE \left[ \left( \int_r^t  \alpha_u \eta_{t,u}^* {g}_{x \theta}(X_u,\theta_u)D_r X_u \thinspace du \right)^{2p} \right]
 + \\ \BE \left[ \left( \int_r^t \alpha_u^2  \eta_{t,u}^* {f}_{x \theta}(X_u,\theta_u) {f}_{\theta \theta}(X_u,\theta_u) D_r X_u \thinspace du \right)^{2p} \right]   \le K \frac{r^{ 2p C_{\bar{g}} C_\alpha-2p}}{t^{ 2p C_{\bar{g}} C_\alpha}}.
\end{multline*}
\end{lemma}

\begin{lemma}\label{L:L-8-deri-1}
Let $\eta_{t,r}^*$ be defined in Equation \eqref{E:Integrating-factor-1-der}. Then for any $p \in \BN,$ there exists a time-independent positive constant $K$ such that
\begin{equation*}
\BE \left[ \left( \int_r^t \left[\alpha_u \eta_{t,u}^* {f}_{x \theta}(X_u,\theta_u)D_r X_u \right]dW_u  \right)^{2p} \right]  \le K \frac{r^{ 2p C_{\bar{g}} C_\alpha-2p}}{t^{ 2p C_{\bar{g}} C_\alpha}}.
\end{equation*}
\end{lemma}

We now prove Lemma \ref{L:1-der-moments}.
\begin{proof}[Proof of Lemma \ref{L:1-der-moments}]
Raising both sides of \eqref{E:First-der-Solution} to the power \(2p\) and taking expectations, we obtain
\begin{equation*}
\begin{aligned}
\BE \left[ \left(D_r \theta_t \right)^{2p} \right] & \le \BE \left[ (\eta_{t,r}^*)^{2p} \alpha_r^{2p} f_\theta(X_r,\theta_r)^{2p} \right]+ \BE \left[ \left( \int_r^t  \alpha_u \eta_{t,u}^* {g}_{x \theta}(X_u,\theta_u)D_r X_u \thinspace du \right)^{2p} \right] \\
& \qquad \qquad \qquad \qquad \qquad   + \BE \left[ \left( \int_r^t \alpha_u^2  \eta_{t,u}^* {f}_{x \theta}(X_u,\theta_u) {f}_{\theta \theta}(X_u,\theta_u) D_r X_u \thinspace du \right)^{2p} \right] \\
& \qquad \qquad \qquad \qquad \qquad \qquad \qquad \quad    + \BE \left[ \left( \int_r^t \left[\alpha_u \eta_{t,u}^* {f}_{x \theta}(X_u,\theta_u)D_r X_u \right]dW_u  \right)^{2p} \right].
\end{aligned}
\end{equation*}
Applying Lemmas \ref{L:L-6-initial-cond-der-1}, \ref{L:L-7-deri-1}, and \ref{L:L-8-deri-1} yields the desired estimate.
\end{proof}

\subsection{Proofs of Lemmas \ref{L:L-1-deri-1} through \ref{L:L-8-deri-1}}\label{S:Proofs-lemmas-3.4-3.11}
\begin{proof}[Proof of Lemma \ref{L:L-1-deri-1}]
We apply It{\^o}'s formula to the process $\varphi(t,\Phi)= \alpha_t \Phi(X_t,\theta_t)$, where $\Phi$ solves the Poisson Equation \eqref{E:Poisson-equation}, and we recall that $\theta_t$ satisfies the {\sc sde} $d \theta_t =  - \alpha_t  g_\theta(X_t,\theta_t)  dt + \alpha_t f_\theta(X_t,\theta_t)dW_t$. This yields
\begin{equation*}
\begin{aligned}
\alpha_t &  \Phi(X_t, \theta_t)  = \alpha_r  \Phi (X_r, \theta_r) + \int_r^t  \frac{d \alpha_u}{du} \Phi(X_u, \theta_u) \thinspace du + \int_r^t \alpha_u \mathscr{L}_x \Phi (X_u, \theta_u)\thinspace du \\
& \qquad + \frac{1}{2}\int_r^t \alpha_u^3 [f_\theta (X_u, \theta_u)]^2 \Phi_{\theta \theta}(X_u, \theta_u) \thinspace du  + \frac{1}{2}\int_r^t \alpha_u^2 f_\theta (X_u, \theta_u) \Phi_{x \theta}(X_u, \theta_u) \thinspace du \\
& \qquad - \int_r^t \alpha_u^2 g_\theta(X_u, \theta_u) \Phi_\theta (X_u, \theta_u) \thinspace du + \int_r^t \alpha_u \Phi_{x}(X_u, \theta_u) \thinspace dW_u +  \int_r^t \alpha_u^2 \Phi_{\theta}(X_u, \theta_u)f_\theta (X_u, \theta_u) \thinspace dW_u.
\end{aligned}
\end{equation*}
Combining this identity with \eqref{E:Poisson-equation} yields
\begin{equation*}
\begin{aligned}
\int_r^t \alpha_u \mathsf{G}(X_u, \theta_u)\thinspace du
& = [\alpha_t \Phi(X_t, \theta_t) - \alpha_r \Phi (X_r, \theta_r) ] - \int_r^t  \frac{d \alpha_u}{du}  \Phi(X_u, \theta_u) \thinspace du\\
& \qquad \quad + \int_r^t \alpha_u^2 g_\theta(X_u, \theta_u) \Phi_\theta (X_u, \theta_u) \thinspace du
 - \frac{1}{2}\int_r^t \alpha_u^3 [f_\theta (X_u, \theta_u)]^2 \Phi_{\theta \theta}(X_u, \theta_u) \thinspace du  \\
 & \qquad \quad  - \frac{1}{2}\int_r^t \alpha_u^2 f_\theta (X_u, \theta_u) \Phi_{x \theta}(X_u, \theta_u) \thinspace du \\
& \qquad \quad - \int_r^t \alpha_u \Phi_{x}(X_u, \theta_u) \thinspace dW_u -  \int_r^t \alpha_u^2 \Phi_{\theta}(X_u, \theta_u)f_\theta (X_u, \theta_u) \thinspace dW_u,
\end{aligned}
\end{equation*}
which is exactly the claimed representation.
\end{proof}

\begin{proof}[Proof of Lemma \ref{L:L-2-deri-1}]
For \(t > r \ge r_0 > 0 \) and \(\alpha_t < \alpha_r < \alpha_{r_0}\), we use the bound
\( \alpha_t \Phi(X_t, \theta_t) - \alpha_r \Phi(X_r, \theta_r) \leq |\alpha_t \Phi(X_t, \theta_t) - \alpha_r \Phi(X_r, \theta_r)| \),
apply the triangle inequality, and then expand the exponential in a power series to obtain
\begin{equation*}
\begin{aligned}
 \BE \left[ e^{10 p[\alpha_t \Phi(X_t, \theta_t) - \alpha_r \Phi (X_r, \theta_r) ]} \right] & \le \BE e^{10 p [\alpha_t|\Phi(X_t, \theta_t)|+ \alpha_r|\Phi(X_r, \theta_r)|]} \\
 & \le \BE e^{10 p \alpha_{r_0} [|\Phi(X_t, \theta_t)|+|\Phi(X_r, \theta_r)|]} = \BE \left[ \sum_{k=0}^\infty \frac{K^k \{|\Phi(X_t, \theta_t)|+|\Phi(X_r,\theta_r)|\}^k}{k!} \right] \\
 & \le \sum_{k=0}^{\infty} \frac{K^k}{k!} \left[ \BE|\Phi(X_t, \theta_t)|^k + \BE|\Phi(X_r, \theta_r)|^k \right],
\end{aligned}
\end{equation*}
where the last step follows from the Monotone Convergence Theorem
({\sc mct}) \cite{folland1999real} since $\Phi(x, \theta)$ has polynomial growth in $x$ and $\theta$, and since we have uniform moment bounds \cite{pardoux2001poisson} $\sup_{t \ge 0}\BE\left[|X_t|^p\right] < \infty$ and \cite{siri_spilio_2020} $\sup_{t \ge 0} \BE\left[ |\theta_t|^p \right] < \infty,$ it follows that there exists $K>0$ such that \\
$\BE \left[ e^{10 p[\alpha_t \Phi(X_t, \theta_t) - \alpha_r \Phi (X_r, \theta_r) ]} \right] \le K$.
\end{proof}

\begin{proof}[Proof of Lemma \ref{L:L-3-deri-1}]
Recalling the definition of $\zeta_2$ in \eqref{E:Zeta-intermediate-terms-deri-1}, we apply H{\"o}lder's inequality and use the bound $\zeta_2 (X_u, \theta_u) \le |\zeta_2 (X_u, \theta_u)|$ to obtain
\begin{equation*}
\begin{aligned}
\BE e^{10 p \int_r^t \alpha_u^2 \zeta_2 (X_u, \theta_u) \thinspace du} & \le \BE \sum_{k=0}^\infty \frac{(10 p)^k \left( \int_r^t \alpha_u^2 \zeta_2 (X_u, \theta_u) \thinspace du \right)^k}{k!}\\
& = \sum_{k=0}^{\infty} \frac{(10 p)^k}{k!} \left\{ \int_r^t \cdots \int_r^t \BE \left[ \prod_{i=1}^k \alpha^2_{s_i} \zeta_2 (X_{u_i}, \theta_{u_i}) \right]du_1 \cdots du_k  \right\} \\
& = \sum_{k=0}^{\infty} \frac{(10 p)^k}{k!} \left\{ \int_r^t \cdots \int_r^t \prod_{i=1}^k \alpha^2_{s_i} \left[ \BE \zeta_2 (X_{u_i}, \theta_{u_i})^k \right]^{\frac{1}{k}} du_1 \cdots du_k  \right\}\\
& \le \sum_{k=0}^{\infty} \frac{(10 p)^k}{k!} \left\{ \int_r^t \cdots \int_r^t \prod_{i=1}^k \alpha^2_{s_i} \left[ \BE |\zeta_2 (X_{u_i}, \theta_{u_i})|^k \right]^{\frac{1}{k}} du_1 \cdots du_k  \right\}.
\end{aligned}
\end{equation*}
Using the polynomial growth of $\zeta_2(x, \theta)$ in $x$ and $\theta$, the uniform moment bounds for $X_t$ and $\theta_t$ (namely, \cite{pardoux2001poisson} $\sup_{t \ge 0}\BE\left[|X_t|^p\right] < \infty$ and \cite{siri_spilio_2020} $\sup_{t \ge 0} \BE\left[ |\theta_t|^p \right] < \infty$), and Assumption \ref{A:Learning-rate} (so that $\int_r^t \alpha_u^2 \thinspace du < \infty$), we obtain a constant $K>0$ such that
$\BE e^{10 p \int_r^t \alpha_u^2 \zeta_2 (X_u, \theta_u) \thinspace du} \le \sum_{k=0}^{\infty} \frac{(10p)^k}{k!} \left( \int_r^t \alpha_u^2 \thinspace du \right)^k \le K.$
\end{proof}

\begin{proof}[Proof of Lemma \ref{L:L-4-der-1}]
Recalling the definition of $\zeta_1$ in
\eqref{E:Zeta-intermediate-terms-deri-1}, we use
\begin{equation*}
\int_r^t \alpha_u \zeta_1 (X_u, \theta_u) \thinspace dW_u \le \left|\int_r^t \alpha_u \zeta_1 (X_u, \theta_u) \thinspace dW_u \right|,
\end{equation*}
expand the exponential in a power series, and apply {\sc mct} (justified by Assumption \ref{A:Learning-rate}) to obtain
\begin{equation*}
\begin{aligned}
\BE \left[e^{10 p \int_r^t \alpha_u \zeta_1 (X_u, \theta_u) \thinspace dW_u} \right]  & \le \BE e^{10p \left|\int_r^t \alpha_u \zeta_1 (X_u, \theta_u) \thinspace dW_u \right|} \\
& = \BE \left[ \sum_{k=0}^\infty \frac{(10 p)^k}{k!} \left|\int_r^t \alpha_u \zeta_1 (X_u, \theta_u) \thinspace dW_u \right|^k \right] \\
& = \sum_{k=0}^{\infty} \frac{(10 p)^k}{k!} \BE \left[ \left|\int_r^t \alpha_u \zeta_1 (X_u, \theta_u) \thinspace dW_u \right|^k \right].
\end{aligned}
\end{equation*}
Using the martingale moment inequality \cite[Proposition 3.26]{KS91} and Jensen's inequality for concave functions yields
\begin{equation*}
\begin{aligned}
\BE \left[e^{10 p \int_r^t \alpha_u \zeta_1 (X_u, \theta_u) \thinspace dW_u} \right] & \le K \sum_{k=0}^{\infty} \frac{(10p)^k}{k!} \BE \left[ \left( \int_r^t \alpha_u^2 \thinspace \zeta_1 (X_u, \theta_u)^2 \thinspace du \right)^{\frac{k}{2}} \right] \\
& \le K \sum_{k=0}^{\infty} \frac{(10p)^k}{k!} \sqrt{\BE \left[ \left( \int_r^t \alpha_u^2 \thinspace \zeta_1 (X_u, \theta_u)^2 \thinspace du \right)^k \right]}.
\end{aligned}
\end{equation*}
By H{\"o}lder's inequality, the polynomial growth of $\zeta_1(x, \theta)$, and the uniform moment bounds for $X_t$ and $\theta_t$ \cite{siri_spilio_2020,pardoux2001poisson}, we obtain
\begin{equation*}
\begin{aligned}
\BE \left[e^{10 p \int_r^t \alpha_u \zeta_1 (X_u, \theta_u) \thinspace dW_u} \right]  & \le K \sum_{k=0}^{\infty} \frac{(10 p)^k}{k!} \sqrt{\BE \left[ \int_r^t \cdots \int_r^t \prod_{i=1}^k \alpha^2_{u_i} \zeta_1(X_{u_i}, \theta_{u_i})^2 \thinspace du_1 \cdots du_k \right] }\\
& \le K \sum_{k=0}^{\infty} \frac{(10p)^k}{k!} \sqrt{ \left( \int_r^t \alpha_u^2 \thinspace du \right)^k}.
\end{aligned}
\end{equation*}
Assumption \ref{A:Learning-rate} and $\int_r^t \alpha_u^2 du < \infty$ then yield the desired bound.
\end{proof}

\begin{proof}[Proof of Lemma \ref{L:L-5-der-1}]
The proof follows from arguments analogous to those used in the proofs
of Lemmas \ref{L:L-3-deri-1} and \ref{L:L-4-der-1}. For brevity, we omit the details.
\end{proof}

\begin{proof}[Proof of Lemma \ref{L:L-6-initial-cond-der-1}]
By H\"older's inequality, Assumption \ref{A:Growth-f-g} on the growth of $f_\theta$, and the uniform moment bounds for $X_t$ and $\theta_t$, we obtain
\begin{equation*}
\begin{aligned}
\BE \left[ (\eta_{t,r}^*)^{2p} \alpha_r^{2p} f_\theta(X_r,\theta_r)^{2p} \right] & \le \alpha_r^{2p} \sqrt{\BE \left[(\eta_{t,r}^*)^{4p}\right]} \sqrt{\BE \left[f_\theta(X_r,\theta_r)^{4p} \right]} \le K \frac{r^{ 2p C_{\bar{g}}C_\alpha-2p}}{t^{ 2p C_{\bar{g}} C_\alpha}}.
\end{aligned}
\end{equation*}
The last inequality follows from Lemma \ref{L:Integrating-Factor-first-der}, which controls $\BE \left[(\eta_{t,r}^*)^{4p}\right]$.
\end{proof}

\begin{proof}[Proof of Lemma \ref{L:L-7-deri-1}]
We first prove the bound for $\BE \left[ \left( \int_r^t  \alpha_u
    \thinspace \eta_{t,u}^* \thinspace {g}_{x \theta}(X_u,\theta_u)D_r
    X_u \thinspace du \right)^{2p} \right]$. The bound for $\BE \left[ \left( \int_r^t \alpha_u^2  \eta_{t,u}^* {f}_{x \theta}(X_u,\theta_u) {f}_{\theta \theta}(X_u,\theta_u) D_r X_u \thinspace du \right)^{2p} \right]$ follows by the same argument. H\"older's inequality and Assumption \ref{A:f*-growth}, together with the bound on $D_rX_u$ solving \eqref{E:first-order-Mal-Der-theta}, yield
\begin{equation*}
\begin{aligned}
& \BE \left[ \left( \int_r^t  \alpha_u \thinspace \eta_{t,u}^* \thinspace {g}_{x \theta}(X_u,\theta_u)D_r X_u \thinspace du \right)^{2p} \right] \\
& \qquad \quad \le \underbrace{\int_r^t \cdots \int_r^t}_{2p-\text{times}} \left\{ \prod_{i=1}^{2p} \alpha_{u_i}e^{-C^*(u_i-r)} \right\} \BE \left[ \prod_{i=1}^{2p} \eta_{t,u_i}^* \thinspace {g}_{x \theta}(X_{u_i},\theta_{u_i}) \right]du_1 \cdots du_{2p}  \\
& \qquad \quad \le  \int_r^t \cdots \int_r^t \left\{ \prod_{i=1}^{2p} \alpha_{u_i}e^{-C^*(u_i-r)} \right\} \left[ \prod_{i=1}^{2p} \left( \BE \left[ {\eta_{t,u_i}^*}^{4p} \right] \right)^{\frac{1}{4p}} \Big( \BE \left[{g}_{x \theta}(X_{u_i},\theta_{u_i})^{4p} \right] \Big)^{\frac{1}{4p}} \right] du_1 \cdots du_{2p}.  \\
\end{aligned}
\end{equation*}
Using Assumption \ref{A:Growth-f-g} for the growth of $g_{x \theta}$, the uniform moment bounds for $X_t$ and $\theta_t$, and Lemma \ref{L:Integrating-Factor-first-der}, we obtain
\begin{equation*}
\begin{aligned}
\BE \left[ \left( \int_r^t  \alpha_u \thinspace \eta_{t,u}^* \thinspace {g}_{x \theta}(X_u,\theta_u)D_r X_u \thinspace du \right)^{2p} \right] & \le \int_r^t \cdots \int_r^t \left\{ \prod_{i=1}^{2p} \alpha_{u_i}e^{-C^*(u_i-r)} \right\} \left\{ \prod_{i=1}^{2p}\left(\frac{u_i}{t} \right)^{C_{\bar{g}}C_\alpha} \right\} du_1 \cdots du_{2p} \\
& = \frac{K}{t^{2pC_{\bar{g}} C_\alpha}} \left( \int_r^t u^{C_{\bar{g}} C_\alpha-1}e^{-C^*(u-r)} \thinspace du \right)^{2p}.
\end{aligned}
\end{equation*}
For $C_{\bar{g}} C_\alpha > \frac{1}{2}$ and $t \ge r \ge 1,$ Lemma \ref{L:Integration-f-g-intermediate} (with $\mathsf{D}=C_{\bar{g}} C_\alpha -1$ and $r_1=r_2=r$) yields
$\BE \left[ \left( \int_r^t  \alpha_u \thinspace \eta_{t,u}^* \thinspace {g}_{x \theta}(X_u,\theta_u)D_r X_u \thinspace du \right)^{2p} \right] \le K \frac{r^{2pC_{\bar{g}} C_\alpha-2p}}{t^{2pC_{\bar{g}} C_\alpha}}.$
\end{proof}

Before proving Lemma \ref{L:L-8-deri-1}, we state and prove an auxiliary result that will also be used in later sections. It is important to note that the constants $K$ and $K_{g_{\theta \theta}}$ appearing in this lemma may differ.

\begin{lemma}\label{L:L1}
Let $\eta_{t,r}^*$ be defined in \eqref{E:Integrating-factor-1-der} and $p \in \BN.$ Then, for any $t \ge r$, there exist time-independent positive constants $K$ and $K_{g_{\theta \theta}}$ (associated with the growth of $g_{\theta \theta}$) such that
$$\BE \left[\left( \eta_{t,r}^* \right)^{-p} \right] \le K \left(\frac{t}{r} \right)^{pC_{\alpha} K_{g_{\theta \theta}}}.$$
\end{lemma}
\begin{remark}
For clarity, we note that in this lemma, when \( t > t^* \) (\( t^* \) being sufficiently large and specified in Lemma \ref{L:moment-bound} and Remark \ref{R:Rem-t*-choice}), we write \( K_{g_{\theta \theta}}^* \) instead of \( K_{g_{\theta \theta}} \). This distinction is important in the proof of the technical Lemma \ref{L:Second-der-Gamma-f-3}.
\end{remark}
\begin{proof}[Proof of Lemma \ref{L:L1}]
Recalling the definition of $\eta_{t,r}^*$ in \eqref{E:Integrating-factor-1-der}, taking expectations, and applying H\"older's inequality, we obtain
\begin{equation}\label{E:Eq-1}
\begin{aligned}
\BE\left[ \left( \eta_{t,r}^* \right)^{-p} \right] &= \BE \left[  e^{ p \int_r^t \alpha_u {g}_{\theta \theta}(X_u,\theta_u) \thinspace du -p \int_r^t \alpha_u f_{\theta \theta}(X_u, \theta_u)\thinspace dW_u +\frac{p}{2} \int_r^t \alpha_u^2 f_{\theta \theta}(X_u, \theta_u)^2 du} \right] \\
& \le \left[ \BE e^{ 3p \int_r^{t} \alpha_s {g}_{\theta \theta}(X_s,\theta_s) \thinspace ds} \right]^{\frac{1}{3}} \left[ \BE e^{- 3 p \int_r^{t} \alpha_s f_{\theta \theta}(X_s, \theta_s)\thinspace dW_s} \right]^{\frac{1}{3}} \left[ \BE e^{1.5p  \int_r^{t} \alpha_s^2 f_{\theta \theta}(X_s, \theta_s)^2 ds} \right]^{\frac{1}{3}}.
\end{aligned}
\end{equation}
The terms
$\left[ \BE e^{- 3 p \int_r^{t} \alpha_s f_{\theta \theta}(X_s, \theta_s)\thinspace dW_s} \right]^{\frac{1}{3}}$
and
$\left[ \BE e^{1.5p  \int_r^{t} \alpha_s^2 f_{\theta \theta}(X_s, \theta_s)^2 ds} \right]^{\frac{1}{3}}$
can be treated using the same arguments as in Lemmas \ref{L:L-4-der-1} and \ref{L:L-3-deri-1}, yielding
\begin{equation}\label{E:Eq-2}
\left[ \BE e^{- 3 p \int_r^{t} \alpha_s f_{\theta \theta}(X_s, \theta_s)\thinspace dW_s} \right]^{\frac{1}{3}} \left[ \BE e^{1.5p  \int_r^{t} \alpha_s^2 f_{\theta \theta}(X_s, \theta_s)^2 ds} \right]^{\frac{1}{3}} \le K.
\end{equation}
It remains to bound $\left[ \BE e^{ 3p \int_r^{t} \alpha_s {g}_{\theta \theta}(X_s,\theta_s) \thinspace ds} \right]^{\frac{1}{3}}$. Using the growth condition $|g_{\theta \theta}(x, \theta)| \le K(1+ |x|^q + |\theta|^2 )$, $q \in \BN,$ together with a Taylor expansion, H\"older's inequality, and the uniform moment bounds for $X_t$ and $\theta_t$, we obtain
\begin{equation}\label{E:Eq-3}
\begin{aligned}
\BE e^{ 3p \int_r^{t} \alpha_s {g}_{\theta \theta}(X_s,\theta_s) \thinspace ds}
& \le  \sum_{k=0}^\infty \frac{(3p)^k}{k!} \left[ \int_r^t \cdots \int_r^t \prod_{i=1}^k \alpha_{s_i} \prod_{i=1}^k  \left\{ \BE g_{\theta \theta} (X_{s_i}, \theta_{s_i})^k \right\}^{\frac{1}{k}}  ds_1 \cdots ds_k  \right] \\
& \le  \sum_{k=0}^\infty \frac{1}{k!}\left( 3p K_{g_{\theta \theta}} \int_r^t \alpha_s \thinspace  ds \right)^k
=  e^{3p K_{g_{\theta \theta}}\left( \int_r^t \alpha_s \thinspace ds \right)}
=  \left(\frac{t}{r} \right)^{(3p)C_{\alpha} K_{g_{\theta \theta}}},
\end{aligned}
\end{equation}
where $\sup_{s \ge r} \left[\BE |g_{\theta \theta} (X_{s}, \theta_{s})|^p \right]^{\frac{1}{p}} \le  K_{g_{\theta \theta}},$ $p \in \BN.$ Combining \eqref{E:Eq-1}, \eqref{E:Eq-2}, and \eqref{E:Eq-3} yields the claim.
\end{proof}

\begin{proof}[Proof of Lemma \ref{L:L-8-deri-1}]
For $r \le u \le t,$ we write $\eta_{t,u}^* = \eta_{t,r}^* \left({\eta_{u,r}^*}\right)^{-1}$ and then apply H\"older's inequality, Lemma \ref{L:Integrating-Factor-first-der}, and the martingale moment inequality to obtain
\begin{equation}\label{E:Eq-4}
\begin{aligned}
& \BE \left[ \left( \int_r^t \left[\alpha_u \eta_{t,u}^* {f}_{x \theta}(X_u,\theta_u)D_r X_u \right]dW_u  \right)^{2p} \right] \\
& \qquad \qquad \qquad \qquad \qquad    = \BE \left[\left(\eta_{t,r}^*\right)^{2p} \left( \int_r^t \left[\alpha_u \left({\eta_{u,r}^*}\right)^{-1} {f}_{x \theta}(X_u,\theta_u)D_r X_u \right]dW_u  \right)^{2p} \right]\\
& \qquad \qquad \qquad \qquad \qquad    \le \left[ \BE \left(\eta_{t,r}^*\right)^{4p} \right]^{\frac{1}{2}} \left[ \BE \left( \int_r^t \left[\alpha_u \left({\eta_{u,r}^*}\right)^{-1} {f}_{x \theta}(X_u,\theta_u)D_r X_u \right]dW_u\right)^{4p} \right]^{\frac{1}{2}} \\
& \qquad \qquad \qquad \qquad \qquad   \le K \frac{r^{2pC_{\bar{g}} C_\alpha}}{t^{2pC_{\bar{g}} C_\alpha}} \left[ \BE \left( \int_r^t \left[\alpha_u^2 \left({\eta_{u,r}^*}\right)^{-2} {f}_{x \theta}(X_u,\theta_u)^2 (D_r X_u)^2 \right] du \right)^{2p} \right]^{\frac{1}{2}}.
\end{aligned}
\end{equation}
We now estimate $\BE \left( \int_r^t \left[\alpha_u^2 \left({\eta_{u,r}^*}\right)^{-2} {f}_{x \theta}(X_u,\theta_u)^2 (D_r X_u)^2 \right] du \right)^{2p}$. By H\"older's inequality, Assumption \ref{A:Growth-f-g} for the growth of $f_{x \theta}$, the uniform moment bounds for $X_t$ and $\theta_t$, and Lemma \ref{L:Mal-der-X}, we obtain
\begin{equation*}
\begin{aligned}
& \BE \left( \int_r^t \left[\alpha_u^2 {\eta_{u,r}^*}^{-2} {f}_{x \theta}(X_u,\theta_u)^2 (D_r X_u)^2 \right] du \right)^{2p}\\
 & \le K \underbrace{\int_r^t \cdots \int_r^t}_{2p \text{ times}} \left\{ \prod_{i=1}^{2p} \alpha_{u_i}^2 e^{-2 C^*(u_i-r)} \right\} \BE \left[ \prod_{i=1}^{2p} (\eta_{u_i,r}^*)^{-2} \thinspace {f}_{x \theta}(X_{u_i},\theta_{u_i})^{2} \right]du_1 \cdots du_{2p}  \\
& \le K \int_r^t \cdots \int_r^t \left\{ \prod_{i=1}^{2p} \alpha_{u_i}^2 e^{-2 C^*(u_i-r)} \right\} \left[ \prod_{i=1}^{2p} \left( \BE \left[ {\eta_{u_i,r}^*}^{-8p} \right] \right)^{\frac{1}{4p}} \Big( \BE \left[{f}_{x \theta}(X_{u_i},\theta_{u_i})^{8p} \right] \Big)^{\frac{1}{4p}} \right] du_1 \cdots du_{2p} \\
& \le K \int_r^t \cdots \int_r^t \left\{ \prod_{i=1}^{2p} \alpha_{u_i}^2 e^{-2 C^*(u_i-r)} \right\} \left[ \prod_{i=1}^{2p} \left( \BE \left[ {\eta_{u_i,r}^*}^{-8p} \right] \right)^{\frac{1}{4p}}\right] du_1 \cdots du_{2p}.
 \end{aligned}
\end{equation*}
Applying Lemma \ref{L:L1} with $\tilde{K} \triangleq p C_\alpha K_{g_{\theta \theta}}$ and Lemma  \ref{L:Integration-f-g-intermediate} with $\mathsf{D}=\tilde{K}-2$, we obtain
\begin{equation}\label{E:Eq-5}
\begin{aligned}
& \BE \left( \int_r^t \left[\alpha_u^2 {\eta_{u,r}^*}^{-2} {f}_{x \theta}(X_u,\theta_u)^2 (D_r X_u)^2 \right] du \right)^{2p} \\
& \quad     \le K \int_r^t \cdots \int_r^t \left\{ \prod_{i=1}^{2p} \alpha_{u_i}^2 e^{-2 C^*(u_i-r)} \left(\frac{u_i}{r} \right)^{\tilde{K}} \right\}du_1 \cdots du_{2p}   \le K \left( \int_r^t \alpha_u^2 e^{-2C^*(u-r)} \frac{u^{\tilde{K}}}{r^{\tilde{K}}} \thinspace du \right)^{2p} \\
& \qquad \qquad \qquad \qquad \qquad \qquad \qquad  \le \frac{K}{r^{2p{\tilde{K}}}}\left( \int_r^t u^{{\tilde{K}}-2} e^{-2C^*(u-r)} \thinspace du \right)^{2p}  \le K \frac{r^{2p{\tilde{K}}-4p}}{r^{2p{\tilde{K}}}}= \frac{K}{r^{4p}}.
\end{aligned}
\end{equation}
Combining \eqref{E:Eq-4} and \eqref{E:Eq-5} yields the desired estimate.
\end{proof}

\section{Second-order Malliavin Derivatives}\label{S:Second-order-derivative}
In this section, we establish Proposition \ref{P:Second-der-product}, which provides an upper bound for
$$\int_{[t^*,t]^4} \left( \BE |D^2_{u,r} \theta_t|^4\right)^{\frac{1}{4}} \left( \BE |D^2_{u,s} \theta_t|^4\right)^{\frac{1}{4}} \left( \BE |D^2_{w,r} \theta_t|^4\right)^{\frac{1}{4}} \left( \BE |D^2_{w,s} \theta_t|^4\right)^{\frac{1}{4}} \thinspace du \thinspace  ds \thinspace dw \thinspace dr$$
appearing in \eqref{E:Main-Eq-Holder} in the proof of Theorem \ref{T:Main-theorem}. Compared with Proposition \ref{P:First-der-product}, the proof of Proposition \ref{P:Second-der-product} is considerably more involved, requiring a substantially finer analysis of the second-order Malliavin derivative. Therefore, our strategy is to divide the proof into a sequence of auxiliary lemmas, each addressing a particular component of the argument. For a concise overview of the proof, we refer to Figure \ref{F:Fig-Proof-Dec}. The principal ingredients are Lemmas \ref{L:2-der-moments}--\ref{L:I-8-Second-der}. In particular, Lemma \ref{L:2-der-moments} establishes moment estimates for the second-order Malliavin derivative $\BE[( D^2_{r_1, r_2} \theta_t)^{2p}]$, while Lemmas \ref{L:I-1-Second-Int}--\ref{L:I-8-Second-der} provide estimates for the various integral terms. These terms must be treated separately because their analysis depends on the relative ordering of the four different parameters appearing in the integral in Proposition \ref{P:Second-der-product}. To improve the readability of the main argument, the proofs of these auxiliary lemmas are deferred to Section
\ref{S:Appendix}.

We next briefly describe the proof of Lemma \ref{L:2-der-moments}, which forms the technical core of this section. Its proof is distributed over Sections \ref{S:Initial-Condition}--\ref{S:Term-f} and relies on three groups of auxiliary lemmas. The first group, Lemmas \ref{L:Mal-der-X}--\ref{L:Sec-der-initial-condition}, establishes estimates for the first- and second-order Malliavin derivatives of the state process $X$ together with the contribution from the initial condition term $\BE[ (\eta^*_{t, r_1 \vee r_2})^{2p}\gamma(X_{r_1}, X_{r_2}, \theta_{r_1}, \theta_{r_2})^{2p}]$. The second and third groups, namely Lemmas \ref{L:Second-der-Gamma-g}--\ref{L:Second-der-Gamma-g-5-6} and Lemmas \ref{L:Second-der-Gamma-f}--\ref{L:Second-der-Gamma-f-3}, are devoted to estimating the terms $\BE[( \int_{r_1 \vee r_2}^t \alpha_u \eta^*_{t,u} \Gamma^g (X_u, \theta_u) \thinspace du)^{2p}]$ and $\BE[( \int_{r_1 \vee r_2}^t \alpha_u \eta^*_{t,u} \Gamma^f (X_u, \theta_u)\thinspace dW_u )^{2p}]$, which are associated with the functions $g$ and $f$, respectively, where the processes $\Gamma^g$ and $\Gamma^f$ are defined in Equation \eqref{E:second-order-Mal-Der-theta}. A common feature of these lemmas is that the quantities of interest decay at the rate $\frac{(r_1 \vee r_2)^{ 2p C_{\bar{g}} C_\alpha-2p}}{t^{ 2p C_{\bar{g}} C_\alpha}}e^{- 2p  C^*(r_1 \vee r_2 - r_1 \wedge r_2)} + \frac{1}{t^{ 2p  C_{\bar{g}} C_\alpha}}\frac{(r_1 r_2)^{ 2p  C_{\bar{g}} C_\alpha-2p}}{(r_1 \vee r_2 )^{ 2p  C_{\bar{g}} C_\alpha}}$ as $t \to \infty,$ which is crucial for proving Proposition \ref{P:Second-der-product}. Although the overall strategy parallels that of the first-order analysis in Section \ref{S:First-order-derivative}, the second-order setting requires considerably more delicate estimates because of the presence of stochastic integrals whose moments must be controlled more carefully. The above discussion is summarized in the following figure.
\begin{figure}[!htbp]
\begin{equation*}
\begin{tikzpicture}[
  level 1/.style={sibling distance=45mm},
  level 2/.style={sibling distance=37mm},
  every node/.style={draw, rounded corners, align=center}
]
\node {Proposition \ref{P:Second-der-product}}
  child {node {Lemma \ref{L:2-der-moments}}
    child {node {Lemma \ref{L:Sec-der-initial-condition}}
    }
    child {node {Lemma \ref{L:Second-der-Gamma-g}}
    child {node {Lemmas \ref{L:Second-der-Gamma-g-1-2}--\ref{L:Second-der-Gamma-g-5-6}}
    }}
    child {node {Lemma \ref{L:Second-der-Gamma-f}}
    child {node {Lemmas \ref{L:Second-der-Gamma-f-1-2}--\ref{L:Second-der-Gamma-f-3}}
    }}
    child {node {Lemma \ref{L:Second-der-Gamma-ff}}}
  }
  child {node {Lemmas \ref{L:I-1-Second-Int}--\ref{L:I-8-Second-der}}
  };
\end{tikzpicture}
\end{equation*}
\caption{Outline of the proof of Proposition \ref{P:Second-der-product}.}\label{F:Fig-Proof-Dec}
\end{figure}


\medskip

The second-order Malliavin derivatives of the processes $X_t$ and $\theta_t$, which satisfy Equations \eqref{E:Process-X} and \eqref{E:Process-theta}, respectively, are given by
\begin{equation}\label{E:Second-der-Malli-X}
D_{r_1, r_2}^2 X_t =  \int_{r_1 \vee r_2}^t f_x^*(X_s)(D_{r_1,r_2}^2 X_s) \thinspace ds + \int_{r_1 \vee r_2}^t f_{xx}^*(X_s)(D_{r_1}X_s)(D_{r_2}X_s) \thinspace ds, \quad \text{and}
\end{equation}
\begin{multline}\label{E:second-order-Mal-Der-theta}
\begin{aligned}
D^2_{r_1, r_2} \theta_t & = \gamma(X_{r_1}, X_{r_2}, \theta_{r_1}, \theta_{r_2})  - \int_{r_1 \vee r_2}^t \alpha_u  \bar{g}_{\theta \theta}(\theta_u) D^2_{r_1, r_2} \theta_u \thinspace du \\
&  +  \int_{r_1 \vee r_2}^t \alpha_u  \left[ \bar{g}_{\theta \theta}(\theta_u)- g_{\theta \theta}(X_u, \theta_u) \right] D^2_{r_1, r_2} \theta_u  \thinspace du   +  \int_{r_1 \vee r_2}^t \alpha_u f_{\theta \theta}(X_u, \theta_u) D^2_{r_1, r_2} \theta_u  \thinspace dW_u \\
& + \int_{r_1 \vee r_2}^t \alpha_u \Gamma^g (X_u, \theta_u) \thinspace du + \int_{r_1 \vee r_2}^t \alpha_u \Gamma^f (X_u, \theta_u) \thinspace dW_u\\
& - \int_{r_1 \vee r_2}^t \alpha_u^2 f_{\theta \theta}(X_u, \theta_u) \Gamma^f (X_u, \theta_u) \thinspace du,
\end{aligned}
\end{multline}
where,
\begin{multline*}
\begin{aligned}
\gamma(X_{r_1}, X_{r_2}, \theta_{r_1}, \theta_{r_2}) & \triangleq \alpha_{r_2} f_{x\theta}(X_{r_2},\theta_{r_2})D_{r_1}X_{r_2} + \alpha_{r_1} f_{x\theta}(X_{r_1},\theta_{r_1})D_{r_2}X_{r_1}\\
& \qquad \qquad \qquad \qquad \qquad + \alpha_{r_2} f_{\theta \theta}(X_{r_2},\theta_{r_2})D_{r_1}\theta_{r_2} + \alpha_{r_1} f_{\theta \theta}(X_{r_1},\theta_{r_1})D_{r_2}\theta_{r_1}, \\
\Gamma^g (X_u, \theta_u) & \triangleq - \bar{g}_{\theta \theta \theta}(\theta_u) D_{r_1} \theta_u \cdot D_{r_2} \theta_u + \left[ \bar{g}_{\theta \theta \theta}(\theta_u) - g_{\theta \theta \theta}(X_u, \theta_u) \right]D_{r_1} \theta_u \cdot D_{r_2} \theta_u \\
& \qquad \qquad \quad  - g_{\theta \theta x}(X_u, \theta_u) D_{r_1} \theta_u \cdot D_{r_2} X_u  - g_{\theta x \theta}(X_u, \theta_u) D_{r_1} X_u \cdot D_{r_2} \theta_u  \\
& \qquad \qquad \qquad \qquad   - g_{xx \theta}(X_u, \theta_u) D_{r_1} X_u \cdot D_{r_2} X_u  - g_{x \theta}(X_u, \theta_u) \cdot D^2_{r_1,r_2}X_u, \\
\Gamma^f (X_u, \theta_u) & \triangleq f_{x \theta \theta }(X_u, \theta_u) D_{r_1} \theta_u \cdot D_{r_2} X_u + f_{\theta x \theta}(X_u, \theta_u) D_{r_1} X_u \cdot D_{r_2} \theta_u   \\
& \qquad \qquad \qquad  + f_{\theta \theta \theta }(X_u, \theta_u) D_{r_1} \theta_u \cdot D_{r_2} \theta_u  + f_{xx \theta}(X_u, \theta_u) D_{r_1} X_u \cdot D_{r_2} X_u \\
& \qquad \qquad \qquad  + f_{x \theta}(X_u, \theta_u) \cdot D^2_{r_1,r_2}X_u.
\end{aligned}
\end{multline*}
Recalling the definition of the process $\eta^*$ from Equation \eqref{E:Integrating-factor-1-der}, we solve Equation \eqref{E:second-order-Mal-Der-theta} to obtain
\begin{equation}\label{E:Second-order-der-Solution}
\begin{aligned}
D^2_{r_1, r_2} \theta_t & = \eta^*_{t, r_1 \vee r_2} \gamma(X_{r_1}, X_{r_2}, \theta_{r_1}, \theta_{r_2}) + \int_{r_1 \vee r_2}^t \alpha_u \eta^*_{t,u} \Gamma^g (X_u, \theta_u) \thinspace du \\
& \qquad \qquad \quad +  \int_{r_1 \vee r_2}^t \alpha_u \eta^*_{t,u} \Gamma^f (X_u, \theta_u) \thinspace dW_u - \int_{r_1 \vee r_2}^t \alpha_u^2 \eta^*_{t,u} f_{\theta \theta}(X_u,\theta_u) \Gamma^f (X_u, \theta_u) \thinspace du.
\end{aligned}
\end{equation}
Raising both sides of Equation \eqref{E:Second-order-der-Solution} to the power \(2p\) and then taking expectations yields
\begin{multline*}
\BE \left[ \left(D^2_{r_1, r_2} \theta_t \right)^{2p} \right]  \le K \BE \left[ (\eta^*_{t, r_1 \vee r_2})^{2p}\gamma(X_{r_1}, X_{r_2}, \theta_{r_1}, \theta_{r_2})^{2p} \right] + K \BE \left[ \left( \int_{r_1 \vee r_2}^t \alpha_u \eta^*_{t,u} \Gamma^g (X_u, \theta_u) \thinspace du \right)^{2p} \right] \\
+ K \BE \left[ \left( \int_{r_1 \vee r_2}^t \alpha_u \eta^*_{t,u} \Gamma^f (X_u, \theta_u) \thinspace dW_u \right)^{2p} \right] + K \BE \left[ \left( \int_{r_1 \vee r_2}^t \alpha_u^2 \eta^*_{t,u} f_{\theta \theta}(X_u,\theta_u) \Gamma^f (X_u, \theta_u) \thinspace du \right)^{2p} \right].
\end{multline*}

\begin{proposition}\label{P:Second-der-product}
Let $D^2_{r_1, r_2} \theta_t$ be the solution of Equation \eqref{E:second-order-Mal-Der-theta} defined in Equation \eqref{E:Second-order-der-Solution}. We assume that Assumptions \ref{A:f*-growth} through \ref{A:Tech-Cond} are satisfied. Then, for any $t > r_1 \vee r_2 \ge t^*,$ there exists a time-independent positive constant $K$ such that
\begin{multline*}
\int_{[t^*,t]^4} \left( \BE |D^2_{u,r} \theta_t|^4\right)^{\frac{1}{4}} \left( \BE |D^2_{u,s} \theta_t|^4\right)^{\frac{1}{4}} \left( \BE |D^2_{w,r} \theta_t|^4\right)^{\frac{1}{4}} \left( \BE |D^2_{w,s} \theta_t|^4\right)^{\frac{1}{4}} \thinspace du \thinspace  ds \thinspace dw \thinspace dr  \\
\le \begin{cases} \frac{K}{t^3},  & {C_{\bar{g}} C_\alpha > \frac{5}{4}}\\  \frac{K \log t + K}{t^3}, & {\frac{3}{4} \le C_{\bar{g}} C_\alpha \le \frac{5}{4}} \\ \frac{K}{t^{4C_{\bar{g}} C_\alpha}}, & \frac{1}{2} < C_{\bar{g}} C_\alpha < \frac{3}{4},\end{cases}
\end{multline*}
where $t^*$ is a sufficiently large number and its choice is specified in Lemma \ref{L:moment-bound} and Remark \ref{R:Rem-t*-choice}.
\end{proposition}

Before proving Proposition \ref{P:Second-der-product}, we state the
series of Lemmas \ref{L:2-der-moments} through \ref{L:I-8-Second-der}
below. Lemma \ref{L:2-der-moments} provides an estimate for the term
$\BE [( D^2_{r_1, r_2} \theta_t )^{2p}]$, whereas Lemmas
\ref{L:I-1-Second-Int} through \ref{L:I-8-Second-der} estimate several
integral terms. The resulting bounds involve different threshold values of $C_{\bar{g}} C_\alpha.$ The proofs of these results are given in Section \ref{S:Aux-Results}.

\begin{lemma}\label{L:2-der-moments}
Let $D^2_{r_1, r_2} \theta_t$ be the solution of Equation \eqref{E:second-order-Mal-Der-theta} defined in Equation \eqref{E:Second-order-der-Solution}. We assume that Assumptions \ref{A:f*-growth} through \ref{A:Tech-Cond} are satisfied. Then, for any $t > r_1 \vee r_2 \ge t^*,$ and $p \in \BN,$ there exists a time-independent positive constant $K$ such that
$$\BE \left[ \left( D^2_{r_1, r_2} \theta_t \right)^{2p} \right] \le K \frac{(r_1 \vee r_2)^{ 2p C_{\bar{g}} C_\alpha-2p}}{t^{ 2p C_{\bar{g}} C_\alpha}}e^{- 2p  C^*(r_1 \vee r_2 - r_1 \wedge r_2)} + \frac{K}{t^{ 2p  C_{\bar{g}} C_\alpha}}\frac{(r_1 r_2)^{ 2p  C_{\bar{g}} C_\alpha-2p}}{(r_1 \vee r_2 )^{ 2p  C_{\bar{g}} C_\alpha}},$$
where $t^*$ is a sufficiently large number and its choice is specified in Lemma \ref{L:moment-bound} and Remark \ref{R:Rem-t*-choice}.
\end{lemma}

In the following lemmas, we estimate the terms $\mathscr{I}_1(t; C_{\bar{g}} C_\alpha)$, $\mathscr{I}_2(t; C_{\bar{g}} C_\alpha),$ $\mathscr{I}_4(t; C_{\bar{g}} C_\alpha)$ and $\mathscr{I}_8(t; C_{\bar{g}} C_\alpha)$ defined in Lemmas \ref{L:I-1-Second-Int}, \ref{L:I-2-Second-der}, \ref{L:I-4-Second-Derivative} and \ref{L:I-8-Second-der}, respectively. These terms appear in the proof of Proposition \ref{P:Second-der-product}.

\begin{lemma}\label{L:I-1-Second-Int}
For $t^* \le u \le s \le w \le r \le t$ and $C_{\bar{g}}C_\alpha > \frac{1}{2},$ let
\begin{equation*}
\mathscr{I}_1(t; C_{\bar{g}} C_\alpha)  \triangleq \frac{K}{t^{4C_{\bar{g}} C_\alpha}} \int_{u=t^*}^t \int_{s=u}^t \int_{w=s}^t \int_{r=w}^t  r^{2C_{\bar{g}}C_\alpha-2}s^{C_{\bar{g}}C_\alpha-1} w^{C_{\bar{g}}C_\alpha-1}e^{-2C^*r+ 2C^*u} \thinspace dr \thinspace dw \thinspace ds \thinspace du.
\end{equation*}
Then, there exists a time-independent positive constant $K$ such that
$$\mathscr{I}_1(t; C_{\bar{g}} C_\alpha) \le \begin{cases} \frac{K}{t^3},  & {C_{\bar{g}}C_\alpha > \frac{3}{4}}\\  \frac{K \log t }{t^3}, & {C_{\bar{g}}C_\alpha = \frac{3}{4}} \\ \frac{K}{t^{4C_{\bar{g}}C_\alpha}}, & \frac{1}{2} < C_{\bar{g}}C_\alpha < \frac{3}{4}.\end{cases}$$
\end{lemma}

\begin{lemma}\label{L:I-2-Second-der}
For $t^* \le u \le s \le w \le r \le t$ and $C_{\bar{g}}C_\alpha > \frac{1}{2},$ let
\begin{equation*}
\mathscr{I}_2(t; C_{\bar{g}} C_\alpha)  \triangleq  \frac{K}{t^{4C_{\bar{g}}C_\alpha}} \int_{u=t^*}^t \int_{s=u}^t \int_{w=s}^t \int_{r=w}^t \frac{r^{2C_{\bar{g}} C_\alpha-2}s^{2C_{\bar{g}}C_\alpha-2}}{w}e^{-2C^*r+2C^*u-C^*s+C^*w}\thinspace dr \thinspace dw \thinspace ds \thinspace du.
\end{equation*}
Then, there exists a time-independent positive constant $K$ such that
$$\mathscr{I}_2(t; C_{\bar{g}} C_\alpha) \le \begin{cases} \frac{K}{t^4},  & {C_{\bar{g}}C_\alpha > 1}\\  \frac{K \log t }{t^4}, & {C_{\bar{g}}C_\alpha = 1} \\ \frac{K}{t^{4C_{\bar{g}}C_\alpha}}, & \frac{1}{2} < C_{\bar{g}}C_\alpha < 1.\end{cases}$$
\end{lemma}

\begin{lemma}\label{L:I-4-Second-Derivative}
For $t^* \le u \le s \le w \le r \le t$ and $C_{\bar{g}}C_\alpha > \frac{1}{2},$ let
\begin{equation*}
\begin{aligned}
\mathscr{I}_4(t; C_{\bar{g}} C_\alpha)  & \triangleq  \frac{K}{t^{4C_{\bar{g}} C_\alpha}} \int_{u=t^*}^t \int_{s=u}^t \int_{w=s}^t \int_{r=w}^t r^{C_{\bar{g}}C_\alpha-2}s^{2C_{\bar{g}}C_\alpha-2} w^{C_{\bar{g}}C_\alpha-2} e^{-C^*r+2C^*u-C^*s}\thinspace dr \thinspace dw \thinspace ds \thinspace du.
\end{aligned}
\end{equation*}
Then, there exists a time-independent positive constant $K$ such that
$$\mathscr{I}_4(t; C_{\bar{g}} C_\alpha) \le \begin{cases} \frac{K}{t^5},  & {C_{\bar{g}}C_\alpha > \frac{5}{4}}\\  \frac{K \log t }{t^5}, & {C_{\bar{g}}C_\alpha = \frac{5}{4}} \\ \frac{K}{t^{4C_{\bar{g}}C_\alpha}}, & \frac{1}{2} < C_{\bar{g}}C_\alpha < \frac{5}{4}.\end{cases}$$
\end{lemma}

\begin{lemma}\label{L:I-8-Second-der}
For $t^* \le u \le s \le w \le r \le t$ and $C_{\bar{g}}C_\alpha > \frac{1}{2},$ let
\begin{equation*}
\begin{aligned}
\mathscr{I}_8(t; C_{\bar{g}} C_\alpha)  & \triangleq  \frac{K}{t^{4C_{\bar{g}}C_\alpha}} \int_{u=t^*}^t \int_{s=u}^t \int_{w=s}^t \int_{r=w}^t r^{C_{\bar{g}}C_\alpha-2}u^{C_{\bar{g}}C_\alpha-1} w^{C_{\bar{g}}C_\alpha-2}s^{C_{\bar{g}}C_\alpha-2}e^{-C^*r + C^* u} \thinspace dr \thinspace dw \thinspace ds \thinspace du.
\end{aligned}
\end{equation*}
Then, there exists a time-independent positive constant $K$ such that
$$\mathscr{I}_8(t; C_{\bar{g}} C_\alpha)  \le \begin{cases} \frac{K}{t^6},  & {C_{\bar{g}} C_\alpha > \frac{3}{2}}\\  \frac{K \log t }{t^6}, & {C_{\bar{g}} C_\alpha = \frac{3}{2}} \\ \frac{K}{t^{4C_{\bar{g}} C_\alpha}}, & \frac{1}{2} < C_{\bar{g}}C_\alpha < \frac{3}{2}.\end{cases}$$
\end{lemma}

We are now in a position to prove Proposition \ref{P:Second-der-product}.

\begin{remark}\label{R:Second-der-Int}
It is important to note that we provide the proof of Proposition \ref{P:Second-der-product} only for the case $\mathsf{A}_1 \triangleq \{(u,s,w,r) \in [t^*,t]^4: u \le s \le w \le r \}$ among the total $4!(=24)$ cases. This choice is made because, after performing algebraic manipulation for all the cases, all terms will correspond to the rates specified in Lemmas \ref{L:I-1-Second-Int}, \ref{L:I-2-Second-der}, \ref{L:I-4-Second-Derivative} and \ref{L:I-8-Second-der} above.
\end{remark}

\begin{proof}[Proof of Proposition \ref{P:Second-der-product}]
From Lemma \ref{L:2-der-moments}, for $p=2$, we have
\begin{equation}\label{E:4th-moment-sec-der}
\BE \left[ \left( D^2_{r_1, r_2} \theta_t \right)^{4} \right] \le K \frac{(r_1 \vee r_2)^{4C_{\bar{g}} C_\alpha-4}}{t^{4C_{\bar{g}}C_\alpha}}e^{-4 C^*(r_1 \vee r_2 - r_1 \wedge r_2)} + \frac{K}{t^{4 C_{\bar{g}}C_\alpha}}\frac{(r_1 r_2)^{4 C_{\bar{g}} C_\alpha-4}}{(r_1 \vee r_2 )^{4 C_{\bar{g}} C_\alpha}}.
\end{equation}
Therefore, for the set $\mathsf{A}_1$ defined in Remark \ref{R:Second-der-Int} above and using Equation \eqref{E:4th-moment-sec-der}, we obtain
\begin{align*}
\int_{\mathsf{A}_1} & \left( \BE |D^2_{u,r} \theta_t|^4\right)^{\frac{1}{4}} \left( \BE |D^2_{u,s} \theta_t|^4\right)^{\frac{1}{4}} \left( \BE |D^2_{w,r} \theta_t|^4\right)^{\frac{1}{4}} \left( \BE |D^2_{w,s} \theta_t|^4\right)^{\frac{1}{4}} \thinspace dr \thinspace  dw \thinspace ds \thinspace du \\
& \le  \frac{K}{t^{4C_{\bar{g}}C_\alpha}} \int_{u=t^*}^t \int_{s=u}^t \int_{w=s}^t \int_{r=w}^t \left\{ (u \vee r)^{C_{\bar{g}} C_\alpha-1}e^{- C^*(u \vee r - u \wedge r)} +  \frac{(ur)^{ C_{\bar{g}} C_\alpha-1}}{(u \vee r )^{ C_{\bar{g}} C_\alpha}} \right\} \\
& \qquad \qquad \qquad \qquad \qquad \qquad  \times  \left\{ (u \vee r)^{C_{\bar{g}} C_\alpha-1}e^{- C^*(u \vee r - u \wedge r)} +  \frac{(ur)^{ C_{\bar{g}} C_\alpha-1}}{(u \vee r )^{ C_{\bar{g}} C_\alpha}} \right\} \\
 & \qquad \qquad \qquad \qquad \qquad \qquad  \times  \left\{ (u \vee r)^{C_{\bar{g}} C_\alpha-1}e^{- C^*(u \vee r - u \wedge r)} +  \frac{(ur)^{ C_{\bar{g}} C_\alpha-1}}{(u \vee r )^{ C_{\bar{g}} C_\alpha}} \right\} \\
& \qquad \qquad \qquad \qquad \qquad \qquad  \times  \left\{ (u \vee r)^{C_{\bar{g}} C_\alpha-1}e^{- C^*(u \vee r - u \wedge r)} +  \frac{(ur)^{ C_{\bar{g}} C_\alpha-1}}{(u \vee r )^{ C_{\bar{g}} C_\alpha}} \right\} dr \thinspace dw \thinspace ds \thinspace du.
\end{align*}
Using the ordering of variables in the set $\mathsf{A}_1$, the above expression becomes
\begin{align*}
\int_{\mathsf{A}_1}  \left( \BE |D^2_{u,r} \theta_t|^4\right)^{\frac{1}{4}} & \left( \BE |D^2_{u,s} \theta_t|^4\right)^{\frac{1}{4}} \left( \BE |D^2_{w,r} \theta_t|^4\right)^{\frac{1}{4}} \left( \BE |D^2_{w,s} \theta_t|^4\right)^{\frac{1}{4}} \thinspace dr \thinspace  dw \thinspace ds \thinspace du \\
& =  \frac{K}{t^{4C_{\bar{g}} C_\alpha}} \int_{u=t^*}^t \int_{s=u}^t \int_{w=s}^t \int_{r=w}^t \left\{ r^{C_{\bar{g}} C_\alpha-1}e^{- C^*(r - u)} +  \frac{u^{C_{\bar{g}} C_\alpha-1}}{r} \right\} \\
& \qquad \qquad \qquad \qquad \qquad \qquad  \times  \left\{ s^{C_{\bar{g}} C_\alpha-1}e^{- C^*(s - u)} +  \frac{u^{C_{\bar{g}} C_\alpha-1}}{s} \right\} \\
 & \qquad \qquad \qquad \qquad \qquad \qquad  \times  \left\{ r^{C_{\bar{g}} C_\alpha-1}e^{- C^*(r - w)} +  \frac{w^{C_{\bar{g}}C_\alpha-1}}{r} \right\}\\
& \qquad \qquad \qquad \qquad \qquad \qquad  \times  \left\{ w^{C_{\bar{g}}C_\alpha-1}e^{- C^*(w - s)} +  \frac{s^{C_{\bar{g}}C_\alpha-1}}{w} \right\} \thinspace dr \thinspace dw \thinspace ds \thinspace du.
\end{align*}
Consequently,
\begin{align*}
\int_{\mathsf{A}_1} \left( \BE |D^2_{u,r} \theta_t|^4 \right)^{\frac{1}{4}}
&\left( \BE |D^2_{u,s} \theta_t|^4 \right)^{\frac{1}{4}}
\left( \BE |D^2_{w,r} \theta_t|^4 \right)^{\frac{1}{4}}
\left( \BE |D^2_{w,s} \theta_t|^4 \right)^{\frac{1}{4}}
\thinspace dr \thinspace dw \thinspace ds \thinspace du \\
&= \frac{K}{t^{4C_{\bar{g}} C_\alpha}}
\int_{u=t^*}^t \int_{s=u}^t \int_{w=s}^t \int_{r=w}^t
\left[
r^{2C_{\bar{g}} C_\alpha-2}s^{C_{\bar{g}} C_\alpha-1} w^{C_{\bar{g}}C_\alpha-1}
e^{-2C^*r+ 2C^*u}
\right. \\
&\qquad
+ \frac{r^{2C_{\bar{g}}C_\alpha-2}s^{2C_{\bar{g}}C_\alpha-2}}{w}
e^{-2C^*r+2C^*u-C^*s+C^*w} \\
&\qquad
+ w^{2C_{\bar{g}} C_\alpha-2}s^{C_{\bar{g}}C_\alpha-1}r^{C_{\bar{g}}C_\alpha-2}
e^{-C^* r +2C^* u -C^*w} \\
&\qquad
+ w^{C_{\bar{g}}C_\alpha-2}s^{2C_{\bar{g}}C_\alpha-2}r^{C_{\bar{g}}C_\alpha-2}
e^{-C^* r +2C^* u -C^*s} \\
&\qquad
+ r^{2C_{\bar{g}} C_\alpha-2} u^{C_{\bar{g}} C_\alpha-1}
\frac{w^{C_{\bar{g}} C_\alpha-1}}{s} e^{-2C^*r + C^*s + C^*u} \\
&\qquad
+ r^{2C_{\bar{g}}C_\alpha-2} u^{C_{\bar{g}}C_\alpha-1}
\frac{s^{C_{\bar{g}}C_\alpha-2}}{w} e^{-2C^*r + C^*u + C^*w} \\
&\qquad
+ r^{C_{\bar{g}} C_\alpha-2} w^{2C_{\bar{g}} C_\alpha-2}
\frac{u^{C_{\bar{g}}C_\alpha-1}}{s} e^{-C^*r + C^*u - C^*w + C^*s} \\
&\qquad
+ r^{C_{\bar{g}} C_\alpha-2}u^{C_{\bar{g}} C_\alpha-1}
w^{C_{\bar{g}} C_\alpha-2}s^{C_{\bar{g}}C_\alpha-2}e^{-C^*r + C^* u } \\
&\qquad
+ r^{C_{\bar{g}} C_\alpha-2}u^{C_{\bar{g}} C_\alpha-1}
w^{C_{\bar{g}} C_\alpha-1}s^{C_{\bar{g}}C_\alpha-1}e^{-C^*r + C^* u } \\
\displaybreak[3]
&\qquad
+ u^{C_{\bar{g}} C_\alpha-1} s^{2 C_{\bar{g}} C_\alpha -2 }
\frac{r^{C_{\bar{g}}C_\alpha-2}}{w}e^{-C^*r + C^* w -C^* s + C^* u} \\
&\qquad
+ u^{C_{\bar{g}} C_\alpha-1} s^{ C_{\bar{g}} C_\alpha -1 }
\frac{w^{2C_{\bar{g}}C_\alpha-2}}{r^2}e^{-C^*w + C^* u} \\
&\qquad
+ u^{C_{\bar{g}} C_\alpha-1} s^{2 C_{\bar{g}} C_\alpha -2 }
\frac{w^{C_{\bar{g}} C_\alpha-2}}{r^2}e^{-C^*s + C^* u} \\
&\qquad
+ u^{2C_{\bar{g}} C_\alpha-2} r^{ C_{\bar{g}} C_\alpha -2 }
\frac{w^{C_{\bar{g}} C_\alpha-1}}{s}e^{-C^*r + C^*s} \\
&\qquad
+ u^{2C_{\bar{g}} C_\alpha-2} r^{ C_{\bar{g}} C_\alpha -2 }
\frac{s^{C_{\bar{g}} C_\alpha-2}}{w}e^{-C^*r + C^*w} \\
&\qquad
+ u^{2C_{\bar{g}} C_\alpha-2}
\frac{w^{2C_{\bar{g}} C_\alpha-2}}{r^2 s}e^{-C^* w + C^*s} \\
&\qquad
\left. + u^{2C_{\bar{g}} C_\alpha-2} w^{ C_{\bar{g}} C_\alpha -2 }
\frac{s^{C_{\bar{g}} C_\alpha-2}}{r^2} \right]
dr \thinspace dw \thinspace ds \thinspace du \\
&\triangleq \sum_{i=1}^{16} \mathscr{I}_i(t;C_{\bar{g}}C_\alpha).
\end{align*}
After straightforward algebra, we observe that the bounds for the terms $\mathscr{I}_i(t;C_{\bar{g}}C_\alpha)$, where $i \in \{2,3,5\} \cup\{9, \dots, 16\}$, are identical. Similarly, for $i \in \{4,6,7\}$, the terms $\mathscr{I}_i(t;C_{\bar{g}}C_\alpha)$ share the same bounds. Therefore, by combining the estimates from Lemmas \ref{L:I-1-Second-Int}, \ref{L:I-2-Second-der}, \ref{L:I-4-Second-Derivative} and \ref{L:I-8-Second-der} to control the terms $\mathscr{I}_1(t;C_{\bar{g}}C_\alpha)$, $\mathscr{I}_2(t;C_{\bar{g}}C_\alpha)$, $\mathscr{I}_4(t;C_{\bar{g}}C_\alpha)$ and $\mathscr{I}_8(t;C_{\bar{g}}C_\alpha)$, respectively, we obtain the desired estimate.
\end{proof}

\subsection{Bound associated with the initial condition term: $\BE \left[ (\eta^*_{t, r_1 \vee r_2})^{2p}\gamma(X_{r_1}, X_{r_2}, \theta_{r_1}, \theta_{r_2})^{2p} \right]$}\label{S:Initial-Condition}
In this subsection, we estimate the term $\BE \left[ (\eta^*_{t, r_1 \vee r_2})^{2p}\gamma(X_{r_1}, X_{r_2}, \theta_{r_1}, \theta_{r_2})^{2p} \right]$ (see Lemma \ref{L:Sec-der-initial-condition}), where $\gamma$ is defined in \eqref{E:second-order-Mal-Der-theta}. The bound obtained here is used in the proof of Lemma \ref{L:2-der-moments}. Before stating Lemma \ref{L:Sec-der-initial-condition}, we first provide estimates for the Malliavin derivatives of the process $X$ in Lemma \ref{L:Mal-der-X}. Lemma \ref{L:Mal-der-X} will be used repeatedly throughout this manuscript.

\begin{lemma}\label{L:Mal-der-X}
Let $D_r X_t$ and $D^2_{r_1, r_2}X_t$ be the solutions of Equations \eqref{E:first-order-Mal-Der-X} and \eqref{E:Second-der-Malli-X}, respectively. Then, there exists a positive constant $K$ such that
\begin{equation*}
D_r X_t \le e^{-C^* (t-r)},\quad \text{and} \quad \BE \left[|D^2_{r_1, r_2}X_t | \right] \le K e^{-C^*(t-r_1 \wedge r_2)}.
\end{equation*}
\end{lemma}

\begin{proof}[Proof of Lemma \ref{L:Mal-der-X}]
Solving Equation \eqref{E:first-order-Mal-Der-X} and using Assumption \ref{A:f*-growth} gives $D_rX_t = e^{\int_r^t f_x^*(X_u)\thinspace du} \le e^{-C^* (t-r)}.$
For the second bound, we use H\"older's inequality, Equation \eqref{E:Second-der-Malli-X}, Assumption \ref{A:f*-growth}, the bound for the term $D_r X_t$, the growth of the function $f_{xx}^*$, and moment bounds for the process $X$ to obtain
\begin{equation*}
\begin{aligned}
\BE \left|D_{r_1, r_2}^2 X_t \right| & = \BE \left| \int_{r_1 \vee r_2}^t  e^{\int_s^t f_x^*(X_u)\thinspace du} f_{xx}^*(X_s)\cdot D_{r_1} X_s \cdot D_{r_2} X_s \thinspace ds \right| \\
& \le K \int_{r_1 \vee r_2}^t e^{-C^*(t-s)} e^{-C^*(s-r_1)} e^{-C^*(s-r_2)}ds \le K e^{-C^*(t- r_1 \wedge r_2)}.
\end{aligned}
\end{equation*}
This completes the proof.
\end{proof}

\begin{lemma}\label{L:Sec-der-initial-condition}
Let the process $\eta^*$ and $\gamma$ be defined in Equations \eqref{E:Integrating-factor-1-der} and \eqref{E:second-order-Mal-Der-theta}, respectively. Then, for any $p \in \BN$, $t > r_1 \vee r_2,$ and $C_{\bar{g}} C_\alpha> \frac{1}{2},$ there exists a positive constant $K$ such that
\begin{multline*}
\BE \left[ (\eta^*_{t, r_1 \vee r_2})^{2p}\gamma(X_{r_1}, X_{r_2}, \theta_{r_1}, \theta_{r_2})^{2p} \right] \le  K \frac{(r_1 \vee r_2)^{ 2p C_{\bar{g}} C_\alpha-2p}}{t^{ 2p C_{\bar{g}} C_\alpha}}e^{- 2p  C^*(r_1 \vee r_2 - r_1 \wedge r_2)} \\
 + \frac{K}{t^{ 2p  C_{\bar{g}} C_\alpha}}\frac{(r_1 r_2)^{ 2p  C_{\bar{g}} C_\alpha-2p}}{(r_1 \vee r_2 )^{ 2p  C_{\bar{g}} C_\alpha}}.
\end{multline*}
\end{lemma}

\begin{proof}[Proof of Lemma \ref{L:Sec-der-initial-condition}]
Recalling the definitions of $\eta$ and $\gamma$ from Equations \eqref{E:Integrating-factor-1-der} and \eqref{E:second-order-Mal-Der-theta}, raising their powers to $2p$ and applying H\"older's inequality, we obtain
\begin{equation*}
\begin{aligned}
\BE \left[ (\eta^*_{t, r_1 \vee r_2})^{2p}\gamma(X_{r_1}, X_{r_2}, \theta_{r_1}, \theta_{r_2})^{2p} \right] & \le \sqrt{\BE \left[ \left( \eta^*_{t, r_1 \vee r_2} \right)^{4p}\right]} \sqrt{\BE \left[\gamma(X_{r_1}, X_{r_2}, \theta_{r_1}, \theta_{r_2})^{4p} \right]}.
\end{aligned}
\end{equation*}
We apply Lemma \ref{L:Integrating-Factor-first-der} to bound the term $\BE \left[ \left( \eta^*_{t, r_1 \vee r_2} \right)^{4p}\right]$, and then use H\"older's inequality once more to obtain
\begin{equation*}
\begin{aligned}
\BE \left[ (\eta^*_{t, r_1 \vee r_2})^{2p}\gamma(X_{r_1}, X_{r_2}, \theta_{r_1}, \theta_{r_2})^{2p} \right] & \le K \left( \frac{r_1 \vee r_2}{t}\right)^{ 2p C_{\bar{g}}C_\alpha} \alpha_{r_2}^{2p} \left[ \BE f_{\theta \theta}(X_{r_2}, \theta_{r_2})^{8p} \right]^{\frac{1}{4}} \left[ \BE D_{r_1}\theta_{r_2}^{8p} \right]^{\frac{1}{4}} \ind_{(r_2 \ge r_1)}\\
& +  K \left( \frac{r_1 \vee r_2}{t}\right)^{ 2p C_{\bar{g}}C_\alpha} \alpha_{r_2}^{2p} \left[ \BE f_{x \theta}(X_{r_2}, \theta_{r_2})^{8p} \right]^{\frac{1}{4}} \left[ \BE D_{r_1} X_{r_2}^{8p} \right]^{\frac{1}{4}} \ind_{(r_2 \ge r_1)}\\
& + K \left( \frac{r_1 \vee r_2}{t}\right)^{ 2p C_{\bar{g}}C_\alpha} \alpha_{r_1}^{2p} \left[ \BE f_{x \theta}(X_{r_1}, \theta_{r_1})^{8p} \right]^{\frac{1}{4}} \left[ \BE D_{r_2}X_{r_1}^{8p} \right]^{\frac{1}{4}} \ind_{(r_1 \ge r_2)}\\
& + K \left( \frac{r_1 \vee r_2}{t}\right)^{ 2p C_{\bar{g}}C_\alpha} \alpha_{r_1}^{2p} \left[ \BE f_{\theta \theta}(X_{r_1}, \theta_{r_1})^{8p} \right]^{\frac{1}{4}} \left[ \BE D_{r_2}\theta_{r_1}^{8p} \right]^{\frac{1}{4}} \ind_{(r_1 \ge r_2)}.
\end{aligned}
\end{equation*}
Next, by applying Assumption \ref{A:Growth-f-g} and using uniform moment bounds for the processes $X_t$ and $\theta_t$ \cite{siri_spilio_2020,pardoux2001poisson}, we obtain a constant $K$ such that
$$\left[ \BE f_{\theta \theta}(X_{r_2}, \theta_{r_2})^{8p} \right]^{\frac{1}{4}}+ \left[ \BE f_{x \theta}(X_{r_2}, \theta_{r_2})^{8p} \right]^{\frac{1}{4}} + \left[ \BE f_{x \theta}(X_{r_1}, \theta_{r_1})^{8p} \right]^{\frac{1}{4}} + \left[ \BE f_{\theta \theta}(X_{r_1}, \theta_{r_1})^{8p} \right]^{\frac{1}{4}} \le K. $$
Combining this with Lemmas \ref{L:1-der-moments} and \ref{L:Mal-der-X}, which control the Malliavin derivatives of the processes $\theta_t$ and $X_t$, respectively, yields
\begin{equation*}
\begin{aligned}
\BE \left[ (\eta^*_{t, r_1 \vee r_2})^{2p}\gamma(X_{r_1}, X_{r_2}, \theta_{r_1}, \theta_{r_2})^{2p} \right] & \le K \left( \frac{r_1 \vee r_2}{t}\right)^{ 2p C_{\bar{g}}C_\alpha} \alpha_{r_2}^{2p}  \frac{r_1^{ 2p C_{\bar{g}}C_\alpha-2p}}{r_2^{ 2p C_{\bar{g}}C_\alpha}}\ind_{(r_2 \ge r_1)}\\
& + K \left( \frac{r_1 \vee r_2}{t}\right)^{ 2p C_{\bar{g}}C_\alpha} \alpha_{r_1}^{2p}  \frac{r_2^{ 2p C_{\bar{g}}C_\alpha-2p}}{r_1^{ 2p C_{\bar{g}}C_\alpha}}\ind_{(r_1 \ge r_2)}\\
& + K \left( \frac{r_1 \vee r_2}{t}\right)^{ 2p C_{\bar{g}}C_\alpha} \alpha_{r_2}^{2p} e^{- 2p C^*(r_2-r_1)}\ind_{(r_2 \ge r_1)}\\
& + K \left( \frac{r_1 \vee r_2}{t}\right)^{ 2p C_{\bar{g}}C_\alpha} \alpha_{r_1}^{2p} e^{- 2p C^*(r_1-r_2)}\ind_{(r_1 \ge r_2)}.
\end{aligned}
\end{equation*}
Finally, the definition of the learning rate and straightforward algebra yield the desired bound.
\end{proof}

\subsection{Bound associated with the function $g(x,\theta)$ term: $\BE \left[ \left( \int_{r_1 \vee r_2}^t \alpha_u \eta^*_{t,u} \Gamma^g (X_u, \theta_u) \thinspace du \right)^{2p} \right]$}\label{S:Term-g}
This subsection focuses on estimating the term $\BE [( \int_{r_1 \vee r_2}^t \alpha_u \eta^*_{t,u} \Gamma^g (X_u, \theta_u) \thinspace du )^{2p}]$, as stated in Lemma \ref{L:Second-der-Gamma-g} below. This lemma plays a crucial role in the proof of Lemma \ref{L:2-der-moments}. The proof of the main result of this subsection, namely Lemma \ref{L:Second-der-Gamma-g}, relies on several auxiliary lemmas---Lemmas \ref{L:Second-der-Gamma-g-1-2} through \ref{L:Second-der-Gamma-g-5-6}. Their proofs are based on careful applications of H\"older's inequality, uniform moment bounds for the processes $X$ and $\theta$, and a key use of Lemma \ref{L:Integration-f-g-intermediate}.

\begin{lemma}\label{L:Second-der-Gamma-g}
Let the processes $\eta^*$ and $\Gamma^g$ be defined in Equations \eqref{E:Integrating-factor-1-der} and \eqref{E:second-order-Mal-Der-theta}, respectively.
Then, for any $p \in \BN$ and $t \ge r_1 \vee r_2,$ there exists a positive constant $K$ such that
\begin{multline*}
\BE \left[ \left( \int_{r_1 \vee r_2}^t \alpha_u \eta^*_{t,u} \Gamma^g (X_u, \theta_u) \thinspace du \right)^{2p} \right] \le \frac{K}{t^{ 2p C_{\bar{g}} C_\alpha}}\frac{(r_1 r_2)^{ 2p C_{\bar{g}} C_\alpha-2p}}{(r_1 \vee r_2)^{ 2p C_{\bar{g}} C_\alpha}} \\
 + K \frac{(r_1 \vee r_2)^{ 2p C_{\bar{g}}C_\alpha-2p}}{t^{ 2p C_{\bar{g}} C_\alpha}}e^{- 2p  C^*(r_1 \vee r_2 - r_1 \wedge r_2)}.
\end{multline*}
\end{lemma}

\begin{proof}[Proof of Lemma \ref{L:Second-der-Gamma-g}]
Recalling the definition of the process $\Gamma^g$ from Equation \eqref{E:second-order-Mal-Der-theta} and applying the triangle inequality, we have
\begin{equation*}
\BE \left[ \left( \int_{r_1 \vee r_2}^t \alpha_u \eta^*_{t,u} \Gamma^g (X_u, \theta_u) \thinspace du \right)^{2p} \right] \le K \sum_{i=1}^6 \BE \left[ \left( \int_{r_1 \vee r_2}^t \alpha_u \eta^*_{t,u} \Gamma^g_i (X_u, \theta_u) \thinspace du \right)^{2p} \right],
\end{equation*}
where the processes $\Gamma^g_i,~ i=1,\dots,6$ are defined in Lemmas \ref{L:Second-der-Gamma-g-1-2}, \ref{L:Second-der-Gamma-g-3-4} and \ref{L:Second-der-Gamma-g-5-6} below. We now apply Lemmas \ref{L:Second-der-Gamma-g-1-2}, \ref{L:Second-der-Gamma-g-3-4} and \ref{L:Second-der-Gamma-g-5-6} to control the right-hand side of the above inequality and obtain
\begin{equation*}
\begin{aligned}
\BE \left[ \left( \int_{r_1 \vee r_2}^t \alpha_u \eta^*_{t,u} \Gamma^g (X_u, \theta_u) \thinspace du \right)^{2p} \right] & \le \frac{K}{t^{ 2p C_{\bar{g}}C_\alpha}}\frac{(r_1 r_2)^{ 2p C_{\bar{g}} C_\alpha-2p}}{(r_1 \vee r_2)^{ 2p C_{\bar{g}}C_\alpha}} \\
& + \frac{K}{t^{ 2p C_{\bar{g}}C_\alpha}}\frac{r_1^{ 2p C_{\bar{g}}C_\alpha-2p}}{(r_1 \vee r_2)^{2p}}e^{ 2p  C^*r_2}e^{- 2p  C^* (r_1 \vee r_2)} \\
& + \frac{K}{t^{ 2p C_{\bar{g}} C_\alpha}}\frac{r_2^{ 2p C_{\bar{g}}C_\alpha-2p}}{(r_1 \vee r_2)^{2p}}e^{ 2p  C^*r_1}e^{- 2p  C^* (r_1 \vee r_2)} \\
& + K \frac{(r_1 \vee r_2)^{ 2p C_{\bar{g}}C_\alpha-2p}}{t^{ 2p C_{\bar{g}}C_\alpha}}e^{- 2p  C^*(r_1 \vee r_2 - r_1 \wedge r_2)}.
\end{aligned}
\end{equation*}
A straightforward algebraic simplification yields the required bound.
\end{proof}

We now provide the auxiliary results (Lemmas \ref{L:Second-der-Gamma-g-1-2} through \ref{L:Second-der-Gamma-g-5-6}) used in the proof of Lemma \ref{L:Second-der-Gamma-g}.

\begin{lemma}\label{L:Second-der-Gamma-g-1-2}
Let $D_r \theta_t$ be the solution of Equation \eqref{E:first-order-Mal-Der-theta} defined in Equation \eqref{E:First-der-Solution} and the processes $\Gamma^g_1 (X_u, \theta_u), $ and $ \Gamma^g_2 (X_u, \theta_u)$ are defined as follows:
\begin{equation*}
\begin{aligned}
\Gamma^g_1 (X_u, \theta_u) & \triangleq - \bar{g}_{\theta \theta \theta}(\theta_u) D_{r_1} \theta_u \cdot D_{r_2} \theta_u, \quad
\Gamma^g_2 (X_u, \theta_u)  \triangleq \left[ \bar{g}_{\theta \theta \theta}(\theta_u) - g_{\theta \theta \theta}(X_u, \theta_u) \right]D_{r_1} \theta_u \cdot D_{r_2} \theta_u.
\end{aligned}
\end{equation*}
Then, for any $t > r_1 \vee r_2 \ge 1,$ and $p \in \BN,$ there exists a time-independent constant $K>0$ such that
\begin{equation*}
\BE \left[ \left( \int_{r_1 \vee r_2}^t \alpha_u \eta^*_{t,u} \Gamma^g_1 (X_u, \theta_u) \thinspace du \right)^{2p} \right] + \BE \left[ \left( \int_{r_1 \vee r_2}^t \alpha_u \eta^*_{t,u} \Gamma^g_2 (X_u, \theta_u) \thinspace du \right)^{2p} \right] \le \frac{K}{t^{ 2p C_{\bar{g}}C_\alpha}}\frac{(r_1 r_2)^{ 2p C_{\bar{g}}C_\alpha-2p}}{(r_1 \vee r_2)^{ 2p C_{\bar{g}}C_\alpha}},
\end{equation*}
where the $\eta^*$ process is defined in Equation \eqref{E:Integrating-factor-1-der}.
\end{lemma}

\begin{proof}
We first derive the bound for the term $\BE \left[ \left( \int_{r_1
      \vee r_2}^t \alpha_u \eta^*_{t,u} \Gamma^g_1 (X_u, \theta_u)
    \thinspace du \right)^{2p} \right]$. The corresponding bound for
$\BE \left[ \left( \int_{r_1 \vee r_2}^t \alpha_u \eta^*_{t,u}
    \Gamma^g_2 (X_u, \theta_u) \thinspace du \right)^{2p} \right]$
follows in the same way. Recalling the definition of the $\eta^*$ process from Equation \eqref{E:Integrating-factor-1-der} and applying H\"older's inequality repeatedly, we have
\begin{equation*}
\begin{aligned}
& \BE \left[ \left( \int_{r_1 \vee r_2}^t \alpha_u \eta^*_{t,u} \Gamma^g_1 (X_u, \theta_u) \thinspace du \right)^{2p} \right] \\
& \qquad \qquad \quad \le \underbrace{\int_{r_1 \vee r_2}^t \cdots \int_{r_1 \vee r_2}^t}_{2p-\text{times}} \left\{ \prod_{i=1}^{2p} \alpha_{u_i} \right\}    \BE \left[ \prod_{i=1}^{2p} \eta_{t,u_i}^* \{- \bar{g}_{\theta \theta \theta}(\theta_{u_i})\} D_{r_1} \theta_{u_i} \cdot D_{r_2} \theta_{u_i}\right]du_1 \cdots du_{2p}  \\
& \qquad \qquad \quad \le {\int_{r_1 \vee r_2}^t \cdots \int_{r_1 \vee r_2}^t} \left\{ \prod_{i=1}^{2p} \alpha_{u_i} \right\} \left\{ \BE \left[ \left( \prod_{i=1}^{2p} \eta_{t,u_i}^* \right)^4 \right]\right\}^{\frac{1}{4}} \left\{ \BE \left( \prod_{i=1}^{2p} \{- \bar{g}_{\theta \theta \theta}(\theta_{u_i})\}  \right)^4\right\}^{\frac{1}{4}}  \\
& \qquad \qquad  \qquad \qquad  \qquad \qquad \qquad \qquad  \times  \left\{ \BE \left( \prod_{i=1}^{2p} D_{r_1} \theta_{u_i} \right)^4 \right\}^{\frac{1}{4}}   \left\{ \BE \left( \prod_{i=1}^{2p} D_{r_2} \theta_{u_i} \right)^4 \right\}^{\frac{1}{4}} du_1 \cdots du_{2p}\\
& \qquad \qquad \quad \le {\int_{r_1 \vee r_2}^t \cdots \int_{r_1 \vee r_2}^t} \left\{ \prod_{i=1}^{2p} \alpha_{u_i} \right\} \prod_{i=1}^{2p} \left[ \left\{\BE (\eta_{t,u_i}^*)^{8p} \right\}^{\frac{1}{8p}} \left\{\BE \bar{g}_{\theta \theta \theta}(\theta_{u_i})^{8p}  \right\}^{\frac{1}{8p}}  \times  \right. \\
& \qquad \qquad  \qquad \qquad  \qquad \qquad \qquad \qquad \qquad \qquad \quad   \left.    \left\{\BE (D_{r_1} \theta_{u_i})^{8p} \right\}^{\frac{1}{8p}} \left\{\BE (D_{r_2} \theta_{u_i})^{8p} \right\}^{\frac{1}{8p}}
 \right] du_1 \cdots du_{2p}.
\end{aligned}
\end{equation*}
We now apply Lemmas \ref{L:Integrating-Factor-first-der} and \ref{L:1-der-moments} to bound the terms involving $\eta^*$ and the Malliavin derivatives of the process $\theta_t$, respectively. We then use Assumption \ref{A:Growth-f-g} to handle the term $\BE \bar{g}_{\theta \theta \theta}(\theta_{u_i})^{8p}$, followed by the uniform moment bounds for the process $\theta_t$ \cite{siri_spilio_2020}, to obtain
\begin{equation*}
\begin{aligned}
& \BE \left[ \left( \int_{r_1 \vee r_2}^t \alpha_u \eta^*_{t,u} \Gamma^g_1 (X_u, \theta_u) \thinspace du \right)^{2p} \right]  \\
& \qquad \qquad \quad  \le {\int_{r_1 \vee r_2}^t \cdots \int_{r_1 \vee r_2}^t} \left\{ \prod_{i=1}^{2p} \alpha_{u_i} \left( \frac{u_i}{t} \right)^{C_{\bar{g}}C_\alpha} \right\}  \left\{ \prod_{i=1}^{2p} \frac{r_1^{C_{\bar{g}}C_\alpha-1}}{u_i^{C_{\bar{g}}C_\alpha}} \right\} \left\{ \prod_{i=1}^{2p} \frac{r_2^{C_{\bar{g}}C_\alpha-1}}{u_i^{C_{\bar{g}}C_\alpha}} \right\} du_1 \cdots du_{2p} \\
& \qquad \qquad \quad \le K \frac{(r_1 r_2)^{ 2p C_{\bar{g}}C_\alpha-2p}}{t^{ 2p C_{\bar{g}}C_\alpha}} {\int_{r_1 \vee r_2}^t \cdots \int_{r_1 \vee r_2}^t} \left\{ \prod_{i=1}^{2p} u_i^{C_{\bar{g}} C_\alpha-1} \right\}  \left\{ \prod_{i=1}^{2p} u_i^{-2 C_{\bar{g}} C_\alpha} \right\} du_1 \cdots du_{2p}\\
& \qquad \qquad \quad = K \frac{(r_1 r_2)^{ 2p C_{\bar{g}}C_\alpha-2p}}{t^{ 2p C_{\bar{g}}C_\alpha}} \left( \int_{r_1 \vee r_2}^t \frac{1}{u^{C_{\bar{g}} C_\alpha+1 }} \thinspace du \right)^{2p}  \le \frac{C}{t^{ 2p C_{\bar{g}} C_\alpha}} \frac{(r_1 r_2)^{ 2p C_{\bar{g}}C_\alpha-2p}}{(r_1 \vee r_2)^{ 2p C_{\bar{g}}C_\alpha}}.
\end{aligned}
\end{equation*}
This completes the proof.
\end{proof}

\begin{lemma}\label{L:Second-der-Gamma-g-3-4}
Let $D_r \theta_t$ and $D_r X_t$ be the solutions of Equations \eqref{E:first-order-Mal-Der-theta} and \eqref{E:first-order-Mal-Der-X}, respectively and the processes $\Gamma^g_3 (X_u, \theta_u),$ and $\Gamma^g_4 (X_u, \theta_u)$ are defined as follows:
\begin{equation*}
\begin{aligned}
\Gamma^g_3 (X_u, \theta_u) & \triangleq - g_{\theta \theta x}(X_u, \theta_u) D_{r_1} \theta_u \cdot D_{r_2} X_u, \quad
\Gamma^g_4 (X_u, \theta_u)  \triangleq - g_{\theta x \theta}(X_u, \theta_u) D_{r_1} X_u \cdot D_{r_2} \theta_u.
\end{aligned}
\end{equation*}
Then, for any $t > r_1 \vee r_2 \ge 1,$ and $p \in \BN,$ there exists a time-independent positive constant $K$ such that
\begin{equation*}
\begin{aligned}
& \BE \left[ \left( \int_{r_1 \vee r_2}^t \alpha_u \eta^*_{t,u} \Gamma^g_3 (X_u, \theta_u) \thinspace du \right)^{2p} \right] + \BE \left[ \left( \int_{r_1 \vee r_2}^t \alpha_u \eta^*_{t,u} \Gamma^g_4 (X_u, \theta_u) \thinspace du \right)^{2p} \right] \\
& \qquad \quad \quad   \le \frac{K}{t^{ 2p C_{\bar{g}} C_\alpha}}\frac{r_1^{ 2p C_{\bar{g}}C_\alpha-2p}}{(r_1 \vee r_2)^{2p}}e^{ 2p  C^*r_2}e^{-2p C^* (r_1 \vee r_2)} + \frac{K}{t^{ 2p C_{\bar{g}} C_\alpha}}\frac{r_2^{ 2p C_{\bar{g}} C_\alpha-2p}}{(r_1 \vee r_2)^{2p}}e^{ 2p  C^*r_1}e^{- 2p  C^* (r_1 \vee r_2)}.
\end{aligned}
\end{equation*}
where the $\eta^*$ process is defined in Equation \eqref{E:Integrating-factor-1-der}.
\end{lemma}

\begin{proof}
We first establish the bound for the term $\BE \left[ \left( \int_{r_1
      \vee r_2}^t \alpha_u \eta^*_{t,u} \Gamma^g_3 (X_u, \theta_u)
    \thinspace du \right)^{2p} \right]$. The bound for $\BE \left[
  \left( \int_{r_1 \vee r_2}^t \alpha_u \eta^*_{t,u} \Gamma^g_4 (X_u,
    \theta_u) \thinspace du \right)^{2p} \right]$ follows
similarly. Recalling the definition of the $\eta^*$ process from Equation \eqref{E:Integrating-factor-1-der}, applying H\"older's inequality repeatedly, and using Lemma \ref{L:Mal-der-X} to handle the term $D_rX_t$, we have
\begin{equation*}
\begin{aligned}
& \BE \left[ \left( \int_{r_1 \vee r_2}^t \alpha_u \eta^*_{t,u} \Gamma^g_3 (X_u, \theta_u) \thinspace du \right)^{2p} \right] \\
& \qquad \qquad  \le \underbrace{\int_{r_1 \vee r_2}^t \cdots \int_{r_1 \vee r_2}^t}_{2p-\text{times}} \left\{ \prod_{i=1}^{2p} \alpha_{u_i} e^{-C^*(u_i-r_2)} \right\}    \BE \left[ \prod_{i=1}^{2p} \eta_{t,u_i}^* \{- {g}_{\theta \theta x}(X_{u_i},\theta_{u_i})\} D_{r_1} \theta_{u_i} \right]du_1 \cdots du_{2p}  \\
& \qquad \qquad  \le {\int_{r_1 \vee r_2}^t \cdots \int_{r_1 \vee r_2}^t} \left\{ \prod_{i=1}^{2p} \alpha_{u_i} e^{-C^*(u_i-r_2)} \right\} \prod_{i=1}^{2p} \left[ \left\{\BE (\eta_{t,u_i}^*)^{6p} \right\}^{\frac{1}{6p}} \left\{\BE {g}_{\theta \theta x}(X_{u_i},\theta_{u_i})^{6p}  \right\}^{\frac{1}{6p}}  \times  \right. \\
& \qquad \qquad  \qquad \qquad  \qquad \qquad \qquad \qquad \qquad \qquad \quad \qquad \qquad \qquad \qquad   \left.    \left\{\BE (D_{r_1} \theta_{u_i})^{6p} \right\}^{\frac{1}{6p}}
 \right] du_1 \cdots du_{2p}.
\end{aligned}
\end{equation*}
We next apply Lemmas \ref{L:Integrating-Factor-first-der} and \ref{L:1-der-moments} to deal with the terms $\eta^*$ and the Malliavin derivative of the process $\theta_t,$ respectively, and then use Assumption \ref{A:Growth-f-g} to handle the term $\BE {g}_{\theta \theta x}(X_{u_i},\theta_{u_i})^{6p}$ followed by the uniform moment bounds for the processes $X_t$ and  $\theta_t$ \cite{pardoux2001poisson, siri_spilio_2020} to get
\begin{equation*}
\begin{aligned}
& \BE \left[ \left( \int_{r_1 \vee r_2}^t \alpha_u \eta^*_{t,u} \Gamma^g_3 (X_u, \theta_u) \thinspace du \right)^{2p} \right]  \\
& \qquad \quad   \le K {\int_{r_1 \vee r_2}^t \cdots \int_{r_1 \vee r_2}^t} \left\{ \prod_{i=1}^{2p} \alpha_{u_i} \left( \frac{u_i}{t} \right)^{C_{\bar{g}} C_\alpha} e^{-C^*(u_i-r_2)} \right\}  \left\{ \prod_{i=1}^{2p} \frac{r_1^{C_{\bar{g}} C_\alpha-1}}{u_i^{C_{\bar{g}}C_\alpha}} \right\} du_1 \cdots du_{2p}  \\
& \qquad \quad  \le K \frac{r_1^{ 2p C_{\bar{g}}C_\alpha-2p}}{t^{ 2p C_{\bar{g}}C_\alpha}} {\int_{r_1 \vee r_2}^t \cdots \int_{r_1 \vee r_2}^t} \left\{ \prod_{i=1}^{2p} u_i^{C_{\bar{g}}C_\alpha-1} e^{-C^*(u_i-r_2)} \right\}  \left\{ \prod_{i=1}^{2p} u_i^{- C_{\bar{g}}C_\alpha} \right\} du_1 \cdots du_{2p}\\
& \qquad \quad  = K \frac{r_1^{ 2p C_{\bar{g}}C_\alpha-2p}}{t^{ 2p C_{\bar{g}}C_\alpha}} \left( \int_{r_1 \vee r_2}^t \frac{e^{-C^*(u-r_2)}}{u} \thinspace du \right)^{2p}  \le \frac{K}{t^{ 2p C_{\bar{g}} C_\alpha}} \frac{r_1^{ 2p C_{\bar{g}}C_\alpha-2p}}{(r_1 \vee r_2)^{2p}}e^{ 2p C^*r_2}e^{- 2p C^*(r_1 \vee r_2)},
\end{aligned}
\end{equation*}
where, in the last inequality, we use $\frac{1}{u} \le \frac{1}{r_1 \vee r_2}.$ This completes the proof.
\end{proof}

\begin{lemma}\label{L:Integration-f-g-intermediate}
For any $t \ge r_1 \vee r_2 \ge 1,$ and $\mathsf{D} \in \BR,$ there exists a time-independent, positive constant $K$ such that
$$\int_{r_1 \vee r_2}^t u^{\mathsf{D}} e^{-C^*(u-r_1)} e^{-C^*(u-r_2)}  \thinspace du \le K {(r_1 \vee r_2)^{\mathsf{D}}}e^{- C^*(r_1 \vee r_2 - r_1 \wedge r_2)}.$$
\end{lemma}

\begin{proof}
It suffices to prove the lemma only for the case where $\mathsf{D}>0,$ as the case for $\mathsf{D} \le 0$ is immediate. We first prove the lemma for integer values of $\mathsf{D}.$ Using integration by parts repeatedly, we have
\begin{equation*}
\begin{aligned}
\int_{r_1 \vee r_2}^t u^{\mathsf{D}} e^{-2C^* u} \thinspace du  & \le K (r_1 \vee r_2)^{\mathsf{D}} e^{-2C^* (r_1 \vee r_2)} + K \int_{r_1 \vee r_2}^t u^{\mathsf{D}-1} e^{-2C^* u} \thinspace du \\
& \le K \sum_{i=0}^{\mathsf{D}} (r_1 \vee r_2)^i e^{-2C^* (r_1 \vee r_2)} \le K (r_1 \vee r_2)^{\mathsf{D}} e^{-2C^* (r_1 \vee r_2)}.
\end{aligned}
\end{equation*}
Now, if $\mathsf{D}$ is not an integer, then there exists a $\beta<0$ (for example, if $\mathsf{D}=\frac{1}{2}$, then $\beta = -\frac{1}{2}$) such that (again using integration by parts)
$\int_{r_1 \vee r_2}^t u^{\mathsf{D}} e^{-2C^* u} \thinspace du \le  K (r_1 \vee r_2)^{\mathsf{D}} e^{-2C^* (r_1 \vee r_2)} + \cdots + K \int_{r_1 \vee r_2}^t u^{\beta} e^{2C^* u}\thinspace du.$
In this inequality, noting $\frac{1}{u} \le \frac{1}{r_1 \vee r_2} \le 1$, one can obtain
$\int_{r_1 \vee r_2}^t u^{\mathsf{D}} e^{-2C^* u} \thinspace du  \le K (r_1 \vee r_2)^{\mathsf{D}} e^{-2C^* (r_1 \vee r_2)},$
which completes the proof of the lemma.
\end{proof}

\begin{lemma}\label{L:Second-der-Gamma-g-5-6}
Let the processes $D_r X_t$ and $D^2_{r_1,r_2}X_t$ be the solutions of Equations \eqref{E:first-order-Mal-Der-theta} and \eqref{E:Second-der-Malli-X}, respectively and the processes $\Gamma^g_5 (X_u, \theta_u), \Gamma^g_6 (X_u, \theta_u)$ are defined as follows:
\begin{equation*}
\begin{aligned}
\Gamma^g_5 (X_u, \theta_u) & \triangleq  - g_{xx \theta}(X_u, \theta_u) D_{r_1} X_u \cdot D_{r_2} X_u, \quad
\Gamma^g_6 (X_u, \theta_u)  \triangleq - g_{x \theta}(X_u, \theta_u) \cdot D^2_{r_1,r_2}X_u.
\end{aligned}
\end{equation*}
Then, for any $t > r_1 \vee r_2,$ and $p \in \BN,$ there exists a time-independent positive constant $K$ such that
\begin{multline*}
\BE \left[ \left( \int_{r_1 \vee r_2}^t \alpha_u \eta^*_{t,u} \Gamma^g_5 (X_u, \theta_u) \thinspace du \right)^{2p} \right] + \BE \left[ \left( \int_{r_1 \vee r_2}^t \alpha_u \eta^*_{t,u} \Gamma^g_6 (X_u, \theta_u) \thinspace du \right)^{2p} \right] \\ \le K \frac{(r_1 \vee r_2)^{ 2p C_{\bar{g}} C_\alpha-2p}}{t^{ 2p C_{\bar{g}} C_\alpha}}e^{- 2p  C^*(r_1 \vee r_2 - r_1 \wedge r_2)},
\end{multline*}
where the $\eta^*$ process is defined in Equation \eqref{E:Integrating-factor-1-der}.
\end{lemma}

\begin{proof}
We first derive the bound for the term $\BE \left[ \left( \int_{r_1 \vee r_2}^t \alpha_u \eta^*_{t,u} \Gamma^g_5 (X_u, \theta_u) \thinspace du \right)^{2p} \right]$; the bound for $\BE \left[ \left( \int_{r_1 \vee r_2}^t \alpha_u \eta^*_{t,u} \Gamma^g_6 (X_u, \theta_u) \thinspace du \right)^{2p} \right]$ follows similarly. Recalling the definition of $\eta^*$ process from Equation \eqref{E:Integrating-factor-1-der}, applying H\"older's inequality repeatedly, and using Lemma \ref{L:Mal-der-X} to handle the term $D_rX_t$, we have
\begin{equation*}
\begin{aligned}
& \BE \left[ \left( \int_{r_1 \vee r_2}^t \alpha_u \eta^*_{t,u} \Gamma^g_5 (X_u, \theta_u) \thinspace du \right)^{2p} \right] \\
& \qquad \quad  \le \underbrace{\int_{r_1 \vee r_2}^t \cdots \int_{r_1 \vee r_2}^t}_{2p-\text{times}} \left\{ \prod_{i=1}^{2p} \alpha_{u_i} e^{-C^*(u_i-r_1)}  e^{-C^*(u_i-r_2)} \right\}    \BE \left[ \prod_{i=1}^{2p} \eta_{t,u_i}^* \{- {g}_{x x \theta}(X_{u_i},\theta_{u_i})\}  \right]du_1 \cdots du_{2p}  \\
& \qquad \quad  \le {\int_{r_1 \vee r_2}^t \cdots \int_{r_1 \vee r_2}^t} \left\{ \prod_{i=1}^{2p} \alpha_{u_i} e^{-C^*(u_i-r_1)} e^{-C^*(u_i-r_2)} \right\} \prod_{i=1}^{2p} \left[ \left\{\BE (\eta_{t,u_i}^*)^{4p} \right\}^{\frac{1}{4p}} \times \right.\\
& \qquad \qquad \qquad \qquad \qquad \qquad \qquad \qquad \qquad \qquad \qquad \qquad \qquad \qquad  \left. \left\{\BE {g}_{xx\theta}(X_{u_i},\theta_{u_i})^{4p}  \right\}^{\frac{1}{4p}}  \right] du_1 \cdots du_{2p}.
\end{aligned}
\end{equation*}
We next apply Lemma \ref{L:Integrating-Factor-first-der} to handle the term $\eta^*$, and then use Assumption \ref{A:Growth-f-g} to control $\BE {g}_{x x \theta}(X_{u_i},\theta_{u_i})^{4p}$, followed by the uniform moment bounds for the processes $X_t$ and  $\theta_t$ \cite{pardoux2001poisson, siri_spilio_2020}, to obtain
\begin{equation*}
\begin{aligned}
& \BE \left[ \left( \int_{r_1 \vee r_2}^t \alpha_u \eta^*_{t,u} \Gamma^g_5 (X_u, \theta_u) \thinspace du \right)^{2p} \right]  \\
& \qquad \qquad   \le K {\int_{r_1 \vee r_2}^t \cdots \int_{r_1 \vee r_2}^t} \left\{ \prod_{i=1}^{2p} \alpha_{u_i} \left( \frac{u_i}{t} \right)^{C_{\bar{g}}C_\alpha} e^{-C^*(u_i-r_1)} e^{-C^*(u_i-r_2)} \right\} du_1 \cdots du_{2p}  \\
& \qquad \qquad  =  \frac{K}{t^{ 2p C_{\bar{g}}C_\alpha}} \left( \int_{r_1 \vee r_2}^t u^{C_{\bar{g}}C_\alpha-1} e^{-C^*(u-r_1)} e^{-C^*(u-r_2)}  \thinspace du \right)^{2p} \\
& \qquad \qquad  \le K \frac{(r_1 \vee r_2)^{ 2p C_{\bar{g}}C_\alpha-2p}}{t^{ 2p C_{\bar{g}}C_\alpha}}e^{-2p C^*(r_1 \vee r_2 - r_1 \wedge r_2)},
\end{aligned}
\end{equation*}
where the last inequality follows from Lemma \ref{L:Integration-f-g-intermediate} with $\mathsf{D}=C_{\bar{g}}C_\alpha-1$. Next, using Lemma \ref{L:Mal-der-X} to handle the term $D^2_{r_1, r_2}X_u$, we obtain the same bound for the term $\BE \left[ \left( \int_{r_1 \vee r_2}^t \alpha_u \eta^*_{t,u} \Gamma^g_6 (X_u, \theta_u) \thinspace du \right)^{2p} \right]$, which completes the proof of the lemma.
\end{proof}

\subsection{Bound associated with the function $f(x, \theta)$ terms: $\BE [( \int_{r_1 \vee r_2}^t \alpha_u \eta^*_{t,u} \Gamma^f (X_u, \theta_u)\thinspace dW_u)^{2p}]$ and $\BE [( \int_{r_1 \vee r_2}^t \alpha_u^2 \eta^*_{t,u} f_{\theta \theta}(X_u,\theta_u) \Gamma^f (X_u, \theta_u) \thinspace du)^{2p}]$}\label{S:Term-f}
This section focuses on estimating the following two terms: \\$\BE [( \int_{r_1 \vee r_2}^t \alpha_u \eta^*_{t,u} \Gamma^f (X_u, \theta_u)\thinspace dW_u)^{2p}]$ and $\BE [( \int_{r_1 \vee r_2}^t \alpha_u^2 \eta^*_{t,u} f_{\theta \theta}(X_u,\theta_u) \Gamma^f (X_u, \theta_u) \thinspace du)^{2p}]$, as stated in Lemmas \ref{L:Second-der-Gamma-f} and \ref{L:Second-der-Gamma-ff}, respectively. These results play a crucial role in the proof of Lemma \ref{L:2-der-moments}. The proofs of the main results of this section, Lemmas \ref{L:Second-der-Gamma-f} and \ref{L:Second-der-Gamma-ff}, rely on several auxiliary lemmas---Lemmas \ref{L:Second-der-Gamma-f-1-2} through \ref{L:Second-der-Gamma-f-3}. The proofs of these results are based on careful applications of H\"older's inequality, uniform moment bounds for the processes $X$ and $\theta$, and a key use of Lemmas \ref{L:Integration-f-g-intermediate} and \ref{L:L1}.

\begin{lemma}\label{L:Second-der-Gamma-f}
Let the processes $\eta^*$ and $\Gamma^f$ be defined in Equations \eqref{E:Integrating-factor-1-der} and \eqref{E:second-order-Mal-Der-theta}, respectively. We assume that Assumptions \ref{A:f*-growth} through \ref{A:Tech-Cond} hold. Then, for any $p \in \BN$, and $t \ge r_1 \vee r_2 \ge t^*,$ there exists a positive constant $K$ such that
\begin{multline*}
\BE \left[ \left( \int_{r_1 \vee r_2}^t \alpha_u \eta^*_{t,u} \Gamma^f (X_u, \theta_u) \thinspace dW_u \right)^{2p} \right] \le \frac{K}{t^{ 2p C_{\bar{g}}C_\alpha}}\frac{(r_1 r_2)^{ 2p C_{\bar{g}} C_\alpha-2p}}{(r_1 \vee r_2)^{ 2p C_{\bar{g}}C_\alpha}} \\
 + K \frac{(r_1 \vee r_2)^{ 2p C_{\bar{g}}C_\alpha-2p}}{t^{ 2p C_{\bar{g}}C_\alpha}}e^{- 2p  C^*(r_1 \vee r_2 - r_1 \wedge r_2)},
 \end{multline*}
where $t^*$ is a sufficiently large number and its choice is specified in Lemma \ref{L:moment-bound} and Remark \ref{R:Rem-t*-choice}.
\end{lemma}

\begin{proof}[Proof of Lemma \ref{L:Second-der-Gamma-f}]
The proof is analogous to the proof of Lemma \ref{L:Second-der-Gamma-g} and follows by combining Lemmas \ref{L:Second-der-Gamma-f-1-2}, \ref{L:Second-der-Gamma-f-3} and \ref{L:Second-der-Gamma-f-4-5} stated below. For brevity, we omit the full details.
\end{proof}

\begin{lemma}\label{L:Second-der-Gamma-ff}
Let the processes $\eta^*$ and $\Gamma^f$ be defined in Equations \eqref{E:Integrating-factor-1-der} and \eqref{E:second-order-Mal-Der-theta}, respectively.
Then, for any $p \in \BN$ and $t \ge r_1 \vee r_2,$ there exists a positive constant $K$ such that
\begin{multline*}
\BE \left[ \left( \int_{r_1 \vee r_2}^t \alpha_u^2 \eta^*_{t,u} f_{\theta \theta}(X_u,\theta_u) \Gamma^f (X_u, \theta_u) \thinspace du \right)^{2p} \right] \le \frac{K}{t^{ 2p C_{\bar{g}}C_\alpha}}\frac{(r_1 r_2)^{ 2p C_{\bar{g}} C_\alpha-2p}}{(r_1 \vee r_2)^{ 2p C_{\bar{g}}C_\alpha}} \\
 + K \frac{(r_1 \vee r_2)^{ 2p C_{\bar{g}}C_\alpha-2p}}{t^{ 2p C_{\bar{g}}C_\alpha}}e^{- 2p  C^*(r_1 \vee r_2 - r_1 \wedge r_2)}.
 \end{multline*}
\end{lemma}

\begin{proof}[Proof of Lemma \ref{L:Second-der-Gamma-ff}]
The definition of the process $\Gamma^f$ from Equation \eqref{E:second-order-Mal-Der-theta} and the triangle inequality yield
\begin{equation*}
\BE \left[ \left( \int_{r_1 \vee r_2}^t \alpha_u^2 \eta^*_{t,u} f_{\theta \theta}(X_u,\theta_u) \Gamma^f (X_u, \theta_u) \thinspace du \right)^{2p} \right] \le K \sum_{i=1}^5 \BE \left[ \left( \int_{r_1 \vee r_2}^t \alpha_u^2 \eta^*_{t,u} f_{\theta \theta}(X_u,\theta_u) \Gamma^f_i (X_u, \theta_u) \thinspace du \right)^{2p} \right],
\end{equation*}
where the processes $\Gamma^f_i,~ i=1,\dots,5$ are defined in Lemmas \ref{L:Second-der-Gamma-f-1-2}, \ref{L:Second-der-Gamma-f-4-5}, and \ref{L:Second-der-Gamma-f-3} below. A direct application of Assumption \ref{A:Growth-f-g} controls the growth of the term $f_{\theta \theta}(X_t,\theta_t)$. Proceeding with calculations similar to those in Lemmas \ref{L:Second-der-Gamma-g-1-2}, \ref{L:Second-der-Gamma-g-3-4}, and \ref{L:Second-der-Gamma-g-5-6}, and using the fact that $\alpha_t^2 \le \alpha_t$, yields the desired estimate.
\end{proof}

We now provide the auxiliary results (Lemmas \ref{L:Second-der-Gamma-f-1-2} through \ref{L:Second-der-Gamma-f-3}) used to prove Lemmas \ref{L:Second-der-Gamma-f} and \ref{L:Second-der-Gamma-ff}.

\begin{lemma}\label{L:Second-der-Gamma-f-1-2}
Let $D_r \theta_t$ and $D_r X_t$ be the solutions of Equations \eqref{E:first-order-Mal-Der-theta} and \eqref{E:first-order-Mal-Der-X}, respectively and the processes $\Gamma^f_1 (X_u, \theta_u), \Gamma^f_2 (X_u, \theta_u)$ are defined as follows:
\begin{equation*}
\begin{aligned}
\Gamma^f_1 (X_u, \theta_u) & \triangleq f_{x \theta \theta }(X_u, \theta_u) D_{r_1} \theta_u \cdot D_{r_2} X_u, \quad
\Gamma^f_2 (X_u, \theta_u)  \triangleq f_{\theta x \theta}(X_u, \theta_u) D_{r_1} X_u \cdot D_{r_2} \theta_u.
\end{aligned}
\end{equation*}
Then, for any $t > r_1 \vee r_2,$ and $p \in \BN,$ there exists a time-independent constant $K>0$ such that
\begin{equation*}
\begin{aligned}
& \BE \left[ \left( \int_{r_1 \vee r_2}^t \alpha_u \eta^*_{t,u} \Gamma^f_1 (X_u, \theta_u) \thinspace dW_u \right)^{2p} \right] + \BE \left[ \left( \int_{r_1 \vee r_2}^t \alpha_u \eta^*_{t,u} \Gamma^f_2 (X_u, \theta_u) \thinspace dW_u \right)^{2p} \right] \\
& \quad    \le \frac{K}{t^{ 2p C_{\bar{g}} C_\alpha}}\frac{r_1^{ 2p C_{\bar{g}}C_\alpha-2p}}{(r_1 \vee r_2)^{2p}}e^{ 2p  C^*r_2}e^{- 2p  C^* (r_1 \vee r_2)} + \frac{K}{t^{ 2p C_{\bar{g}}C_\alpha}}\frac{r_2^{ 2p C_{\bar{g}}C_\alpha-2p}}{(r_1 \vee r_2)^{2p}}e^{ 2p  C^*r_1}e^{- 2p  C^* (r_1 \vee r_2)},
\end{aligned}
\end{equation*}
where the $\eta^*$ process is defined in Equation \eqref{E:Integrating-factor-1-der}.
\end{lemma}

\begin{proof}[Proof of Lemma \ref{L:Second-der-Gamma-f-1-2}]
For $r_1 \vee r_2 \le u \le t,$ we write the process $\eta_{t,u}^*$ as $\eta_{t,{r_1 \vee r_2}}^* \left({\eta_{u,{r_1 \vee r_2}}^*}\right)^{-1}$. Applying H\"older's inequality, Lemma \ref{L:Integrating-Factor-first-der}, and the martingale moment inequality \cite[Proposition 3.26]{KS91}, we obtain
\begin{equation}\label{E:Eq-6}
\begin{aligned}
& \BE \left[ \left( \int_{r_1 \vee r_2}^t \alpha_u \eta^*_{t,u} \Gamma^f_1 (X_u, \theta_u) \thinspace dW_u \right)^{2p} \right] =  \BE \left[ \left( \int_{r_1 \vee r_2}^t \left[\alpha_u \eta_{t,u}^* {f}_{x \theta \theta}(X_u,\theta_u)D_{r_1} \theta_u \cdot D_{r_2} X_u \right]dW_u  \right)^{2p} \right] \\
& \qquad    = \BE \left[\left(\eta_{t,{r_1 \vee r_2}}^*\right)^{2p} \left( \int_{r_1 \vee r_2}^t \left[\alpha_u \left({\eta_{u,{r_1 \vee r_2}}^*}\right)^{-1} {f}_{x \theta \theta}(X_u,\theta_u)D_{r_1} \theta_u \cdot D_{r_2} X_u  \right]dW_u  \right)^{2p} \right]\\
& \qquad  \le \left[ \BE \left(\eta_{t,{r_1 \vee r_2}}^*\right)^{4p} \right]^{\frac{1}{2}} \left[ \BE \left( \int_{{r_1 \vee r_2}}^t \left[\alpha_u \left({\eta_{u,{r_1 \vee r_2}}^*}\right)^{-1} {f}_{x \theta \theta}(X_u,\theta_u)D_{r_1} \theta_u \cdot D_{r_2} X_u \right]dW_u\right)^{4p} \right]^{\frac{1}{2}} \\
& \qquad  \le K \frac{(r_1 \vee r_2)^{ 2p C_{\bar{g}} C_\alpha}}{t^{ 2p C_{\bar{g}} C_\alpha}} \left[ \BE \left( \int_{r_1 \vee r_2}^t \left[\alpha_u^2 \left({\eta_{u,r_1 \vee r_2}^*}\right)^{-2} {f}_{x \theta \theta}(X_u,\theta_u)^2 (D_{r_1} \theta_u)^2 \cdot (D_{r_2} X_u)^2 \right] du \right)^{2p} \right]^{\frac{1}{2}}.
\end{aligned}
\end{equation}
We now estimate the term $\BE \left( \int_{r_1 \vee r_2}^t \left[\alpha_u^2 \left({\eta_{u,r_1 \vee r_2}^*}\right)^{-2} {f}_{x \theta \theta}(X_u,\theta_u)^2 (D_{r_1} \theta_u)^2 (D_{r_2} X_u)^2 \right] du \right)^{2p}$ appearing in the right-hand side of \eqref{E:Eq-6}. Using H\"older's inequality, Assumption \ref{A:Growth-f-g} to control the growth of $f_{x \theta \theta}$, the uniform moment bounds for the processes $X_t$ and $\theta_t$, and Lemmas \ref{L:Mal-der-X} and \ref{L:1-der-moments} (to control the Malliavin derivatives $D_r X_t$ and $D_r \theta_t$, respectively), we obtain
\begin{equation*}
\begin{aligned}
& \BE \left[ \left( \int_{r_1 \vee r_2}^t \alpha_u^2 {\eta_{u,{r_1 \vee r_2}}^*}^{-2} {f}_{x \theta \theta}(X_u,\theta_u)^2 (D_{r_1} \theta_u)^2 (D_{r_2} X_u)^2 \thinspace du \right)^{2p} \right]\\
 & \le K \underbrace{\int_{r_1 \vee r_2}^t \cdots \int_{r_1 \vee r_2}^t}_{2p-\text{times}} \left\{ \prod_{i=1}^{2p} \alpha_{u_i}^2 e^{-2 C^*(u_i-r_2)} \right\} \BE \left[ \prod_{i=1}^{2p} (\eta_{u_i,r_1 \vee r_2}^*)^{-2} \thinspace {f}_{x \theta \theta}(X_{u_i},\theta_{u_i})^{2} (D_{r_1} \theta_{u_i})^2 \right]du_1 \cdots du_{2p}  \\
& \le K \int_{r_1 \vee r_2}^t \cdots \int_{r_1 \vee r_2}^t \left\{ \prod_{i=1}^{2p} \alpha_{u_i}^2 e^{-2 C^*(u_i-r_2)} \right\} \left[ \prod_{i=1}^{2p} \left( \BE \left[ {\eta_{u_i,r_1 \vee r_2}^*}^{-12p} \right] \right)^{\frac{1}{6p}} \Big( \BE \left[(D_{r_1} \theta_{u_i})^{12p} \right] \Big)^{\frac{1}{6p}} \right] du_1 \cdots du_{2p}.
 \end{aligned}
\end{equation*}
Next, applying Lemmas \ref{L:L1} and \ref{L:1-der-moments} to handle the terms $\BE \left[ \left({\eta_{u_i,r_1 \vee r_2}^*}\right)^{-12p} \right]$ and $\BE \left[(D_{r_1} \theta_{u_i})^{12p} \right]$, respectively, and then using Lemma \ref{L:Integration-f-g-intermediate} with $\mathsf{D}=2 C_\alpha K_{g_{\theta \theta}} -2C_{\bar{g}}C_\alpha- 2, $ we obtain
\begin{equation}\label{E:Eq-7}
\begin{aligned}
&  \BE \left( \int_{r_1 \vee r_2}^t \left[\alpha_u^2 ({\eta_{u,{r_1 \vee r_2}}^*})^{-2} {f}_{x \theta \theta}(X_u,\theta_u)^2 (D_{r_1} \theta_u)^2 (D_{r_2} X_u)^2 \right] du \right)^{2p}\\
& \qquad \qquad \le K \int_{r_1 \vee r_2}^t \cdots \int_{r_1 \vee r_2}^t \left\{ \prod_{i=1}^{2p} \alpha_{u_i}^2 e^{-2 C^*(u_i-r_2)} \left( \frac{u_i}{r_1 \vee r_2} \right)^{2 C_\alpha K_{g_{\theta \theta}}} \frac{r_1^{2C_{\bar{g}}C_\alpha-2}}{{u_i}^{2C_{\bar{g}}C_\alpha}} \right\} du_1 \cdots du_{2p}\\
& \qquad \qquad = K \frac{r_1^{(4p)C_{\bar{g}}C_\alpha-4p}}{(r_1 \vee r_2)^{(4p)C_\alpha K_{g_{\theta \theta}}}} \left( \int_{r_1 \vee r_2}^t u^{2 C_\alpha K_{g_{\theta \theta}} -2C_{\bar{g}}C_\alpha- 2} e^{-2 C^*(u-r_2)} \thinspace du\right)^{2p} \\
& \qquad \qquad \le K \frac{r_1^{(4p)C_{\bar{g}}C_\alpha-4p}}{(r_1 \vee r_2)^{(4p)C_\alpha K_{g_{\theta \theta}}}} e^{(4p)C^* r_2} e^{-4pC^*(r_1 \vee r_2)} (r_1 \vee r_2)^{2p(2 C_\alpha K_{g_{\theta \theta}} -2C_{\bar{g}}C_\alpha- 2)} \\
& \qquad \qquad = K \frac{r_1^{(4p)C_{\bar{g}}C_\alpha-4p}}{(r_1 \vee r_2)^{(4p)C_{\bar{g}}C_\alpha+4p}}e^{(4p)C^* r_2} e^{-4pC^*(r_1 \vee r_2)}.
\end{aligned}
\end{equation}
Finally, combining \eqref{E:Eq-6} and \eqref{E:Eq-7}, we obtain
\begin{equation*}
\begin{aligned}
\BE \left[ \left( \int_{r_1 \vee r_2}^t \left[\alpha_u \eta_{t,u}^* {f}_{x \theta \theta}(X_u,\theta_u)D_{r_1} \theta_u \cdot D_{r_2} X_u \right]dW_u  \right)^{2p} \right] \le \frac{K}{t^{ 2p C_{\bar{g}} C_\alpha}}\frac{r_1^{ 2p C_{\bar{g}}C_\alpha-2p}}{(r_1 \vee r_2)^{2p}}e^{2p C^*r_2}e^{-2p C^* (r_1 \vee r_2)}.
\end{aligned}
\end{equation*}
The bound for the term $\BE \left[ \left( \int_{r_1 \vee r_2}^t \alpha_u \eta^*_{t,u} \Gamma^f_2 (X_u, \theta_u) \thinspace dW_u \right)^{2p} \right]$ can be obtained similarly. This completes the proof.
\end{proof}
\color{black}

\begin{lemma}\label{L:Second-der-Gamma-f-4-5}
Let the processes $D_r X_t$ and $D^2_{r_1,r_2}X_t$ be the solutions of Equations \eqref{E:first-order-Mal-Der-theta} and \eqref{E:Second-der-Malli-X}, respectively and the processes $\Gamma^f_4 (X_u, \theta_u), \Gamma^f_5 (X_u, \theta_u)$ are defined as follows:
\begin{equation*}
\begin{aligned}
\Gamma^f_4 (X_u, \theta_u) & \triangleq f_{xx \theta}(X_u, \theta_u) D_{r_1} X_u \cdot D_{r_2} X_u, \quad
\Gamma^f_5 (X_u, \theta_u)  \triangleq f_{x \theta}(X_u, \theta_u) \cdot D^2_{r_1,r_2}X_u.
\end{aligned}
\end{equation*}
Then, for any $t > r_1 \vee r_2,$ and $p \in \BN,$ there exists a time-independent constant $K>0$ such that
\begin{equation*}
\begin{aligned}
& \BE \left[ \left( \int_{r_1 \vee r_2}^t \alpha_u \eta^*_{t,u} \Gamma^f_4 (X_u, \theta_u) \thinspace dW_u \right)^{2p} \right] + \BE \left[ \left( \int_{r_1 \vee r_2}^t \alpha_u \eta^*_{t,u} \Gamma^f_5 (X_u, \theta_u) \thinspace dW_u \right)^{2p} \right] \\
& \qquad \qquad \qquad \qquad \qquad \qquad \qquad \qquad \qquad \qquad \qquad \quad  \le  K \frac{(r_1 \vee r_2)^{ 2p C_{\bar{g}} C_\alpha-2p}}{t^{ 2p C_{\bar{g}} C_\alpha}}e^{- 2p  C^*(r_1 \vee r_2 - r_1 \wedge r_2)},
 \end{aligned}
\end{equation*}
where the $\eta^*$ process is defined in Equation \eqref{E:Integrating-factor-1-der}.
\end{lemma}

\begin{proof}[Proof of Lemma \ref{L:Second-der-Gamma-f-4-5}]
For $r_1 \vee r_2 \le u \le t,$ we write the process $\eta_{t,u}^*$ as $\eta_{t,{r_1 \vee r_2}}^* \left({\eta_{u,{r_1 \vee r_2}}^*}\right)^{-1}$. Applying H\"older's inequality, Lemma \ref{L:Integrating-Factor-first-der}, and the martingale moment inequality \cite[Proposition 3.26]{KS91}, we obtain
\begin{equation}\label{E:Eq-8}
\begin{aligned}
& \BE \left[ \left( \int_{r_1 \vee r_2}^t \alpha_u \eta^*_{t,u} \Gamma^f_4 (X_u, \theta_u) \thinspace dW_u \right)^{2p} \right] = \BE \left[ \left( \int_{r_1 \vee r_2}^t \left[\alpha_u \eta_{t,u}^* {f}_{x x \theta}(X_u,\theta_u)D_{r_1} X_u \cdot D_{r_2} X_u \right]dW_u  \right)^{2p} \right] \\
&  \quad    = \BE \left[\left(\eta_{t,{r_1 \vee r_2}}^*\right)^{2p} \left( \int_{r_1 \vee r_2}^t \left[\alpha_u \left({\eta_{u,{r_1 \vee r_2}}^*}\right)^{-1} {f}_{x x \theta}(X_u,\theta_u)D_{r_1} X_u \cdot D_{r_2} X_u  \right]dW_u  \right)^{2p} \right]\\
& \quad  \le \left[ \BE \left(\eta_{t,{r_1 \vee r_2}}^*\right)^{4p} \right]^{\frac{1}{2}} \left[ \BE \left( \int_{{r_1 \vee r_2}}^t \left[\alpha_u \left({\eta_{u,{r_1 \vee r_2}}^*}\right)^{-1} {f}_{x x \theta}(X_u,\theta_u)D_{r_1} X_u \cdot D_{r_2} X_u \right]dW_u\right)^{4p} \right]^{\frac{1}{2}} \\
& \quad \le K \frac{(r_1 \vee r_2)^{ 2p C_{\bar{g}} C_\alpha}}{t^{ 2p C_{\bar{g}} C_\alpha}} \left[ \BE \left( \int_{r_1 \vee r_2}^t \left[\alpha_u^2 \left({\eta_{u,r_1 \vee r_2}^*}\right)^{-2} {f}_{xx \theta}(X_u,\theta_u)^2 (D_{r_1} X_u)^2 (D_{r_2} X_u)^2 \right] du \right)^{2p} \right]^{\frac{1}{2}}.
\end{aligned}
\end{equation}
We now estimate the term $\BE \left[ \left( \int_{r_1 \vee r_2}^t \alpha_u^2 \left({\eta_{u,r_1 \vee r_2}^*}\right)^{-2} {f}_{x x \theta}(X_u,\theta_u)^2 (D_{r_1} X_u)^2 (D_{r_2} X_u)^2 \thinspace du \right)^{2p} \right]$ appearing in the right-hand side of \eqref{E:Eq-8}. Using H\"older's inequality, Assumption \ref{A:Growth-f-g} to control the growth of $f_{x x \theta}$, the uniform moment bounds for the processes $X_t$ and $\theta_t$, and Lemma \ref{L:Mal-der-X} (to control the Malliavin derivative $D_r X_t$), we obtain
\begin{align*}
& \BE \left( \int_{r_1 \vee r_2}^t \left[\alpha_u^2 \left({\eta_{u,{r_1 \vee r_2}}^*}\right)^{-2} {f}_{x x \theta}(X_u,\theta_u)^2 (D_{r_1} X_u)^2 (D_{r_2} X_u)^2 \right] du \right)^{2p}\\
 & \qquad \qquad \qquad \le K \underbrace{\int_{r_1 \vee r_2}^t \cdots \int_{r_1 \vee r_2}^t}_{2p-\text{times}} \left\{ \prod_{i=1}^{2p} \alpha_{u_i}^2 e^{-2 C^*(u_i-r_1)} e^{-2 C^*(u_i-r_2)} \right\} \times \\
 &  \qquad \qquad  \qquad \qquad  \qquad \qquad \qquad \qquad \qquad \qquad \quad \BE \left[ \prod_{i=1}^{2p} (\eta_{u_i,r_1 \vee r_2}^*)^{-2} \thinspace {f}_{x x \theta}(X_{u_i},\theta_{u_i})^{2} \right]du_1 \cdots du_{2p}  \\
& \qquad \qquad \qquad  \le K \int_{r_1 \vee r_2}^t \cdots \int_{r_1 \vee r_2}^t \left\{ \prod_{i=1}^{2p} \alpha_{u_i}^2 e^{-2 C^*(u_i-r_1)} e^{-2 C^*(u_i-r_2)} \right\} \times\\
& \qquad \qquad  \qquad \qquad  \qquad \qquad \quad   \left[ \prod_{i=1}^{2p} \left( \BE \left[ {\eta_{u_i,r_1 \vee r_2}^*}^{-8p} \right] \right)^{\frac{1}{4p}} \Big( \BE \left[{f}_{x x \theta}(X_{u_i},\theta_{u_i})^{8p} \right] \Big)^{\frac{1}{4p}} \right] du_1 \cdots du_{2p}.
\end{align*}
Next, applying Lemma \ref{L:L1} to handle the term $\BE \left[ \left({\eta_{u_i,r_1 \vee r_2}^*}\right)^{-8p} \right]$ and then using Lemma \ref{L:Integration-f-g-intermediate} with $\mathsf{D}=2 C_\alpha K_{g_{\theta \theta}}- 2, $ we obtain
\begin{equation}\label{E:Eq-9}
\begin{aligned}
&  \BE \left( \int_{r_1 \vee r_2}^t \left[\alpha_u^2 ({\eta_{u,{r_1 \vee r_2}}^*})^{-2} {f}_{x x \theta}(X_u,\theta_u)^2 (D_{r_1} X_u)^2 (D_{r_2} X_u)^2 \right] du \right)^{2p}\\
& \qquad  \le K \int_{r_1 \vee r_2}^t \cdots \int_{r_1 \vee r_2}^t \left\{ \prod_{i=1}^{2p} \alpha_{u_i}^2 e^{-2 C^*(u_i-r_1)} e^{-2 C^*(u_i-r_2)} \left( \frac{u_i}{r_1 \vee r_2} \right)^{2 C_\alpha K_{g_{\theta \theta}}} \right\} du_1 \cdots du_{2p}\\
& \qquad  =  \frac{K}{(r_1 \vee r_2)^{(4p)C_\alpha K_{g_{\theta \theta}}}} \left( \int_{r_1 \vee r_2}^t u^{2 C_\alpha K_{g_{\theta \theta}} - 2} e^{-2 C^*(u-r_1)} e^{-2 C^*(u-r_2)} \thinspace du\right)^{2p} \\
& \qquad  \le  \frac{K}{(r_1 \vee r_2)^{(4p)C_\alpha K_{g_{\theta \theta}}}} e^{(4p)C^* r_1} e^{(4p)C^* r_2} e^{-8pC^*(r_1 \vee r_2)} (r_1 \vee r_2)^{2p(2 C_\alpha K_{g_{\theta \theta}} - 2)} \\
& \qquad  =  \frac{K}{(r_1 \vee r_2)^{4p}} e^{-4pC^*(r_1 \vee r_2 - r_1 \wedge r_2)}.
\end{aligned}
\end{equation}
Finally, combining \eqref{E:Eq-8} and \eqref{E:Eq-9}, we obtain
\begin{equation*}
\begin{aligned}
\BE \left[ \left( \int_{r_1 \vee r_2}^t \left[\alpha_u \eta_{t,u}^* {f}_{x x \theta}(X_u,\theta_u)D_{r_1} X_u D_{r_2} X_u \right]dW_u  \right)^{2p} \right] \le K \frac{(r_1 \vee r_2)^{ 2p C_{\bar{g}} C_\alpha-2p}}{t^{ 2p C_{\bar{g}} C_\alpha}}e^{-2p C^*(r_1 \vee r_2 - r_1 \wedge r_2)}.
\end{aligned}
\end{equation*}
Next, using Lemma \ref{L:Mal-der-X} to handle the term $D^2_{r_1, r_2}X_u$, we obtain the same bound for the term $\BE \left[ \left( \int_{r_1 \vee r_2}^t \alpha_u \eta^*_{t,u} \Gamma^f_5 (X_u, \theta_u) \thinspace dW_u \right)^{2p} \right]$, which completes the proof.
\end{proof}
\color{black}

We now present Lemma \ref{L:Second-der-Gamma-f-3}, whose proof is technically delicate and requires careful intermediate estimates. The main difficulty is that the process $\Gamma_3^f(X_t, \theta_t)$ does not exhibit the same Malliavin-derivative structure that, in the analogous terms $\Gamma_i^f(X_t, \theta_t)$ for $i \in \{1,2,4,5\}$ (see Lemmas \ref{L:Second-der-Gamma-f-1-2} and \ref{L:Second-der-Gamma-f-4-5}), yields an additional exponential decay. More precisely, the process $\Gamma_3^f(X_t, \theta_t)$ contains the Malliavin derivative $D \theta_t$, whose magnitude is, in a suitable sense, of order $\mathscr{O}\left(\frac{r^{C_{\bar{g}}C_\alpha-1}}{t^{C_{\bar{g}}C_\alpha}}\right)$, whereas the terms $\Gamma_i^f(X_t, \theta_t)$, $i \in \{1,2,4,5\}$, involve the derivatives $D X_t$ or $D^2 X_t$, which enjoy exponential decay by Lemma \ref{L:Mal-der-X}. Before stating the lemma, we recall that $|g_{\theta \theta}(x, \theta)| \le K(1+ |x|^q + |\theta|^2 )$, and that the choice of a sufficiently large $t^*$ is specified in Lemma \ref{L:moment-bound} and Remark \ref{R:Rem-t*-choice}. More importantly, owing to this choice of $t^*$, the constant $K_{g_{\theta \theta}}^*$ is an upper bound for the quantity $\sup_{t \ge t^*}K_{g_{\theta \theta}}(t) \triangleq \sup_{t \ge t^*} \left[\BE|g_{\theta \theta}(X_t, \theta_t)|^p\right]^{\frac{1}{p}}$, $p \in \BN$. Here, for all $t \ge t^*$, the function $K_{g_{\theta \theta}}(t)$ is uniformly bounded by a constant independent of $C_\alpha$, and therefore $K_{g_{\theta \theta}}^* \triangleq \sup_{t \ge t^*}K_{g_{\theta \theta}}(t)$ is independent of $C_\alpha.$

\begin{lemma}\label{L:Second-der-Gamma-f-3}
Let $D_r \theta_t$ be the solution of Equation \eqref{E:first-order-Mal-Der-theta} and the process $\Gamma^f_3 (X_u, \theta_u)$ be defined as follows:
$\Gamma^f_3 (X_u, \theta_u)  \triangleq f_{\theta \theta \theta }(X_u, \theta_u) D_{r_1} \theta_u \cdot D_{r_2} \theta_u.$
We assume that the learning rate magnitude $C_\alpha$ is small enough so that $K_{g_{\theta \theta}}^* < \frac{1}{2C_\alpha} + 2 C_{\bar{g}}$ holds (i.e., Assumption \ref{A:Tech-Cond}). Then, for any $t > r_1 \vee r_2 \ge t^*,$ and $p \in \BN,$ there exists a time-independent positive constant $K$ such that
\begin{equation*}
\BE \left[ \left( \int_{r_1 \vee r_2}^t \alpha_u \eta^*_{t,u} \Gamma^f_3 (X_u, \theta_u) \thinspace dW_u \right)^{2p} \right]  \le \frac{K}{t^{ 2p C_{\bar{g}} C_\alpha}}\frac{(r_1 r_2)^{ 2p C_{\bar{g}} C_\alpha-2p}}{(r_1 \vee r_2)^{ 2p C_{\bar{g}}C_\alpha}},
\end{equation*}
where the $\eta^*$ process is defined in Equation \eqref{E:Integrating-factor-1-der}.
\end{lemma}

\begin{proof}[Proof of Lemma \ref{L:Second-der-Gamma-f-3}]
For notational convenience, we assume $Z_{t, r_1 \vee r_2} \triangleq \int_{r_1 \vee r_2}^t \alpha_u \eta^*_{t,u} \Gamma^f_3 (X_u, \theta_u)\thinspace dW_u$. We begin by applying H\"older's inequality, Lemma \ref{L:Integrating-Factor-first-der}, and the martingale moment inequality \cite[Proposition 3.26]{KS91} to obtain
\begin{equation}\label{E:E11}
\begin{aligned}
 \BE \left[ \left(Z_{t, r_1 \vee r_2}\right)^{2p} \right] & =  \BE \left[  \left(\eta_{t,{r_1 \vee r_2}}^*\right)^{{2p}} \left( \int_{r_1 \vee r_2}^t \left[\alpha_u ({\eta_{u,{r_1 \vee r_2}}^*})^{-1} \Gamma^f_3 (X_u, \theta_u) \right]dW_u\right)^{{2p}}  \right] \\
& \le   \left( \BE \left[ \left(\eta_{t,{r_1 \vee r_2}}^*\right)^{4p} \right] \right)^{\frac{1}{2}} \left( \BE \left[ \left( \int_{r_1 \vee r_2}^t \left[\alpha_u ({\eta_{u,{r_1 \vee r_2}}^*})^{-1} \Gamma^f_3 (X_u, \theta_u) \right]dW_u \right)^{4p} \right] \right)^{\frac{1}{2}} \\
& \le K  \frac{(r_1 \vee r_2)^{ 2p  C_{\bar{g}} C_\alpha}}{t^{ 2p  C_{\bar{g}} C_\alpha}} \left( \BE \left[ \left( \int_{r_1 \vee r_2}^t \left[\alpha_u ({\eta_{u,{r_1 \vee r_2}}^*})^{-1} \Gamma^f_3 (X_u, \theta_u) \right]dW_u \right)^{4p} \right] \right)^{\frac{1}{2}} \\
& =  K \frac{ (r_1 \vee r_2)^{ 2p  C_{\bar{g}}C_\alpha}}{t^{ 2p  C_{\bar{g}}  C_\alpha}}  \left( \BE \left[ \left( \int_{r_1 \vee r_2}^t \left[\alpha_u ({\eta_{u,{r_1 \vee r_2}}^*})^{-1} \Gamma^f_3 (X_u, \theta_u) \right]dW_u \right)^{4p} \right] \right)^{\frac{1}{2}}\\
& \le  K \frac{ (r_1 \vee r_2)^{ 2p  C_{\bar{g}}C_\alpha}}{t^{ 2p  C_{\bar{g}}  C_\alpha}}  \left( \BE \left[ \left( \int_{r_1 \vee r_2}^t \left[\alpha_u^2 ({\eta_{u,{r_1 \vee r_2}}^*})^{-2} \Gamma^f_3 (X_u, \theta_u)^2 \right]du \right)^{2p} \right] \right)^{\frac{1}{2}}.
\end{aligned}
\end{equation}
We now focus on the term $\BE \left[ \left( \int_{r_1 \vee r_2}^t \left[\alpha_u^2 ({\eta_{u,{r_1 \vee r_2}}^*})^{-2} \Gamma^f_3 (X_u, \theta_u)^2 \right]du \right)^{2p} \right]$ in the above equation. Using H\"older's inequality once more, we have
\begin{equation*}
\begin{aligned}
& \BE \left[ \left( \int_{r_1 \vee r_2}^t \left[\alpha_u^2 ({\eta_{u,{r_1 \vee r_2}}^*})^{-2} \Gamma^f_3 (X_u, \theta_u)^2 \right]du \right)^{2p} \right] \\
 & \qquad \qquad \quad \le K \underbrace{{\int_{r_1 \vee r_2}^t \cdots \int_{r_1 \vee r_2}^t}}_{2p-\text{times}} \left\{ \prod_{i=1}^{2p} \alpha_{u_i}^2 \right\} \prod_{i=1}^{2p} \left[ \left\{\BE (\eta_{u_i, r_1 \vee r_2}^*)^{-16p} \right\}^{\frac{1}{8p}} \left\{\BE {f}_{\theta \theta \theta}(X_{u_i}\theta_{u_i})^{16p}  \right\}^{\frac{1}{8p}}  \times  \right. \\
& \qquad \qquad  \qquad \qquad  \qquad \qquad \qquad \qquad \qquad \qquad \quad   \left.    \left\{\BE (D_{r_1} \theta_{u_i})^{16p} \right\}^{\frac{1}{8p}} \left\{\BE (D_{r_2} \theta_{u_i})^{16p} \right\}^{\frac{1}{8p}}
 \right] du_1 \cdots du_{2p}.
\end{aligned}
\end{equation*}
We next apply Lemma \ref{L:L1} and Assumption \ref{A:Growth-f-g} to obtain
\begin{equation}\label{E:E12}
\begin{aligned}
& \BE \left[ \left( \int_{r_1 \vee r_2}^t \left[\alpha_u^2 ({\eta_{u,{r_1 \vee r_2}}^*})^{-2} \Gamma^f_3 (X_u, \theta_u)^2 \right]du \right)^{2p} \right]  \\
&  \le K {\int_{r_1 \vee r_2}^t \cdots \int_{r_1 \vee r_2}^t} \left\{ \prod_{i=1}^{2p} \alpha_{u_i}^2 \left( \frac{u_i}{r_1 \vee r_2} \right)^{2 C_\alpha K_{g_{\theta \theta}}^*} \right\}  \left\{ \prod_{i=1}^{2p} \frac{r_1^{2C_{\bar{g}}C_\alpha-2}}{u_i^{2C_{\bar{g}}C_\alpha}} \right\} \left\{ \prod_{i=1}^{2p} \frac{r_2^{2C_{\bar{g}}C_\alpha-2}}{u_i^{2C_{\bar{g}}C_\alpha}} \right\} du_1 \cdots du_{2p} \\
& \qquad \qquad    \le K \left( \int_{r_1 \vee r_2}^t  \alpha_{u}^2 \left( \frac{u}{r_1 \vee r_2} \right)^{2 C_\alpha K_{g_{\theta \theta}}^*}   \left\{  \frac{r_1^{2C_{\bar{g}}C_\alpha-2}}{u^{2C_{\bar{g}}C_\alpha}} \right\} \left\{  \frac{r_2^{2C_{\bar{g}}C_\alpha-2}}{u^{2C_{\bar{g}}C_\alpha}} \right\} du \right)^{2p} \\
& \qquad \qquad    = K \frac{(r_1 r_2)^{(4p)C_{\bar{g}}C_\alpha-4p}}{(r_1 \vee r_2)^{(4p) C_\alpha K_{g_{\theta \theta}}^*}}  \left( \int_{r_1 \vee r_2}^t u^{2 C_\alpha K_{g_{\theta \theta}}^* - 2 - 4C_{\bar{g}}C_\alpha} \thinspace du\right)^{2p}\\
& \qquad \qquad   = K \frac{(r_1 r_2)^{(4p)C_{\bar{g}}C_\alpha-4p}}{(r_1 \vee r_2)^{(4p) C_\alpha K_{g_{\theta \theta}}^*}}  \left[ \frac{1}{2 C_\alpha K_{g_{\theta \theta}}^* - 1 - 4C_{\bar{g}}C_\alpha} u^{2 C_\alpha K_{g_{\theta \theta}}^* - 1 - 4C_{\bar{g}}C_\alpha} |_{u=r_1 \vee r_2}^t  \right]^{2p}.
\end{aligned}
\end{equation}
Next, if $2 C_\alpha K_{g_{\theta \theta}}^* - 1 - 4C_{\bar{g}}C_\alpha <0,$ i.e., $K_{g_{\theta \theta}}^* < \frac{1}{2C_\alpha} + 2 C_{\bar{g}},$ then from the above Equation \eqref{E:E12}, we have
\begin{equation*}
\BE \left[ \left( \int_{r_1 \vee r_2}^t \left[\alpha_u^2 ({\eta_{u,{r_1 \vee r_2}}^*})^{-2} \Gamma^f_3 (X_u, \theta_u)^2 \right]du \right)^{2p} \right] \le K \frac{(r_1 r_2)^{(4p)C_{\bar{g}}C_\alpha-4p}}{(r_1 \vee r_2)^{(4p) C_\alpha K_{g_{\theta \theta}}^*}} (r_1 \vee r_2)^{(4p) C_\alpha K_{g_{\theta \theta}}^* - 2p - (8p) C_{\bar{g}}C_\alpha},
\end{equation*}
which, together with \eqref{E:E11}, yields
\begin{equation*}
\begin{aligned}
 \BE \left[ \left(Z_{t, r_1 \vee r_2}\right)^{2p} \right] & \le K \frac{ (r_1 \vee r_2)^{ 2p  C_{\bar{g}}C_\alpha}}{t^{ 2p  C_{\bar{g}}  C_\alpha}} \frac{(r_1 r_2)^{ 2p C_{\bar{g}}C_\alpha-2p}}{(r_1 \vee r_2)^{ 2p  C_\alpha K_{g_{\theta \theta}}^*}} (r_1 \vee r_2)^{ 2p  C_\alpha K_{g_{\theta \theta}}^* - p - (4p) C_{\bar{g}}C_\alpha} \\
& \le  \frac{K}{t^{ 2p C_{\bar{g}} C_\alpha}}\frac{(r_1 r_2)^{ 2p C_{\bar{g}} C_\alpha-2p}}{(r_1 \vee r_2)^{ 2p C_{\bar{g}}C_\alpha}},
 \end{aligned}
\end{equation*}
which completes the proof.
\end{proof}

\section{Rates corresponding to the cases $K_{g_{\theta \theta}}^* = \frac{1}{2C_\alpha} + 2 C_{\bar{g}}$, and $K_{g_{\theta \theta}}^* > \frac{1}{2C_\alpha} + 2 C_{\bar{g}}$}\label{S:Slower-rates}
In this section, we discuss the rates for the term \( d_W \left( \mathsf{F}_t, N \right) \) in Theorem \ref{T:Main-theorem}, corresponding to the cases \( K_{g_{\theta \theta}}^* = \frac{1}{2C_\alpha} + 2 C_{\bar{g}} \) and \( K_{g_{\theta \theta}}^* > \frac{1}{2C_\alpha} + 2 C_{\bar{g}} \). We note that these two cases arise from Equation \eqref{E:E12} in the proof of Lemma \ref{L:Second-der-Gamma-f-3}. In both instances, the rate of convergence is slower than in the previous case. For these two regimes, we first compute the rates for the terms $\BE \left[ \left( D^2_{r_1, r_2} \theta_t \right)^{2p} \right]$ and
$\int_{[t^*,t]^4} \left( \BE |D^2_{u,r} \theta_t|^4\right)^{\frac{1}{4}} \left( \BE |D^2_{u,s} \theta_t|^4\right)^{\frac{1}{4}} \left( \BE |D^2_{w,r} \theta_t|^4\right)^{\frac{1}{4}} \left( \BE |D^2_{w,s} \theta_t|^4\right)^{\frac{1}{4}} \thinspace du \thinspace  ds \thinspace dw \thinspace dr $.

It is important to note that, in Case II below (see Equation \eqref{E:Eq71}), for small learning rate magnitude \( C_\alpha \), the condition $K_{g_{\theta \theta}}^* < \frac{1}{2C_\alpha} + 2 C_{\bar{g}}$ can be weakened to $K_{g_{\theta \theta}}^* < \frac{1}{2C_\alpha} + 3 C_{\bar{g}}$. The difference between these two conditions comes from the appearance of \( 2 C_{\bar{g}} \) versus \( 3 C_{\bar{g}} \). In particular, under the stronger condition $K_{g_{\theta \theta}}^* < \frac{1}{2C_\alpha} + 2 C_{\bar{g}}$, Cases I and II do not arise.

\textbf{Case I: $K_{g_{\theta \theta}}^* = \frac{1}{2C_\alpha} + 2 C_{\bar{g}}.$}
In Equation \eqref{E:E12}, if $2 C_\alpha K_{g_{\theta \theta}}^* - 1 - 4C_{\bar{g}}C_\alpha =0,$ then
\begin{equation*}
\begin{aligned}
\BE \left[ \left( \int_{r_1 \vee r_2}^t \left[\alpha_u^2 ({\eta_{u,{r_1 \vee r_2}}^*})^{-2} \Gamma^f_3 (X_u, \theta_u)^2 \right]du \right)^{2p} \right] & \le K \frac{(r_1 r_2)^{(4p)C_{\bar{g}}C_\alpha-4p}}{(r_1 \vee r_2)^{(4p) C_\alpha K_{g_{\theta \theta}}^*}}  \left( \int_{r_1 \vee r_2}^t u^{2 C_\alpha K_{g_{\theta \theta}}^* - 2 - 4C_{\bar{g}}C_\alpha} \thinspace du\right)^{2p}\\
& \le K \frac{(r_1 r_2)^{(4p)C_{\bar{g}}C_\alpha-4p}}{(r_1 \vee r_2)^{(4p) C_\alpha K_{g_{\theta \theta}}^*}} (\log t)^{2p},
\end{aligned}
\end{equation*}
which, together with \eqref{E:E11}, yields
$ \BE \left[ \left(Z_{t, r_1 \vee r_2}\right)^{2p} \right] \le \frac{K}{t^{ 2p C_{\bar{g}} C_\alpha}}\frac{(r_1 r_2)^{ 2p C_{\bar{g}} C_\alpha-2p}}{(r_1 \vee r_2)^{ 2p C_{\bar{g}}C_\alpha}} (\log t)^p .$ In line with this, a careful modification of Lemma \ref{L:2-der-moments} and Proposition \ref{P:Second-der-product} yields the following estimates:
\begin{multline*}
\begin{aligned}
\BE \left[ \left( D^2_{r_1, r_2} \theta_t \right)^{2p} \right]
& \le K \frac{(r_1 \vee r_2)^{ 2p C_{\bar{g}} C_\alpha-2p}}{t^{ 2p C_{\bar{g}} C_\alpha}}e^{-2p C^*(r_1 \vee r_2 - r_1 \wedge r_2)} + \frac{K (\log t)^p }{t^{ 2p  C_{\bar{g}} C_\alpha}}\frac{(r_1 r_2)^{ 2p  C_{\bar{g}} C_\alpha-2p}}{(r_1 \vee r_2 )^{ 2p  C_{\bar{g}} C_\alpha}}\\
& \le K \frac{(r_1 \vee r_2)^{ 2p C_{\bar{g}} C_\alpha-2p}}{t^{ 2p C_{\bar{g}} C_\alpha}}e^{-2p C^*(r_1 \vee r_2 - r_1 \wedge r_2)} (\log t)^p + \frac{K (\log t)^p }{t^{ 2p  C_{\bar{g}} C_\alpha}}\frac{(r_1 r_2)^{ 2p  C_{\bar{g}} C_\alpha-2p}}{(r_1 \vee r_2 )^{ 2p  C_{\bar{g}} C_\alpha}},
\end{aligned}
\end{multline*}
and
\begin{multline*}
\int_{[t^*,t]^4} \left( \BE |D^2_{u,r} \theta_t|^4\right)^{\frac{1}{4}} \left( \BE |D^2_{u,s} \theta_t|^4\right)^{\frac{1}{4}} \left( \BE |D^2_{w,r} \theta_t|^4\right)^{\frac{1}{4}} \left( \BE |D^2_{w,s} \theta_t|^4\right)^{\frac{1}{4}} \thinspace du \thinspace  ds \thinspace dw \thinspace dr  \\
\le \begin{cases} \frac{K (\log t)^2}{t^3},  & {C_{\bar{g}} C_\alpha > \frac{5}{4}}\\  \frac{K (\log t)^3 }{t^3}, & {\frac{3}{4} \le C_{\bar{g}} C_\alpha \le \frac{5}{4}} \\ \frac{K (\log t)^2}{t^{4C_{\bar{g}} C_\alpha}}, & \frac{1}{2} < C_{\bar{g}} C_\alpha < \frac{3}{4}.\end{cases}
\end{multline*}
We now follow the same steps as in Section \ref{S:theorem-proof} and combine the above rates with those for the first derivative to obtain
$$d_W \left( \mathsf{F}_t,N \right) \leq \begin{cases} \frac{K \log t}{{t}^{\frac{1}{4}}},  & {C_{\bar{g}} C_\alpha \ge \frac{3}{4}}\\  \frac{K (\log t)^{\frac{3}{4}} }{t^{C_{\bar{g}} C_\alpha- \frac{1}{2}}}, & \frac{1}{2} < C_{\bar{g}}C_\alpha < \frac{3}{4}.\end{cases}$$

\textbf{Case II: $K_{g_{\theta \theta}}^* > \frac{1}{2C_\alpha} + 2 C_{\bar{g}}.$}
In Equation \eqref{E:E12}, if $2 C_\alpha K_{g_{\theta \theta}}^* - 1 - 4C_{\bar{g}}C_\alpha >0,$ then
$$\BE \left[ \left( \int_{r_1 \vee r_2}^t \left[\alpha_u^2 ({\eta_{u,{r_1 \vee r_2}}^*})^{-2} \Gamma^f_3 (X_u, \theta_u)^2 \right]du \right)^{2p} \right]   \le K \frac{(r_1 r_2)^{(4p)C_{\bar{g}}C_\alpha-4p}}{(r_1 \vee r_2)^{(4p) C_\alpha K_{g_{\theta \theta}}^*}} t^{4p C_\alpha K_{g_{\theta \theta}}^* - 2p - 8pC_{\bar{g}}C_\alpha},$$
which, together with Equation \eqref{E:E11}, yields
$\BE \left[ \left(Z_{t, r_1 \vee r_2}\right)^{2p} \right] \le \frac{K}{t^{ 2p C_{\bar{g}} C_\alpha}}\frac{(r_1 r_2)^{ 2p C_{\bar{g}} C_\alpha-2p}}{(r_1 \vee r_2)^{ 2p C_{\bar{g}}C_\alpha}} t^{p \delta^*},$
where $\delta^* \triangleq 2 C_\alpha K_{g_{\theta \theta}}^* - 1 - 4C_{\bar{g}}C_\alpha > 0$.
In line with this, a careful modification of Lemma \ref{L:2-der-moments} and Proposition \ref{P:Second-der-product} yields the following estimates:
\begin{equation*}
 \BE \left[ \left( D^2_{r_1, r_2} \theta_t \right)^{2p} \right]
\le K \frac{(r_1 \vee r_2)^{ 2p C_{\bar{g}} C_\alpha-2p}}{t^{ 2p C_{\bar{g}} C_\alpha}}e^{-2p C^*(r_1 \vee r_2 - r_1 \wedge r_2)} t^{p \delta^*} + \frac{K t^{p \delta^*} }{t^{ 2p  C_{\bar{g}} C_\alpha}}\frac{(r_1 r_2)^{ 2p  C_{\bar{g}} C_\alpha-2p}}{(r_1 \vee r_2 )^{ 2p  C_{\bar{g}} C_\alpha}},
\end{equation*}
and
\begin{multline}\label{E:Eq71}
\int_{[t^*,t]^4} \left( \BE |D^2_{u,r} \theta_t|^4\right)^{\frac{1}{4}} \left( \BE |D^2_{u,s} \theta_t|^4\right)^{\frac{1}{4}} \left( \BE |D^2_{w,r} \theta_t|^4\right)^{\frac{1}{4}} \left( \BE |D^2_{w,s} \theta_t|^4\right)^{\frac{1}{4}} \thinspace du \thinspace  ds \thinspace dw \thinspace dr  \\
\le \begin{cases} \frac{K t^{2 \delta^*}}{t^3},  & {C_{\bar{g}} C_\alpha > \frac{5}{4}}\\  \frac{K \log t }{t^3}t^{2 \delta^*}, & {\frac{3}{4} \le C_{\bar{g}} C_\alpha \le \frac{5}{4}} \\ \frac{K t^{2 \delta^*}}{t^{4C_{\bar{g}} C_\alpha}}, & \frac{1}{2} < C_{\bar{g}} C_\alpha < \frac{3}{4},\end{cases}
\end{multline}
whenever $4 C_\alpha K_{g_{\theta \theta}}^* - 2 - 12C_{\bar{g}}C_\alpha <0,$ i.e., $K_{g_{\theta \theta}}^* < \frac{1}{2C_\alpha} + 3 C_{\bar{g}}.$ We now follow the same steps as in Section \ref{S:theorem-proof} and combine the above rates with those for the first derivative to obtain
$$d_W \left( \mathsf{F}_t,N \right) \leq \begin{cases} \frac{K t^{\frac{\delta^*}{2}} \log t}{{t}^{\frac{1}{4}}},  & {C_{\bar{g}} C_\alpha \ge \frac{3}{4}}\\  \frac{K t^{\frac{\delta^*}{2}}}{t^{C_{\bar{g}} C_\alpha- \frac{1}{2}}}, & \frac{1}{2} < C_{\bar{g}}C_\alpha < \frac{3}{4}\end{cases}$$
where $\delta^* \triangleq 2 C_\alpha K_{g_{\theta \theta}}^* - 1 - 4C_{\bar{g}}C_\alpha > 0.$ Note that, in the latter expression,
$\frac{K t^{\frac{\delta^*}{2}}}{t^{C_{\bar{g}} C_\alpha- \frac{1}{2}}} = K t^{C_\alpha (K_{g_{\theta \theta}}^* - 3C_{\bar{g}})}$,
which converges to zero provided that $K_{g_{\theta \theta}}^* - 3C_{\bar{g}}<0$, i.e., the objective function must be more convex.

\section{Bounds associated with Pre-limit expectation and variance}\label{S:Pre-limit-Exp-Var}
In this section, we analyze the terms \(\mathbb{E} \left( |{\mathsf{F}_t}| \right)\left|{1 - \sqrt{\frac{\bar{\Sigma}}{\operatorname{Var}(\mathsf{F}_t)}}}\right|\) and \(\sqrt{\frac{\bar{\Sigma}}{\operatorname{Var}(\mathsf{F}_t)}} \thinspace |{\mathbb{E}(\mathsf{F}_t)}|\) appearing in Propositions \ref{P:Pre-limit-Expectation-P-1} and \ref{P:Prelimit-Sec-2-P1}, respectively. These results are essential for proving Theorem \ref{T:Main-theorem}. The key idea is to construct an appropriate Poisson equation to control the fluctuation terms, and then to use uniform-in-time moment bounds for the processes \(X_t\) and \(\theta_t\).

\subsection{Bound for the term $\sqrt{\frac{\bar{\Sigma}}{\operatorname{Var}(\mathsf{F}_t)}} \thinspace |{\mathbb{E}(\mathsf{F}_t)}|$}
In this section, we obtain a bound for the term \\ $\sqrt{\frac{\bar{\Sigma}}{\operatorname{Var}(\mathsf{F}_t)}} \thinspace |{\mathbb{E}(\mathsf{F}_t)}|$ in Proposition \ref{P:Pre-limit-Expectation-P-1}, where $\mathsf{F}_t \triangleq \sqrt{t}(\theta_t-\theta^*)$ and the limiting variance $\bar{\Sigma}$ is defined in Equation \eqref{E:Limiting-Variance}. This bound is used in Equation \eqref{E:Main-result-Proof-Eq-1}. We begin by using the second-order Taylor expansion $\bar{g}_{\theta}(\theta_t) = \bar{g}_{\theta}(\theta^*) + \bar{g}_{\theta \theta}(\theta^*) (\theta_t-\theta^*) + \frac{1}{2}\bar{g}_{\theta \theta \theta}(\theta_t^1) (\theta_t-\theta^*)^2 = \bar{g}_{\theta \theta}(\theta^*) (\theta_t-\theta^*) + \frac{1}{2}\bar{g}_{\theta \theta \theta}(\theta_t^1) (\theta_t-\theta^*)^2$ in Equation \eqref{E:Process-theta} to obtain
\begin{multline}\label{E:Pre-limit-Sec-2-Eq2}
d(\theta_t - \theta^*) = -\alpha_t  \bar{g}_{\theta \theta}(\theta^*) (\theta_t-\theta^*) \thinspace dt - \frac{1}{2} \alpha_t  \bar{g}_{\theta \theta \theta}(\theta_t^1) (\theta_t-\theta^*)^2 \thinspace dt \\
 + \alpha_t \Big(
  \bar{g}_{\theta}(\theta_t) - g_\theta(X_t,\theta_t)  \Big) \thinspace dt + \alpha_t f_\theta(X_t,\theta_t)  \thinspace dW_t.
\end{multline}
Defining $\Xi_{t,1}^* \triangleq e^{-\int_1^t \alpha_u \bar{g}_{\theta \theta}(\theta^*) \thinspace du}$, we may first solve \eqref{E:Pre-limit-Sec-2-Eq2} and then take expectations to obtain
\begin{equation*}
\begin{aligned}
\theta_t - \theta^*  & = \Xi_{t,1}^* (\theta_1 - \theta^*)  -  \frac{1}{2} \int_1^t \alpha_u \Xi_{t,u}^* \bar{g}_{\theta \theta \theta}(\theta_u^1) (\theta_u-\theta^*)^2 \thinspace du  \\
& \qquad \qquad \quad \qquad \qquad + \int_1^t \alpha_u \Xi_{t,u}^* \Big(
  \bar{g}_{\theta}(\theta_u) - g_\theta(X_u,\theta_u)  \Big) \thinspace du + \int_1^t \alpha_u \Xi_{t,u}^* f_\theta(X_u,\theta_u)\thinspace dW_u,
  \end{aligned}
  \end{equation*}
  \begin{equation}\label{E:Pre-limit-Exp-2}
\begin{aligned}
  \BE [\theta_t - \theta^*] &= \BE [\Xi_{t,1}^* (\theta_1 - \theta^*)] + \BE \left[ \int_1^t \alpha_u \Xi_{t,u}^* \Big(
  \bar{g}_{\theta}(\theta_u) - g_\theta(X_u,\theta_u)  \Big) \thinspace du\right] \\
  & \qquad \qquad  \qquad - \frac{1}{2} \BE \left[  \int_1^t \alpha_u \Xi_{t,u}^* \bar{g}_{\theta \theta \theta}(\theta_u^1) (\theta_u-\theta^*)^2 \thinspace du \right]  \triangleq \mathscr{J}_1(t)+ \mathscr{J}_2(t) + \mathscr{J}_3(t).
  \end{aligned}
  \end{equation}
For $\mathscr{J}_1(t)$, applying the strong convexity of $\bar{g}$ from Assumption \ref{A:Growth-f-g} yields
\begin{equation}\label{E:Prelimit-exp-J1}
|\mathscr{J}_1(t)| \le |\BE [\Xi_{t,1}^* (\theta_1 - \theta^*)]| \le K  |\BE \Xi_{t,1}^*| \le \frac{K}{t^{C_{\bar{g}}C_\alpha}}.
\end{equation}
To control the term $\mathscr{J}_2(t)$, we introduce the Poisson equation
\begin{equation}\label{E:Poisson-equation-prelimit}
\mathscr{L}_x \Psi (x, \theta)= \mathtt{G}(x,\theta), \quad {\text{where}} \quad \mathtt{G}(x,\theta) \triangleq \bar{g}_{\theta}(\theta)-{g}_{ \theta}(x,\theta),
\end{equation}
where $\mathscr{L}_x$ is the infinitesimal generator of the $X$ process. We apply It\^o's formula to the function $\psi(s,\Psi) = \alpha_s \Xi_{t,s}^* \Psi(X_s, \theta_s)$ (where $\Psi$ solves \eqref{E:Poisson-equation-prelimit}) to obtain
\begin{equation*}
\begin{aligned}
\alpha_t \Xi_{t,t}^* \Psi(X_t, \theta_t) & = \alpha_1 \Xi_{t,1}^* \Psi(X_1, \theta_1) + \int_1^t  \frac{d \alpha_u}{du} \Xi_{t,u}^* \Psi(X_u, \theta_u) \thinspace du + \int_1^t \alpha_u \frac{\partial \Xi_{t,u}^*}{\partial u}  \Psi(X_u, \theta_u) \thinspace du \\
& \qquad + \int_1^t \alpha_u \Xi_{t,u}^* \mathscr{L}_x \Psi (X_u, \theta_u)\thinspace du - \int_1^t \alpha_u^2 \Xi_{t,u}^* g_\theta(X_u, \theta_u) \Psi_\theta (X_u, \theta_u) \thinspace du \\
& \qquad + \frac{1}{2}\int_1^t \alpha_u^3 \Xi_{t,u}^* [f_\theta (X_u, \theta_u)]^2 \Psi_{\theta \theta}(X_u, \theta_u) \thinspace du  + \frac{1}{2}\int_1^t \alpha_u^2 \Xi_{t,u}^* f_\theta (X_u, \theta_u) \Psi_{x \theta}(X_u, \theta_u) \thinspace du \\
& \qquad + \int_1^t \alpha_u \Xi_{t,u}^* \Psi_{x}(X_u, \theta_u) \thinspace dW_u +  \int_1^t \alpha_u^2 \Xi_{t,u}^* \Psi_{\theta}(X_u, \theta_u)f_\theta (X_u, \theta_u) \thinspace dW_u,
\end{aligned}
\end{equation*}
which, using Equation \eqref{E:Poisson-equation-prelimit}, implies
\begin{equation}\label{E:Poisson-equation-prelimit-J}
\begin{aligned}
\mathscr{J}_2(t)& = \BE \left[ \int_1^t \alpha_u \Xi_{t,u}^* \mathtt{G}(X_u, \theta_u) \thinspace du\right] = \BE [\alpha_t \Xi_{t,t}^* \Psi(X_t, \theta_t)] - \BE[\alpha_1 \Xi_{t,1}^* \Psi(X_1, \theta_1) ] \\
& \qquad \quad -\BE \int_1^t  \frac{d \alpha_u}{du}  \Xi_{t,u}^* \Psi(X_u, \theta_u) \thinspace du  - \BE \int_1^t \alpha_u \frac{\partial \Xi_{t,u}^*}{\partial u}  \Psi(X_u, \theta_u) \thinspace du \\
& \qquad \quad + \BE \int_1^t \alpha_u^2 \Xi_{t,u}^* g_\theta(X_u, \theta_u) \Psi_\theta (X_u, \theta_u) \thinspace du  - \frac{1}{2} \BE \int_1^t \alpha_u^3 \Xi_{t,u}^* [f_\theta (X_u, \theta_u)]^2 \Psi_{\theta \theta}(X_u, \theta_u) \thinspace du \\
& \qquad \quad - \frac{1}{2} \BE \int_1^t \alpha_u^2 \Xi_{t,u}^* f_\theta (X_u, \theta_u) \Psi_{x \theta}(X_u, \theta_u) \thinspace du  \triangleq \sum_{i=1}^7 \mathscr{J}_2^i(t).
\end{aligned}
\end{equation}
The bounds for the terms above are established in the following lemmas.

\begin{lemma}\label{L:Prelimit-exp-J-12}
Let $\Psi(x, \theta)$ be the solution of Poisson Equation \eqref{E:Poisson-equation-prelimit} and let the processes $\mathscr{J}_2^1(t)$ and $\mathscr{J}_2^2(t)$ be defined in Equation \eqref{E:Poisson-equation-prelimit-J}:
\begin{equation*}
\mathscr{J}_2^1(t)= \BE [\alpha_t \Xi_{t,t}^* \Psi(X_t, \theta_t)],   \qquad \mathscr{J}_2^2(t)= - \BE[\alpha_1 \Xi_{t,1}^* \Psi(X_1, \theta_1)].
\end{equation*}
Then, for any $t \ge 1,$ there exists a time-independent positive constant $K$ such that
\begin{equation*}
|\mathscr{J}_2^1(t)| \le \frac{K}{t}, \qquad \text{and} \qquad
|\mathscr{J}_2^2(t)| \le \frac{K}{t^{C_{\bar{g}}C_\alpha}}.
\end{equation*}
\end{lemma}
\begin{proof}
The bound for $\mathscr{J}_2^1(t)$ follows directly from the definition of the learning rate, the polynomial growth of $\Psi$, and the uniform moment bounds for the processes $X$ and $\theta$ established in \cite{pardoux2001poisson, siri_spilio_2020}. For $\mathscr{J}_2^2(t)$, we first use the strong convexity of $\bar{g}$ from Assumption \ref{A:Growth-f-g}, which yields $\Xi_{t,1}^* \le \frac{1}{t^{C_{\bar{g}} C_\alpha}}.$ We then apply the same uniform moment bounds for $X$ and $\theta$ to complete the proof.
\end{proof}

\begin{lemma}\label{L:Poisson-equation-prelimit-J-34}
Let $\Psi(x, \theta)$ be the solution of Poisson Equation \eqref{E:Poisson-equation-prelimit} and let the processes
$\mathscr{J}_2^3(t)$ and $\mathscr{J}_2^4(t)$ be defined in Equation \eqref{E:Poisson-equation-prelimit-J}:
\begin{equation*}
\mathscr{J}_2^3(t)= - \BE \int_1^t   \frac{d \alpha_s}{ds}  \Xi_{t,s}^* \Psi(X_s, \theta_s) \thinspace ds,   \qquad \mathscr{J}_2^4(t)= - \BE \int_1^t \alpha_s \frac{\partial \Xi_{t,s}^*}{\partial s}  \Psi(X_s, \theta_s) \thinspace ds.
\end{equation*}
Then, for any $t\ge 1,$ there exists a time-independent constant $K>0$ such that
\begin{equation*}
|\mathscr{J}_2^3(t)| + |\mathscr{J}_2^4(t)| \le \begin{cases} \frac{K}{t},  & {C_{\bar{g}} C_\alpha > 1}\\  \frac{K \log t }{t}, & {C_{\bar{g}} C_\alpha = 1} \\ \frac{K}{t^{C_{\bar{g}} C_\alpha}}, & \frac{1}{2} < C_{\bar{g}} C_\alpha < 1.\end{cases}
\end{equation*}
\end{lemma}
\begin{proof}
To bound $\mathscr{J}_2^3(t)$, we use the definition of $\Xi^*$ (recall $\Xi_{t,1}^* \triangleq e^{-\int_1^t \alpha_u \bar{g}_{\theta \theta}(\theta^*) \thinspace du}$), the polynomial growth of $\Psi$, and the uniform moment bounds for the processes $X_t$ and $\theta_t$ to obtain
\begin{equation}\label{E:Poisson-equation-prelimit-J2}
\begin{aligned}
|\mathscr{J}_2^3(t)|  \le   \BE \int_1^t \frac{C_\alpha}{s^2} \Xi_{t,s}^* \left|\Psi(X_s, \theta_s)\right| \thinspace ds & \le \int_1^t \frac{K}{s^2} \left(\frac{s}{t}\right)^{C_{\bar{g}} C_\alpha} \BE \left|\Psi(X_s, \theta_s)\right| \thinspace ds  \le K \int_1^t \frac{s^{C_{\bar{g}} C_\alpha-2}}{t^{C_{\bar{g}} C_\alpha}} \thinspace ds.
\end{aligned}
\end{equation}
If $C_{\bar{g}} C_\alpha>1$, integrating \eqref{E:Poisson-equation-prelimit-J2} yields $|\mathscr{J}_2^3(t)| \le \frac{K}{t^{C_{\bar{g}} C_\alpha}}\left[ \frac{1}{C_{\bar{g}} C_\alpha-1}s^{C_{\bar{g}}C_\alpha-1} \right]_{s=1}^t \le \frac{K}{t}.$ If $C_{\bar{g}}C_\alpha=1$, we obtain $|\mathscr{J}_2^3(t)| \le \frac{K}{t}\int_1^t \frac{1}{s} \thinspace ds \le \frac{K \log t}{t}.$ Finally, when $\frac{1}{2}<C_{\bar{g}} C_\alpha < 1$, Equation \eqref{E:Poisson-equation-prelimit-J2} gives $|\mathscr{J}_2^3(t)| \le \frac{K}{t^{C_{\bar{g}} C_\alpha}}.$

We now bound $\mathscr{J}_2^4(t)$. Using straightforward algebra, the strong convexity of $\bar{g}$, Assumption \ref{A:Growth-f-g} to control the growth of $\bar{g}_{\theta \theta}$, and the uniform moment bounds for $X_t$ and $\theta_t$, we obtain
\begin{equation*}
\begin{aligned}
|\mathscr{J}_2^4(t)| & =\BE \left| \int_1^t \alpha_s \frac{\partial \Xi_{t,s}^*}{\partial s}  \Psi(X_s, \theta_s) \right| ds \le  \BE \int_1^t \alpha_s e^{-\int_s^t \alpha_u \bar{g}_{\theta \theta}(\theta_u^1)   \thinspace du} \alpha_s | \bar{g}_{\theta \theta}(\theta_s^1) \Psi(X_s, \theta_s)| \thinspace  ds\\
& \qquad \qquad \qquad \qquad \qquad \qquad \quad  \le K \int_1^t \frac{s^{C_{\bar{g}} C_\alpha-2}}{t^{C_{\bar{g}} C_\alpha}} \BE\left|\bar{g}_{\theta \theta}(\theta_s^1) \Psi(X_s, \theta_s)  \right| ds \le K \int_1^t \frac{s^{C_{\bar{g}} C_\alpha-2}}{t^{C_{\bar{g}} C_\alpha}} \thinspace ds.
\end{aligned}
\end{equation*}
The remaining steps follow exactly as in the analysis of $\mathscr{J}_2^3(t)$ and yield the same bound. This completes the proof.
\end{proof}

\begin{lemma}\label{L:Prelimit-exp-J-567}
Let $\Psi(x, \theta)$ be the solution of Poisson Equation \eqref{E:Poisson-equation-prelimit} and let the processes
$\mathscr{J}_2^5(t)$, $\mathscr{J}_2^6(t)$ and $\mathscr{J}_2^7(t)$ be defined in Equation \eqref{E:Poisson-equation-prelimit-J}:
\begin{equation*}
\begin{aligned}
\mathscr{J}_2^5(t) & = \BE \int_1^t \alpha_u^2 \Xi_{t,u}^* g_\theta(X_u, \theta_u) \Psi_\theta (X_u, \theta_u) \thinspace du,   \quad \mathscr{J}_2^6(t)=  \frac{-1}{2} \BE \int_1^t \alpha_u^3 \Xi_{t,u}^* [f_\theta (X_u, \theta_u)]^2 \Psi_{\theta \theta}(X_u, \theta_u) \thinspace du, \\
\mathscr{J}_2^7(t) & = - \frac{1}{2} \BE \int_1^t \alpha_u^2 \Xi_{t,u}^* f_\theta (X_u, \theta_u) \Psi_{x \theta}(X_u, \theta_u) \thinspace du.
\end{aligned}
\end{equation*}
Then, for any $t \ge 1,$ there exists a time-independent constant $K>0$ such that
\begin{equation*}
|\mathscr{J}_2^5(t)|+ |\mathscr{J}_2^6(t)| + |\mathscr{J}_2^7(t)| \le \begin{cases} \frac{K}{t},  & {C_{\bar{g}} C_\alpha > 1}\\  \frac{K \log t }{t}, & {C_{\bar{g}} C_\alpha = 1} \\ \frac{K}{t^{C_{\bar{g}} C_\alpha}}, & \frac{1}{2} < C_{\bar{g}} C_\alpha < 1.\end{cases}
\end{equation*}
\end{lemma}
\begin{proof}
The proof follows the same arguments as in Lemma \ref{L:Poisson-equation-prelimit-J-34}. For brevity, we omit the details.
\end{proof}

Finally, we bound the term $\mathscr{J}_3(t)$ in Equation \eqref{E:Pre-limit-Exp-2}.

\begin{lemma}\label{L:Prelimit-exp-J3}
Let $\mathscr{J}_3(t)$ be defined in Equation \eqref{E:Pre-limit-Exp-2}:
\begin{equation*}
\mathscr{J}_3(t)= - \frac{1}{2} \BE \left[  \int_1^t \alpha_u \Xi_{t,u}^* \bar{g}_{\theta \theta \theta}(\theta_u^1) (\theta_u-\theta^*)^2 du \right].
\end{equation*}
Then, for any $t \ge 1,$ there exists a time-independent constant $K$ such that
\begin{equation*}
|\mathscr{J}_3(t)| \le \begin{cases} \frac{K}{t},  & {C_{\bar{g}} C_\alpha > 1}\\  \frac{K \log t }{t}, & {C_{\bar{g}} C_\alpha = 1} \\ \frac{K}{t^{C_{\bar{g}} C_\alpha}}, & \frac{1}{2} < C_{\bar{g}} C_\alpha < 1.\end{cases}
\end{equation*}
\end{lemma}
\begin{proof}
Using the strong convexity of $\bar{g}$ in the definition of $\Xi_{t,u}^*$ and applying H\"older's inequality to handle the expectation of $\bar{g}_{\theta \theta \theta}(\theta_u^1) (\theta_u-\theta^*)^2$, we obtain
\begin{equation*}
|\mathscr{J}_3(t)| \le K \int_1^t \alpha_u \left(\frac{u}{t}\right)^{C_{\bar{g}} C_\alpha} \left[ \BE \bar{g}_{\theta \theta \theta}(\theta_u^1)^2 \right]^{\frac{1}{2}} \left[ \BE (\theta_u-\theta^*)^4 \right]^{\frac{1}{2}} \thinspace du \le K \int_1^t \frac{u^{C_{\bar{g}} C_\alpha-2}}{t^{C_{\bar{g}} C_\alpha}} \thinspace du,
\end{equation*}
where the final inequality follows from Assumption \ref{A:Growth-f-g} and the bounds from \cite{siri_spilio_2020}, namely $\BE \left[ |\theta_t - \theta^*|^p \right]\le \frac{K}{t^{\frac{p}{2}}}$ and $\sup_{t \ge 0 }\BE \left[ |\theta_t|^p \right] \le K$. The conclusion follows by the same integration argument as in Lemma \ref{L:Poisson-equation-prelimit-J-34}.
\end{proof}

\begin{proposition}\label{P:Pre-limit-Expectation-P-1}
Suppose that $\mathsf{F}_t \triangleq \sqrt{t}(\theta_t-\theta^*).$ Let $\theta_t$ be the solution of Equation \eqref{E:Process-theta}, and let the limiting variance $\bar{\Sigma}$ is defined in Equation \eqref{E:Limiting-Variance}. Then, for any $t \ge 1$ and $C_{\bar{g}}C_\alpha> \frac{1}{2},$ there exists a time-independent, positive constant $K$ such that
\begin{equation*}
\sqrt{\frac{\bar{\Sigma}}{\operatorname{Var}(\mathsf{F}_t)}} \thinspace |{\mathbb{E}(\mathsf{F}_t)}| \le \frac{1}{\sqrt{\operatorname{Var}(\theta_t)}}\times \begin{cases} \frac{K}{t},  & {C_{\bar{g}}C_\alpha > 1}\\ \frac{K}{t} +  \frac{K \log t }{t}, & {C_{\bar{g}}C_\alpha = 1} \\ \frac{K}{t^{C_{\bar{g}} C_\alpha}}, & \frac{1}{2} < C_{\bar{g}} C_\alpha < 1.\end{cases}
\end{equation*}
\end{proposition}
\begin{proof}
Using the definition of $\mathsf{F}_t$, we have
$\sqrt{\frac{\bar{\Sigma}}{\operatorname{Var}(\mathsf{F}_t)}} \thinspace |{\mathbb{E}(\mathsf{F}_t)}| = \sqrt{\frac{\bar{\Sigma}}{t \operatorname{Var}(\theta_t)}} \thinspace |{\mathbb{E}\sqrt{t}(\theta_t-\theta^*)}| = \sqrt{\frac{\bar{\Sigma}}{\operatorname{Var}(\theta_t)}} \thinspace |{\mathbb{E}(\theta_t-\theta^*)}|.$
Combining Lemmas \ref{L:Prelimit-exp-J-12} through \ref{L:Prelimit-exp-J-567} (which control the components of $\mathscr{J}_2(t)$ in \eqref{E:Poisson-equation-prelimit-J}), Lemma \ref{L:Prelimit-exp-J3} (which controls $\mathscr{J}_3(t)$), Equation \eqref{E:Prelimit-exp-J1} (which controls $\mathscr{J}_1(t)$), and the decomposition \eqref{E:Pre-limit-Exp-2}, yields the desired bound.
\end{proof}

\subsection{Bound for the term $\mathbb{E} \left(|{\mathsf{F}_t}| \right)\left|{1-\sqrt{\frac{\bar{\Sigma}}{\operatorname{Var}(\mathsf{F}_t)}}}\right|$}
This section is devoted to deriving a bound for the term $\mathbb{E} \left(|{\mathsf{F}_t}| \right)\left|{1-\sqrt{\frac{\bar{\Sigma}}{\operatorname{Var}(\mathsf{F}_t)}}}\right|$, which appears in Equation \eqref{E:Main-result-Proof-Eq-1}. The corresponding estimate is stated precisely in Proposition \ref{P:Prelimit-Sec-2-P1}. Using the definition $\mathsf{F}_t \triangleq \sqrt{t}(\theta_t-\theta^*)$, we have
\begin{equation}\label{E:Prelimit-Sec2-eq1}
\begin{aligned}
\mathbb{E} \left( |{\mathsf{F}_t}| \right)\left|{1-\sqrt{\frac{\bar{\Sigma}}{\operatorname{Var}(\mathsf{F}_t)}}}\right| & = \frac{\mathbb{E} \left( |{\mathsf{F}_t}| \right)}{\sqrt{{\operatorname{Var}(\mathsf{F}_t)}}}\left| \sqrt{{\operatorname{Var}(\mathsf{F}_t)}} - \sqrt{\bar{\Sigma}} \right|  = \frac{\mathbb{E} \left( |{\theta_t-\theta^*}| \right)}{\sqrt{{\operatorname{Var}(\theta_t-\theta^*)}}}\left| \sqrt{{\operatorname{Var}(\mathsf{F}_t)}} - \sqrt{\bar{\Sigma}} \right|.
\end{aligned}
\end{equation}
In view of \eqref{E:Prelimit-Sec2-eq1}, it is therefore sufficient to obtain a bound for the term $|{{\operatorname{Var}(\mathsf{F}_t)}} - {\bar{\Sigma}}|.$ By the definition of variance,
\begin{equation}\label{E:Var}
{{\operatorname{Var}(\mathsf{F}_t)}} - {\bar{\Sigma}} = t {{\operatorname{Var}(\theta_t-\theta^*)}} - \bar{\Sigma} = t \BE [|\theta_t-\theta^*|^2] - t |\BE (\theta_t-\theta^*)|^2 - \bar{\Sigma}.
\end{equation}
Recalling $\Xi_{t,1}^* \triangleq e^{-\int_1^t \alpha_u \bar{g}_{\theta \theta}(\theta^*) \thinspace du}$, we solve Equation \eqref{E:Pre-limit-Sec-2-Eq2} to obtain
\begin{equation*}
\theta_t -\theta^* = \Xi_{t,1}^* (\theta_1 - \theta^*) + \int_1^t \alpha_u \Xi_{t,u}^* \left[ f_\theta(X_u, \theta_u) - \Psi_x(X_u, \theta_u) \right] dW_u + \mathsf{R}(t;\Psi),
\end{equation*}
where $\Psi(x,\theta)$ is the solution of the Poisson Equation \eqref{E:Poisson-equation-prelimit}, and the remainder $\mathsf{R}(t; \Psi)$ is defined in Equation \eqref{E:Pre-limit-Sec-2-remainder} below.
\begin{equation}\label{E:Pre-limit-Sec-2-remainder}
 \begin{aligned}
 \mathsf{R}(t; \Psi) & \triangleq \alpha_t \Xi_{t,t}^* \Psi(X_t, \theta_t) - \alpha_1 \Xi_{t,1}^* \Psi(X_1, \theta_1) -  \int_1^t  \frac{d \alpha_u}{du}  \Xi_{t,u}^* \Psi(X_u, \theta_u) \thinspace du \\
 & \qquad \quad - \int_1^t \alpha_u \frac{\partial \Xi_{t,u}^*}{\partial u}  \Psi(X_u, \theta_u) \thinspace du - \frac{1}{2} \int_1^t \alpha_u \Xi_{t,u}^* \bar{g}_{\theta \theta \theta}(\theta_u^1) (\theta_u-\theta^*)^2 \thinspace du \\
& \qquad \quad + \int_1^t \alpha_u^2 \Xi_{t,u}^* g_\theta(X_u, \theta_u) \Psi_\theta (X_u, \theta_u) \thinspace du  - \frac{1}{2}\int_1^t \alpha_u^3 \Xi_{t,u}^* [f_\theta (X_u, \theta_u)]^2 \Psi_{\theta \theta}(X_u, \theta_u) \thinspace du \\
& \qquad \quad - \frac{1}{2}\int_1^t \alpha_u^2 \Xi_{t,u}^* f_\theta (X_u, \theta_u) \Psi_{x \theta}(X_u, \theta_u) \thinspace du   -  \int_1^t \alpha_u^2 \Xi_{t,u}^* \Psi_{\theta}(X_u, \theta_u)f_\theta (X_u, \theta_u) \thinspace dW_u.
 \end{aligned}
 \end{equation}
Next, by the triangle inequality and a martingale moment inequality, we obtain
\begin{equation}\label{E:PrelimitExp-Sec-2-eq3}
\begin{aligned}
 \BE \left[ |\theta_t - \theta^*|^2 \right] & \le K \BE \left[ |\Xi_{t,1}^* (\theta_1 - \theta^*)|^2 \right]  + K \BE \left[\left|\mathsf{R}(t;\Psi)\right|^2 \right] \\
& \qquad \qquad \qquad \qquad \qquad \quad + K \BE \left[\int_1^t \alpha_u^2 (\Xi_{t,u}^*)^2 \left[ f_\theta(X_u, \theta_u) - \Psi_x(X_u, \theta_u) \right]^2 du \right].
\end{aligned}
\end{equation}

\begin{lemma}\label{L:Prelimit-Remainder-Sec-2}
Let $\Psi$ be the solution of the Poisson Equation \eqref{E:Poisson-equation-prelimit} and let the remainder $\mathsf{R}(t; \Psi)$ be defined in Equation \eqref{E:Pre-limit-Sec-2-remainder}. Then, for any $t \ge 1,$ and $C_{\bar{g}} C_\alpha > \frac{1}{2},$ there exists a time-independent positive constant $K$ such that
\begin{equation*}
  \BE \left[ |\Xi_{t,1}^* (\theta_1 - \theta^*)|^2 \right]  +  \BE \left[\left|\mathsf{R}(t;\Psi)\right|^2 \right] \le \begin{cases} \frac{K}{t^2},  & {C_{\bar{g}} C_\alpha > 1}\\ \frac{K}{t^2} +  \frac{K( \log t)^2 }{t^2}, & {C_{\bar{g}} C_\alpha = 1} \\ \frac{K}{t^{2C_{\bar{g}} C_\alpha}}, & \frac{1}{2} < C_{\bar{g}} C_\alpha < 1.\end{cases}
 \end{equation*}
\end{lemma}
\begin{proof}
The proof follows from the triangle inequality and the same arguments used in Lemmas \ref{L:Prelimit-exp-J-12}, \ref{L:Poisson-equation-prelimit-J-34}, \ref{L:Prelimit-exp-J-567}, and \ref{L:Prelimit-exp-J3} to bound the individual terms appearing in the definition of $\mathsf{R}(t;\Psi)$. For brevity, we omit the details.
\end{proof}

Combining Equations \eqref{E:Var} and \eqref{E:PrelimitExp-Sec-2-eq3} with Lemma \ref{L:Prelimit-Remainder-Sec-2}, we obtain
\begin{multline}\label{E:PrelimitExp-Sec-2-eq4}
{{\operatorname{Var}(\mathsf{F}_t)}} - {\bar{\Sigma}} \le K t \BE \left[\int_1^t \alpha_u^2 (\Xi_{t,u}^*)^2 \left[ f_\theta(X_u, \theta_u) - \Psi_x(X_u, \theta_u) \right]^2 du \right] - {\bar{\Sigma}} \\
+ \begin{cases} \frac{K}{t},  & {C_{\bar{g}} C_\alpha > 1}\\ \frac{K}{t} +  \frac{K( \log t)^2 }{t}, & {C_{\bar{g}}C_\alpha = 1} \\ \frac{K}{t^{2C_{\bar{g}} C_\alpha-1}}, & \frac{1}{2} < C_{\bar{g}} C_\alpha < 1.\end{cases}
\end{multline}
We now focus on estimating the term\footnote{In this term, we do not multiply by $t$, since this factor is already incorporated in the definitions of the processes $\bar{\Sigma}_t$, $\Sigma_t$, and $\bar{V}_t$ in Equation \eqref{E:Prelimit-Exp-Notations}.}
\( K \BE \left[\int_1^t \alpha_u^2 (\Xi_{t,u}^*)^2 \left[ f_\theta(X_u, \theta_u) - \Psi_x(X_u, \theta_u) \right]^2 du \right] - \bar{\Sigma} \)
in Equation \eqref{E:PrelimitExp-Sec-2-eq4}. To this end, we introduce the following processes, which will play a central role in the calculations below. Since (see Equation \eqref{E:Limiting-Variance}) $\bar{\Sigma} \triangleq C_\alpha^2 \int_0^\infty e^{-2s \left(C_\alpha \bar{g}_{\theta\theta}(\theta^*)- \frac{1}{2} \right)}\bar{h}(\theta^*) \thinspace ds,$ we define
\begin{equation}\label{E:Prelimit-Exp-Notations}
\begin{aligned}
h(x,\theta) & \triangleq \left[ f_\theta(x, \theta) - \Psi_x(x, \theta) \right]^2, \qquad \quad
\bar{\Sigma}_t  \triangleq t \int_1^t \alpha^2_u (\Xi_{t,u}^*)^2 \bar{h}(\theta^*) \thinspace du, \\
\Sigma_t & \triangleq t \int_1^t \alpha^2_u (\Xi_{t,u}^*)^2 h(X_u, \theta_u) \thinspace du, \qquad \quad   \bar{V}_t  \triangleq t \int_1^t \alpha^2_u (\Xi_{t,u}^*)^2 \bar{h}(\theta_u) \thinspace du.
\end{aligned}
\end{equation}
With these notations, adding and subtracting $\BE \bar{\Sigma}_t$ in Equation \eqref{E:PrelimitExp-Sec-2-eq4} yields
\begin{equation}\label{E:PrelimitExp-Sec-2-eq5}
{{\operatorname{Var}(\mathsf{F}_t)}} - {\bar{\Sigma}} \le K [\BE \Sigma_t - \BE \bar{\Sigma}_t ] + [K \thinspace \BE \bar{\Sigma}_t - \bar{\Sigma} ]
+ \begin{cases} \frac{K}{t},  & {C_{\bar{g}} C_\alpha > 1}\\ \frac{K}{t} +  \frac{K( \log t)^2 }{t}, & {C_{\bar{g}}C_\alpha = 1} \\ \frac{K}{t^{2C_{\bar{g}} C_\alpha-1}}, & \frac{1}{2} < C_{\bar{g}} C_\alpha < 1.\end{cases}
\end{equation}

\begin{lemma}\label{L:Prelimit-Sec-2-lem1}
Let the processes $\Sigma_t$ and $\bar{\Sigma}_t$ be defined in Equation \eqref{E:Prelimit-Exp-Notations}. Then, for any $t \ge 1,$ and $C_{\bar{g}} C_\alpha > \frac{1}{2},$ there exists a time-independent positive constant $K$ such that
$$\left|\BE \Sigma_t - \BE \bar{\Sigma}_t \right| \le \begin{cases} \frac{K+K (\log t)^2}{\sqrt{t}},  & {C_{\bar{g}} C_\alpha \ge \frac{3}{4}}\\  \frac{K}{t^{2C_{\bar{g}} C_\alpha-1}}, & \frac{1}{2} < C_{\bar{g}} C_\alpha < \frac{3}{4}.\end{cases}$$
\end{lemma}
\begin{proof}
We begin by adding and subtracting $\BE \bar{V}_t$ in $[\BE \Sigma_t - \BE \bar{\Sigma}_t ]$ to obtain
\begin{equation}\label{E:Prelimit-Sec-2-eq6}
[\BE \Sigma_t - \BE \bar{\Sigma}_t ] = \BE [\Sigma_t - \bar{V}_t] + \BE [\bar{V}_t - \bar{\Sigma}_t] \triangleq  L_1(t) + L_2(t).
\end{equation}

We first consider the term $L_1(t)$ in \eqref{E:Prelimit-Sec-2-eq6}. To control it, we introduce the Poisson equation (where $h(x,\theta)$ is defined in \eqref{E:Prelimit-Exp-Notations})
\begin{equation}\label{E:Prelimit-exp-Sec2-Eq7-Poisson}
\mathscr{L}_x \Upsilon(x, \theta) = \mathtt{H}(x, \theta) \triangleq h(x,\theta)- \bar{h}(\theta).
\end{equation}
Applying It\^o's formula to the function $\varphi(u, \Upsilon) \triangleq \alpha_u^2 (\Xi_{t,u}^*)^2 \Upsilon(X_u, \theta_u)$ (where $\Upsilon$ solves \eqref{E:Prelimit-exp-Sec2-Eq7-Poisson}), we obtain
\begin{equation*}
\begin{aligned}
\alpha_t^2 (\Xi_{t,t}^*)^2 \Upsilon(X_t, \theta_t) & = \alpha_1^2 (\Xi_{t,1}^*)^2 \Upsilon(X_1, \theta_1) + \int_1^t  \frac{d \alpha_u^2}{du} (\Xi_{t,u}^*)^2 \Upsilon(X_u, \theta_u) \thinspace du \\
& + \int_1^t  \alpha_u^2 \frac{\partial}{\partial u} (\Xi_{t,u}^*)^2 \Upsilon(X_u, \theta_u) \thinspace du + \int_1^t  \alpha_u^2 (\Xi_{t,u}^*)^2 \mathscr{L}_x \Upsilon(X_u, \theta_u) \thinspace du \\
& - \int_1^t  \alpha_u^3 (\Xi_{t,u}^*)^2 g_\theta(X_u, \theta_u) \Upsilon_\theta (X_u, \theta_u) \thinspace du  + \frac{1}{2} \int_1^t  \alpha_u^4 (\Xi_{t,u}^*)^2 [f_\theta(X_u, \theta_u)]^2 \Upsilon_{\theta \theta} (X_u, \theta_u) \thinspace du \\
& + \frac{1}{2} \int_1^t  \alpha_u^3 (\Xi_{t,u}^*)^2 f_\theta(X_u, \theta_u) \Upsilon_{x \theta} (X_u, \theta_u) \thinspace du +  \int_1^t  \alpha_u^2 (\Xi_{t,u}^*)^2  \Upsilon_{x} (X_u, \theta_u) \thinspace dW_u \\
& +  \int_1^t  \alpha_u^3 (\Xi_{t,u}^*)^2 f_\theta(X_u, \theta_u) \Upsilon_{\theta} (X_u, \theta_u) \thinspace dW_u.
\end{aligned}
\end{equation*}
Combining this identity with \eqref{E:Prelimit-exp-Sec2-Eq7-Poisson} yields
\begin{equation*}
\begin{aligned}
L_1(t) & =  \BE [\Sigma_t - \bar{V}_t] = t \BE \int_1^t \alpha_u^2 (\Xi_{t,u}^*)^2 \left[h(X_u, \theta_u) - \bar{h}(\theta_u)\right] du = t \BE \int_1^t \alpha_u^2 (\Xi_{t,u}^*)^2 \mathscr{L}_x \Upsilon(X_u, \theta_u) \thinspace du \\
  & = t \alpha_t^2 \BE[ (\Xi_{t,t}^*)^2 \Upsilon(X_t, \theta_t) ] - t \alpha_1^2 \BE[ (\Xi_{t,1}^*)^2 \Upsilon(X_1, \theta_1) ] - t \BE \int_1^t \frac{d \alpha_u^2}{du}  (\Xi_{t,u}^*)^2 \Upsilon(X_u, \theta_u) \thinspace du \\
& \qquad - t \BE \int_1^t  \alpha_u^2 \frac{\partial}{\partial u} (\Xi_{t,u}^*)^2 \Upsilon(X_u, \theta_u) \thinspace du + t \BE \int_1^t  \alpha_u^3 (\Xi_{t,u}^*)^2 g_\theta(X_u, \theta_u) \Upsilon_\theta (X_u, \theta_u) \thinspace du \\
& \qquad   - \frac{t}{2} \BE \int_1^t  \alpha_u^4 (\Xi_{t,u}^*)^2 [f_\theta(X_u, \theta_u)]^2 \Upsilon_{\theta \theta} (X_u, \theta_u) \thinspace du  - \frac{t}{2} \BE  \int_1^t  \alpha_u^3 (\Xi_{t,u}^*)^2 f_\theta(X_u, \theta_u) \Upsilon_{x \theta} (X_u, \theta_u) \thinspace du.
\end{aligned}
\end{equation*}
Using the polynomial growth of $\Upsilon$ and the uniform moment bounds for $\theta_t$ and $X_t$ \cite{pardoux2001poisson, siri_spilio_2020}, and by repeating the same algebra as in Lemma \ref{L:Prelimit-Remainder-Sec-2}, we conclude that
\begin{equation}\label{E:Prelimit-exp-Sec2-Eq8}
 |\BE [\Sigma_t - \bar{V}_t]| \le \begin{cases} \frac{K}{t},  & {C_{\bar{g}} C_\alpha > 1}\\ \frac{K}{t} +  \frac{K( \log t)^2 }{t}, & {C_{\bar{g}}C_\alpha = 1} \\ \frac{K}{t^{2C_{\bar{g}} C_\alpha-1}}, & \frac{1}{2} < C_{\bar{g}}C_\alpha < 1.\end{cases}
 \end{equation}

We now turn to the term $L_2(t)$ in Equation \eqref{E:Prelimit-Sec-2-eq6}. By Taylor's theorem,
$L_2(t) = \BE [\bar{V}_t - \bar{\Sigma}_t] = t \BE \int_1^t \alpha^2_u (\Xi_{t,u}^*)^2 \left[\bar{h}(\theta_u) - \bar{h}(\theta^*)  \right] du = t \BE \int_1^t \alpha^2_u (\Xi_{t,u}^*)^2 \bar{h}_\theta (\theta_u^1) (\theta_u - \theta^*) \thinspace  du. $
Applying H{\"o}lder's inequality and the bound $\BE |\theta_t - \theta^*|^2 \le \frac{K}{t}$ \cite{siri_spilio_2020}, we obtain
\begin{equation}\label{E:Prelimit-exp-Sec2-Eq9}
\begin{aligned}
 |\BE [\bar{V}_t - \bar{\Sigma}_t] |  &  \le t \int_1^t \alpha^2_u \left( \frac{u}{t} \right)^{2C_{\bar{g}}C_\alpha} \left\{ \BE |\bar{h}_\theta (\theta_u^1)|^2 \right\}^{\frac{1}{2}} \left\{ \BE (\theta_u -\theta^*)^2 \right\}^{\frac{1}{2}} du \\
& \qquad \qquad \qquad \quad \le \frac{K}{t^{2C_{\bar{g}} C_\alpha-1}} \int_1^t u^{2C_{\bar{g}} C_\alpha-\frac{5}{2}} du
 \le \begin{cases} \frac{K}{\sqrt{t}},  & {C_{\bar{g}}C_\alpha > \frac{3}{4}}\\   \frac{K \log t }{\sqrt{t}}, & {C_{\bar{g}} C_\alpha = \frac{3}{4}} \\ \frac{K}{t^{2C_{\bar{g}}C_\alpha-1}} & \frac{1}{2} < C_{\bar{g}}C_\alpha < \frac{3}{4}.\end{cases}
\end{aligned}
\end{equation}
Combining \eqref{E:Prelimit-Sec-2-eq6}, \eqref{E:Prelimit-exp-Sec2-Eq8}, and \eqref{E:Prelimit-exp-Sec2-Eq9} yields the desired bound.
\end{proof}

\begin{lemma}\label{L:Prelimit-Sec-2-lem2}
Let the process $\bar{\Sigma}_t$ be defined in Equation \eqref{E:Prelimit-Exp-Notations} and the limiting variance $\bar{\Sigma}$ is defined in Equation \eqref{E:Limiting-Variance}. Then, for any $t \ge 1,$ and $C_{\bar{g}} C_\alpha > \frac{1}{2},$ there exists a time-independent positive constant $K$ such that
$$\left|\BE \bar{\Sigma}_t - \bar{\Sigma} \right| \le \frac{K}{t^{2C_{\bar{g}} C_\alpha -1}}.$$
\end{lemma}

\begin{proof}
Recalling the definitions of $\bar{\Sigma}_t$ and $\bar{\Sigma}$, we have
\begin{equation}\label{E:Prelimit-exp-Sec2-Eq10}
\BE \bar{\Sigma}_t - \bar{\Sigma} = t \BE \int_1^t \alpha^2_u (\Xi_{t,u}^*)^2 \bar{h}(\theta^*) \thinspace du - C_\alpha^2 \int_0^\infty e^{-2u \left(C_\alpha \bar{g}_{\theta\theta}(\theta^*)- \frac{1}{2} \right)}\bar{h}(\theta^*) \thinspace du.
\end{equation}
We first examine the term \( t \BE \int_1^t \alpha^2_u (\Xi_{t,u}^*)^2 \bar{h}(\theta^*) \thinspace du \). Using the definition of the learning rate $\alpha_t$ and the process $\Xi_{t,u}^*$, followed by a straightforward integration, we obtain
\begin{equation}\label{E:Prelimit-exp-Sec2-Eq11}
\begin{aligned}
t \BE \int_1^t \alpha^2_u (\Xi_{t,u}^*)^2 \bar{h}(\theta^*) \thinspace du & \le \frac{C_\alpha^2}{t^{2C_{\bar{g}} C_\alpha-1}} \int_1^t u^{2C_{\bar{g}} C_\alpha-2} \bar{h}(\theta^*) \thinspace du  = \frac{C_\alpha^2 \bar{h}(\theta^*)}{2C_{\bar{g}} C_\alpha-1}  \left[1 - \frac{1}{t^{2C_{\bar{g}}C_\alpha -1 }} \right].
\end{aligned}
\end{equation}
Next, for the second term $C_\alpha^2 \int_0^\infty e^{-2u \left(C_\alpha \bar{g}_{\theta\theta}(\theta^*)- \frac{1}{2} \right)}\bar{h}(\theta^*) \thinspace du$, strong convexity of $\bar{g}$ from Assumption \ref{A:Growth-f-g} implies
\begin{equation}\label{E:Prelimit-exp-Sec2-Eq12}
\begin{aligned}
C_\alpha^2 \int_0^\infty e^{-2u \left(C_\alpha \bar{g}_{\theta\theta}(\theta^*)- \frac{1}{2} \right)}\bar{h}(\theta^*) \thinspace du & = C_\alpha^2 \bar{h}(\theta^*) \int_0^\infty e^{-2u \left(C_\alpha \bar{g}_{\theta\theta}(\theta^*)- \frac{1}{2} \right)} du = \frac{C_\alpha^2 \bar{h}(\theta^*)}{2C_{\bar{g}} C_\alpha -1}.
\end{aligned}
\end{equation}
Combining \eqref{E:Prelimit-exp-Sec2-Eq10}, \eqref{E:Prelimit-exp-Sec2-Eq11}, and \eqref{E:Prelimit-exp-Sec2-Eq12} yields the desired result.
\end{proof}

\begin{proposition}\label{P:Prelimit-Sec-2-P1}
Let the process $\mathsf{F}_t \triangleq \sqrt{t}(\theta_t-\theta^*)$ and the limiting variance $\bar{\Sigma}$ is defined in Equation \eqref{E:Limiting-Variance}. Then, for any $t \ge 1,$ and $C_{\bar{g}} C_\alpha > \frac{1}{2},$ there exists a time-independent positive constant $K$ such that
\begin{equation*}
\mathbb{E} \left( |{\mathsf{F}_t}| \right)\left|{1-\sqrt{\frac{\bar{\Sigma}}{\operatorname{Var}(\mathsf{F}_t)}}}\right| \le  \frac{\mathbb{E} \left( |{\theta_t-\theta^*}| \right)}{\sqrt{{\operatorname{Var}(\theta_t-\theta^*)}}} \times  \begin{cases} \frac{K+K \log t}{{t}^{\frac{1}{4}}},  & {C_{\bar{g}} C_\alpha \ge \frac{3}{4}}\\  \frac{K}{t^{C_{\bar{g}} C_\alpha- \frac{1}{2}}}, & \frac{1}{2} < C_{\bar{g}}C_\alpha < \frac{3}{4}.\end{cases}
\end{equation*}
\end{proposition}
\begin{proof}
From the definition of ${\mathsf{F}_t}$, we obtain
$$\mathbb{E} \left( |{\mathsf{F}_t}| \right)\left|{1-\sqrt{\frac{\bar{\Sigma}}{\operatorname{Var}(\mathsf{F}_t)}}}\right|  = \frac{\mathbb{E} \left( |{\theta_t-\theta^*}| \right)}{\sqrt{{\operatorname{Var}(\theta_t-\theta^*)}}}\left| \sqrt{{\operatorname{Var}(\mathsf{F}_t)}} - \sqrt{\bar{\Sigma}} \right| .$$
Combining Equation \eqref{E:PrelimitExp-Sec-2-eq5} with Lemmas \ref{L:Prelimit-Sec-2-lem1} and \ref{L:Prelimit-Sec-2-lem2}, and performing straightforward algebra, yields the stated rate and completes the proof.
\end{proof}

\section{Proofs of auxiliary results and Preliminaries}\label{S:Appendix}
\subsection{Proofs of auxiliary results: Lemmas \ref{L:2-der-moments} through \ref{L:I-8-Second-der}} \label{S:Aux-Results}
\begin{proof}[Proof of Lemma \ref{L:2-der-moments}]
The proof follows by combining Lemmas \ref{L:Sec-der-initial-condition}, \ref{L:Second-der-Gamma-g}, \ref{L:Second-der-Gamma-f}, and \ref{L:Second-der-Gamma-ff} to control the terms $\BE \left[ (\eta^*_{t, r_1 \vee r_2})^{2p}\gamma(X_{r_1}, X_{r_2}, \theta_{r_1}, \theta_{r_2})^{2p} \right]$, $\BE \left[ \left( \int_{r_1 \vee r_2}^t \alpha_u \eta^*_{t,u} \Gamma^g (X_u, \theta_u) \thinspace du \right)^{2p} \right]$,\\ $\BE \left[ \left( \int_{r_1 \vee r_2}^t \alpha_u \eta^*_{t,u} \Gamma^f (X_u, \theta_u) \thinspace dW_u \right)^{2p} \right]$, and $\BE \left[ \left( \int_{r_1 \vee r_2}^t \alpha_u^2 \eta^*_{t,u} f_{\theta \theta}(X_u,\theta_u) \Gamma^f (X_u, \theta_u) \thinspace du \right)^{2p} \right]$, respectively.
\end{proof}

\begin{proof}[Proof of Lemma \ref{L:I-1-Second-Int}]
For $C_{\bar{g}}C_\alpha \ge 1,$ using integration by parts (Lemma \ref{L:Integration-f-g-intermediate} with $\mathsf{D}= 2C_{\bar{g}} C_\alpha-2, 3C_{\bar{g}} C_\alpha-3$ and $4C_{\bar{g}} C_\alpha-4$), together with straightforward algebra, we obtain
\begin{equation*}
\begin{aligned}
\mathscr{I}_1(t; C_{\bar{g}} C_\alpha)
& =  \frac{K}{t^{4C_{\bar{g}}C_\alpha}} \int_{u=t^*}^t \int_{s=u}^t \int_{w=s}^t s^{C_{\bar{g}} C_\alpha-1}w^{C_{\bar{g}}C_\alpha-1}e^{2C^*u} \left( \int_{r=w}^t r^{2C_{\bar{g}} C_\alpha-2}e^{-2C^*r}\thinspace dr \right)\thinspace dw \thinspace ds \thinspace du \\
& \le  \frac{K}{t^{4C_{\bar{g}}C_\alpha}} \int_{u=t^*}^t \int_{s=u}^t \int_{w=s}^t s^{C_{\bar{g}} C_\alpha-1}w^{C_{\bar{g}} C_\alpha-1}e^{2C^*u} w^{2C_{\bar{g}} C_\alpha-2}e^{-2C^*w} \thinspace dw \thinspace ds \thinspace du\\
& = \frac{K}{t^{4C_{\bar{g}} C_\alpha}} \int_{u=t^*}^t \int_{s=u}^t \int_{w=s}^t s^{C_{\bar{g}} C_\alpha-1}w^{3C_{\bar{g}} C_\alpha-3}e^{2C^*u} e^{-2C^*w} \thinspace dw \thinspace ds \thinspace du\\
& \le \frac{K}{t^{4C_{\bar{g}} C_\alpha}} \int_{u=t^*}^t \int_{s=u}^t  s^{4C_{\bar{g}} C_\alpha-4}e^{2C^*u} e^{-2C^*s} \thinspace ds \thinspace du  \le \frac{K}{t^{4C_{\bar{g}} C_\alpha}} \int_{u=t^*}^t u^{4C_{\bar{g}}C_\alpha-4} \thinspace du \\
& \le  \frac{K}{t^{4C_{\bar{g}} C_\alpha}} \left[ t^{4C_{\bar{g}}C_\alpha-3} - t^{* 4C_{\bar{g}}C_\alpha-3} \right] \le \frac{K}{t^3}.
\end{aligned}
\end{equation*}
Now, for $ \frac{1}{2} < C_{\bar{g}} C_\alpha < 1$, using the relation $\frac{1}{r} \le \frac{1}{w} \le \frac{1}{s} \le \frac{1}{u},$ we obtain
\begin{equation}\label{E:I-11-Second-int}
\begin{aligned}
\mathscr{I}_1(t; C_{\bar{g}} C_\alpha) & =  \frac{K}{t^{4C_{\bar{g}} C_\alpha}} \int_{u=t^*}^t \int_{s=u}^t \int_{w=s}^t \int_{r=w}^t  r^{2C_{\bar{g}}C_\alpha-2}s^{C_{\bar{g}}C_\alpha-1} w^{C_{\bar{g}}C_\alpha-1}e^{-2C^*r+ 2C^*u} \thinspace dr \thinspace dw \thinspace ds \thinspace du\\
& \le  \frac{K}{t^{4C_{\bar{g}}C_\alpha}} \int_{u=t^*}^t \int_{s=u}^t \int_{w=s}^t s^{C_{\bar{g}}C_\alpha-1}w^{C_{\bar{g}} C_\alpha-1}e^{2C^*u} w^{2C_{\bar{g}} C_\alpha-2}e^{-2C^*w} \thinspace dw \thinspace ds \thinspace du\\
& = \frac{K}{t^{4C_{\bar{g}} C_\alpha}} \int_{u=t^*}^t \int_{s=u}^t \int_{w=s}^t s^{C_{\bar{g}} C_\alpha-1}w^{3C_{\bar{g}} C_\alpha-3}e^{2C^*u} e^{-2C^*w} \thinspace dw \thinspace ds \thinspace du\\
& \le \frac{K}{t^{4C_{\bar{g}} C_\alpha}} \int_{u=t^*}^t \int_{s=u}^t  s^{4C_{\bar{g}} C_\alpha-4}e^{2C^*u} e^{-2C^*s} \thinspace ds \thinspace du  \le \frac{K}{t^{4C_{\bar{g}}C_\alpha}} \int_{u=t^*}^t u^{4C_{\bar{g}} C_\alpha-4} \thinspace du.
\end{aligned}
\end{equation}
Integrating yields
\begin{equation}\label{E:I-12-Second-int}
\mathscr{I}_1(t; C_{\bar{g}} C_\alpha) \le \frac{K}{t^{4C_{\bar{g}} C_\alpha}}\frac{1}{4C_{\bar{g}} C_\alpha-3} \left[t^{4C_{\bar{g}} C_\alpha-3}- t^{*4C_{\bar{g}} C_\alpha-3} \right].
\end{equation}
In Equation \eqref{E:I-12-Second-int}, if $4C_{\bar{g}}C_\alpha -3 >0$, then $\mathscr{I}_1(t; C_{\bar{g}} C_\alpha) \le \frac{K}{t^{4C_{\bar{g}} C_\alpha}} t^{4C_{\bar{g}}C_\alpha-3} = \frac{K}{t^3}$; whereas if $4C_{\bar{g}} C_\alpha -3 <0$, then $\mathscr{I}_1(t; C_{\bar{g}} C_\alpha) \le \frac{K}{t^{4C_{\bar{g}} C_\alpha}}.$ Finally, for $C_{\bar{g}}C_\alpha = \frac{3}{4}$, Equation \eqref{E:I-11-Second-int} gives
$\mathscr{I}_1(t; C_{\bar{g}} C_\alpha) \le  \frac{K}{t^3}\int_{u=t^*}^t \frac{1}{u}\thinspace du \le  K \frac{\log t}{t^3}.$
Combining these bounds yields
$$\mathscr{I}_1(t; C_{\bar{g}} C_\alpha) \le \begin{cases} \frac{K}{t^3},  & {C_{\bar{g}} C_\alpha > \frac{3}{4}}\\  \frac{K \log t }{t^3}, & {C_{\bar{g}}C_\alpha = \frac{3}{4}} \\ \frac{K}{t^{4C_{\bar{g}} C_\alpha}}, & \frac{1}{2} < C_{\bar{g}}C_\alpha < \frac{3}{4}.\end{cases}$$
\end{proof}

\begin{proof}[Proof of Lemma \ref{L:I-2-Second-der}]
Depending on the value of $C_{\bar{g}}C_\alpha$, we split the proof into three cases: $C_{\bar{g}}C_\alpha >1$, $C_{\bar{g}}C_\alpha = 1$, and $\frac{1}{2}< C_{\bar{g}}C_\alpha < 1.$  \\
\textbf{Case I: $C_{\bar{g}} C_\alpha >1.$} Using integration by parts (Lemma \ref{L:Integration-f-g-intermediate} with $\mathsf{D}= 2C_{\bar{g}} C_\alpha-2$) and straightforward algebra, we obtain
\begin{equation}\label{E:I-21-Second-der}
\begin{aligned}
\mathscr{I}_2(t; C_{\bar{g}} C_\alpha) & \le  \frac{K}{t^{4C_{\bar{g}}C_\alpha}} \int_{u=t^*}^t \int_{s=u}^t \int_{w=s}^t s^{2C_{\bar{g}}C_\alpha-2} w^{2C_{\bar{g}}C_\alpha-3} e^{-C^*w +2C^*u-C^*s} \thinspace dw \thinspace ds \thinspace du.
\end{aligned}
\end{equation}
If $2C_{\bar{g}} C_\alpha > 3$, then Equation \eqref{E:I-21-Second-der} and another application of integration by parts (Lemma \ref{L:Integration-f-g-intermediate} with $\mathsf{D}= 2C_{\bar{g}} C_\alpha-3, 4C_{\bar{g}} C_\alpha-5$) yield
\begin{equation*}
\begin{aligned}
\mathscr{I}_2(t; C_{\bar{g}} C_\alpha) & \le \frac{K}{t^{4C_{\bar{g}}C_\alpha}} \int_{u=t^*}^t \int_{s=u}^t s^{4C_{\bar{g}} C_\alpha-5} e^{2C^*u-2C^*s} \thinspace ds \thinspace du \\
& \le \frac{K}{t^{4C_{\bar{g}} C_\alpha}} \int_{u=t^*}^t  u^{4C_{\bar{g}} C_\alpha-5} \thinspace du = \frac{K}{t^{4C_{\bar{g}} C_\alpha}} \left[ \frac{1}{4C_{\bar{g}} C_\alpha-4} u^{4C_{\bar{g}} C_\alpha-4} \right]_{u=t^*}^t \le \frac{K}{t^4}.
\end{aligned}
\end{equation*}
Next, if $2C_{\bar{g}}C_\alpha <3$, then using $\frac{1}{w} \le \frac{1}{s}$ in \eqref{E:I-21-Second-der}, we obtain
\begin{equation}\label{E:I-22-Second-der}
\begin{aligned}
\mathscr{I}_2(t; C_{\bar{g}} C_\alpha) & \le \frac{K}{t^{4C_{\bar{g}} C_\alpha}} \int_{u=t^*}^t \int_{s=u}^t \frac{1}{s^{3-2C_{\bar{g}}C_\alpha}}s^{2C_{\bar{g}}C_\alpha-2} e^{2C^*u-2C^*s} \thinspace ds \thinspace du \\
& =  \frac{K}{t^{4C_{\bar{g}}C_\alpha}} \int_{u=t^*}^t \int_{s=u}^t s^{4C_{\bar{g}}C_\alpha-5} e^{2C^*u-2C^*s} \thinspace ds \thinspace du.
\end{aligned}
\end{equation}
In Equation \eqref{E:I-22-Second-der}, if $4C_{\bar{g}} C_\alpha > 5$, then integration by parts (Lemma \ref{L:Integration-f-g-intermediate} with $\mathsf{D}= 4C_{\bar{g}} C_\alpha-5$), followed by direct integration, yields
$\mathscr{I}_2(t; C_{\bar{g}} C_\alpha) \le \frac{K}{t^{4C_{\bar{g}}C_\alpha}} \int_{u=t^*}^t u^{4C_{\bar{g}} C_\alpha-5} \thinspace du \le \frac{K}{t^4}.$
If $4C_{\bar{g}}C_\alpha < 5$, then using $\frac{1}{s} \le \frac{1}{u}$ in \eqref{E:I-22-Second-der}, we obtain
\begin{equation*}
\mathscr{I}_2(t; C_{\bar{g}} C_\alpha) \le \frac{K}{t^{4C_{\bar{g}} C_\alpha}} \int_{u=t^*}^t u^{4C_{\bar{g}} C_\alpha-5} \thinspace du = \frac{K}{t^{4C_{\bar{g}}C_\alpha}} \left[ \frac{1}{4C_{\bar{g}}C_\alpha-4}u^{4C_{\bar{g}}C_\alpha-4} \right]_{u=t^*}^t \le \frac{K}{t^4}.
\end{equation*}
Finally, at the threshold points $C_{\bar{g}}C_\alpha =\frac{3}{2}, \frac{5}{4}$, it follows directly from the definition of $\mathscr{I}_2(t; C_{\bar{g}} C_\alpha)$ that $\mathscr{I}_2(t; C_{\bar{g}} C_\alpha) \le \frac{K}{t^4}.$
\\
\textbf{Case II: $C_{\bar{g}} C_\alpha = 1.$} Recalling the definition of $\mathscr{I}_2(t; C_{\bar{g}} C_\alpha)$, using the inequalities $\frac{1}{w} \le \frac{1}{s} \le \frac{1}{u}$ and basic integration estimates, we obtain
\begin{equation*}
\begin{aligned}
\mathscr{I}_2(t; C_{\bar{g}} C_\alpha) & = \frac{K}{t^{4}} \int_{u=t^*}^t \int_{s=u}^t \int_{w=s}^t \int_{r=w}^t \frac{1}{w} e^{-2C^*r+2C^*u-C^*s+C^*w} \thinspace dr \thinspace dw \thinspace ds \thinspace du\\
& \le \frac{K}{t^{4}} \int_{u=t^*}^t \int_{s=u}^t \int_{w=s}^t \frac{1}{w} e^{-C^*w +2C^*u-C^*s} \thinspace dw \thinspace ds \thinspace du \\
& \le \frac{K}{t^{4}} \int_{u=t^*}^t \int_{s=u}^t  \frac{1}{s} e^{2C^*u-2C^*s} \thinspace ds \thinspace du  \le \frac{K}{t^4} \int_{u=t^*}^t \frac{1}{u} \thinspace du \le  K \frac{\log t}{t^4}. \\
\end{aligned}
\end{equation*}
\textbf{Case III: $\frac{1}{2} < C_{\bar{g}}C_\alpha <1.$} Starting from the definition of $\mathscr{I}_2(t; C_{\bar{g}} C_\alpha)$, using $\frac{1}{r} \le \frac{1}{w} \le \frac{1}{s} \le \frac{1}{u}$ and basic integration estimates, we obtain
\begin{equation*}
\begin{aligned}
\mathscr{I}_2(t; C_{\bar{g}} C_\alpha) & \le \frac{K}{t^{4C_{\bar{g}} C_\alpha}} \int_{u=t^*}^t \int_{s=u}^t \int_{w=s}^t s^{2C_{\bar{g}}C_\alpha-2} w^{2C_{\bar{g}}C_\alpha-3} e^{-C^*w +2C^*u-C^*s} \thinspace dw \thinspace ds \thinspace du \\
& \le \frac{K}{t^{4C_{\bar{g}} C_\alpha}} \int_{u=t^*}^t \int_{s=u}^t s^{4C_{\bar{g}} C_\alpha-5} e^{2C^*u-2C^*s} \thinspace ds \thinspace du \le \frac{K}{t^{4C_{\bar{g}} C_\alpha}} \int_{u=t^*}^t  u^{4C_{\bar{g}} C_\alpha-5} \thinspace du \\
& \le \frac{K}{t^{4C_{\bar{g}} C_\alpha}}\left[ \frac{1}{4C_{\bar{g}}C_\alpha-4} u^{4C_{\bar{g}} C_\alpha-4} \right]_{u=t^*}^t \le  \frac{K}{t^{4C_{\bar{g}}C_\alpha}}.
\end{aligned}
\end{equation*}
Combining the three cases yields the desired bound and completes the proof.
\end{proof}

\begin{proof}[Proof of Lemma \ref{L:I-4-Second-Derivative}]
Depending on the value of $C_{\bar{g}}C_\alpha$, we split the proof into three cases: $C_{\bar{g}} C_\alpha >\frac{5}{4}$, $C_{\bar{g}} C_\alpha = \frac{5}{4}$, and $\frac{1}{2}< C_{\bar{g}} C_\alpha < \frac{5}{4}.$  \\
\textbf{Case I: $C_{\bar{g}} C_\alpha > \frac{5}{4}.$}
We first consider the subcase $C_{\bar{g}} C_\alpha>2$. Integration by parts (Lemma \ref{L:Integration-f-g-intermediate} with $\mathsf{D}= C_{\bar{g}} C_\alpha-2$, $2C_{\bar{g}} C_\alpha-4$, and $4C_{\bar{g}} C_\alpha-6$) yields
\begin{equation*}
\begin{aligned}
\mathscr{I}_4(t; C_{\bar{g}} C_\alpha) & \le \frac{K}{t^{4C_{\bar{g}}C_\alpha}} \int_{u=t^*}^t \int_{s=u}^t \int_{w=s}^t s^{2C_{\bar{g}} C_\alpha-2} w^{2C_{\bar{g}} C_\alpha-4} e^{-C^*w+2C^*u-C^*s} \thinspace  dw \thinspace ds \thinspace du \\
& \le \frac{K}{t^{4C_{\bar{g}} C_\alpha}} \int_{u=t^*}^t \int_{s=u}^t s^{4C_{\bar{g}} C_\alpha-6} e^{2C^*u-2C^*s} \thinspace ds \thinspace du \le \frac{K}{t^{4C_{\bar{g}} C_\alpha}} \int_{u=t^*}^t u^{4C_{\bar{g}}C_\alpha-6} \thinspace du \le \frac{K}{t^5}.
\end{aligned}
\end{equation*}
Next, for $C_{\bar{g}} C_\alpha < 2,$ using the definition of $\mathscr{I}_4(t; C_{\bar{g}} C_\alpha)$ and the inequality $\frac{1}{r} \le \frac{1}{w} \le \frac{1}{s} \le \frac{1}{u},$ we obtain
\begin{equation}\label{E:I-41-Second-der}
\mathscr{I}_4(t; C_{\bar{g}} C_\alpha) \le \frac{K}{t^{4C_{\bar{g}} C_\alpha}} \int_{u=t^*}^t \int_{s=u}^t s^{4C_{\bar{g}} C_\alpha-6} e^{2C^*u-2C^*s}\thinspace ds \thinspace du.
\end{equation}
In Equation \eqref{E:I-41-Second-der}, if $4C_{\bar{g}}C_\alpha -6 >0$, then applying Lemma \ref{L:Integration-f-g-intermediate} with $\mathsf{D}= 4C_{\bar{g}} C_\alpha-6$ gives $\mathscr{I}_4(t; C_{\bar{g}} C_\alpha) \le \frac{K}{t^{4C_{\bar{g}} C_\alpha}} \int_{u=t^*}^t u^{4C_{\bar{g}} C_\alpha-6} \thinspace du \le \frac{K}{t^5}$. If $4C_{\bar{g}} C_\alpha - 6<0$, then using $\frac{1}{s} \le \frac{1}{u}$ yields
\begin{equation}\label{E:I-42-Second-der}
\mathscr{I}_4(t; C_{\bar{g}} C_\alpha) \le \frac{K}{t^{4C_{\bar{g}}C_\alpha}} \int_{u=t^*}^t u^{4C_{\bar{g}}C_\alpha-6} \thinspace du.
\end{equation}
Hence, for $-1 \neq 4C_{\bar{g}} C_\alpha - 6<0 $,
\begin{equation}\label{E:I-43-Second-der}
\mathscr{I}_4(t; C_{\bar{g}} C_\alpha) \le \frac{K}{t^{4C_{\bar{g}}C_\alpha}}\left[ \frac{1}{4C_{\bar{g}}C_\alpha -5}u^{4C_{\bar{g}}C_\alpha-5} \right]_{u=t^*}^t,
\end{equation}
which in turn implies that if $4C_{\bar{g}}C_\alpha -5 >0$, then
$\mathscr{I}_4(t; C_{\bar{g}} C_\alpha) \le \frac{K}{t^{4C_{\bar{g}}C_\alpha}}t^{4C_{\bar{g}} C_\alpha-5} = \frac{K}{t^5}.$ Finally, at the threshold points $C_{\bar{g}}C_\alpha =\frac{3}{2}, 2$, it follows directly from the definition of $\mathscr{I}_4(t; C_{\bar{g}} C_\alpha)$ that $\mathscr{I}_4(t; C_{\bar{g}} C_\alpha) \le \frac{K}{t^5}.$ \\
\noindent
\textbf{Case II: $C_{\bar{g}} C_\alpha = \frac{5}{4}.$}
In Equation \eqref{E:I-42-Second-der}, if $4C_{\bar{g}} C_\alpha-6=-1$, then
$\mathscr{I}_4(t; C_{\bar{g}} C_\alpha) \le \frac{K}{t^{5}} \int_{u=t^*}^t u^{-1} \thinspace du \le K \frac{\log t}{t^5}.$ \\
\noindent
\textbf{Case III: $\frac{1}{2} < C_{\bar{g}}C_\alpha <\frac{5}{4}.$}
If $1 < C_{\bar{g}} C_\alpha < \frac{5}{4}$, then Equation \eqref{E:I-43-Second-der} implies that $\mathscr{I}_4(t; C_{\bar{g}} C_\alpha) \le \frac{K}{t^{4C_{\bar{g}}C_\alpha}}.$
If $C_{\bar{g}} C_\alpha =1,$ then using $\frac{1}{r} \le \frac{1}{w}$ yields
\begin{equation*}
\begin{aligned}
\mathscr{I}_4(t; C_{\bar{g}} C_\alpha) & = \frac{K}{t^4}  \int_{u=t^*}^t \int_{s=u}^t \int_{w=s}^t \int_{r=w}^t \frac{1}{rw} e^{-C^*r+2C^*u-C^*s} \thinspace dr \thinspace dw \thinspace ds \thinspace du \\
& \le \frac{K}{t^4}  \int_{u=t^*}^t \int_{s=u}^t   \frac{1}{s^2} e^{2C^*u-2C^*s} \thinspace ds \thinspace du  \le \frac{K}{t^4}.
\end{aligned}
\end{equation*}
Finally, if $\frac{1}{2} < C_{\bar{g}}C_\alpha <1,$ then by the definition of $\mathscr{I}_4(t; C_{\bar{g}} C_\alpha)$, the inequality $\frac{1}{r} \le \frac{1}{w} \le \frac{1}{s} \le \frac{1}{u}$, and direct integration estimates, we obtain
\begin{equation*}
\begin{aligned}
\mathscr{I}_4(t; C_{\bar{g}} C_\alpha) & \le \frac{K}{t^{4C_{\bar{g}} C_\alpha}} \int_{u=t^*}^t \int_{s=u}^t \int_{w=s}^t s^{2C_{\bar{g}}C_\alpha-2} w^{2C_{\bar{g}}C_\alpha-4} e^{-C^*w +2C^*u-C^*s} \thinspace dw \thinspace ds \thinspace du \\
& \le \frac{K}{t^{4C_{\bar{g}}C_\alpha}} \int_{u=t^*}^t \int_{s=u}^t s^{4C_{\bar{g}}C_\alpha-6} e^{2C^*u-2C^*s} \thinspace ds \thinspace du \le \frac{K}{t^{4C_{\bar{g}}C_\alpha}} \int_{u=t^*}^t  u^{4C_{\bar{g}}C_\alpha-6} \thinspace du \\
& \le \frac{K}{t^{4C_{\bar{g}}C_\alpha}}\left[ \frac{1}{4C_{\bar{g}}C_\alpha-5} u^{4C_{\bar{g}}C_\alpha-5} \right]_{u=t^*}^t \le  \frac{K}{t^{4C_{\bar{g}}C_\alpha}}.
\end{aligned}
\end{equation*}
Combining the three cases yields
$$\mathscr{I}_4(t; C_{\bar{g}} C_\alpha) \le \begin{cases} \frac{K}{t^5},  & {C_{\bar{g}}C_\alpha > \frac{5}{4}}\\  \frac{K \log t }{t^5}, & {C_{\bar{g}}C_\alpha = \frac{5}{4}} \\ \frac{K}{t^{4C_{\bar{g}}C_\alpha}} & \frac{1}{2} < C_{\bar{g}}C_\alpha < \frac{5}{4}.\end{cases}$$
\end{proof}

\begin{proof}[Proof of Lemma \ref{L:I-8-Second-der}]
Depending on the value of $C_{\bar{g}}C_\alpha$, we split the proof into three cases: $C_{\bar{g}} C_\alpha >\frac{3}{2}$, $C_{\bar{g}} C_\alpha = \frac{3}{2}$, and $\frac{1}{2}< C_{\bar{g}}C_\alpha < \frac{3}{2}.$  \\
\textbf{Case I: $C_{\bar{g}} C_\alpha > \frac{3}{2}.$}
We first consider the subcase $C_{\bar{g}} C_\alpha>2$. Integration by parts (Lemma \ref{L:Integration-f-g-intermediate} with $\mathsf{D}= C_{\bar{g}} C_\alpha-2$, $2C_{\bar{g}} C_\alpha-4$, and $3C_{\bar{g}} C_\alpha-6$) yields
\begin{equation*}
\begin{aligned}
\mathscr{I}_8(t; C_{\bar{g}} C_\alpha) & \le \frac{K}{t^{4C_{\bar{g}}C_\alpha}} \int_{u=t^*}^t \int_{s=u}^t \int_{w=s}^t s^{C_{\bar{g}}C_\alpha-2}w^{2C_{\bar{g}} C_\alpha-4} u^{C_{\bar{g}}C_\alpha-1} e^{-C^*w+C^*u} \thinspace  dw \thinspace ds \thinspace du \\
& \le \frac{K}{t^{4C_{\bar{g}}C_\alpha}} \int_{u=t^*}^t \int_{s=u}^t s^{3C_{\bar{g}}C_\alpha-6} u^{C_{\bar{g}}C\alpha-1} e^{C^*u-C^*s} \thinspace ds \thinspace du \le \frac{K}{t^{4C_{\bar{g}}C_\alpha}} \int_{u=t^*}^t u^{4C_{\bar{g}}C_\alpha-7} \thinspace du \le \frac{K}{t^6}.
\end{aligned}
\end{equation*}
Next, for $C_{\bar{g}}C_\alpha < 2,$ using the definition of $\mathscr{I}_8(t; C_{\bar{g}} C_\alpha)$ and the inequality $\frac{1}{r} \le \frac{1}{w} \le \frac{1}{s} \le \frac{1}{u},$ we obtain
\begin{equation}\label{E:I-81-Second-der}
\mathscr{I}_8(t; C_{\bar{g}} C_\alpha) \le \frac{K}{t^{4C_{\bar{g}}C_\alpha}} \int_{u=t^*}^t \int_{s=u}^t s^{3C_{\bar{g}} C_\alpha-6} u^{C_{\bar{g}}C_\alpha-1} e^{2C^*u-2C^*s} \thinspace ds \thinspace du \le \frac{K}{t^{4C_{\bar{g}} C_\alpha}} \int_{u=t^*}^t u^{4C_{\bar{g}}C_\alpha-7} \thinspace du.
\end{equation}
In \eqref{E:I-81-Second-der}, if $4C_{\bar{g}} C_\alpha -7>0$, then
$\mathscr{I}_8(t; C_{\bar{g}} C_\alpha) \le \frac{K}{t^6}.$
Next, if $-1 \neq 4C_{\bar{g}} C_\alpha - 7 < 0$, then \eqref{E:I-81-Second-der} yields
\begin{equation}\label{E:I-82-Second-der}
\mathscr{I}_8(t; C_{\bar{g}} C_\alpha) \le  \frac{K}{t^{4C_{\bar{g}}C_\alpha}}\left[ \frac{1}{4C_{\bar{g}} C_\alpha -6}u^{4C_{\bar{g}} C_\alpha-6} \right]_{u=t^*}^t,
\end{equation}
which implies that if $4C_{\bar{g}}C_\alpha -6 >0,$ then
$\mathscr{I}_8(t; C_{\bar{g}} C_\alpha) \le \frac{K}{t^{4C_{\bar{g}}C_\alpha}}t^{4C_{\bar{g}} C_\alpha-6} = \frac{K}{t^6}.$ Finally, at the threshold points $C_{\bar{g}}C_\alpha =\frac{7}{4}, 2$, it follows directly from the definition of $\mathscr{I}_8(t; C_{\bar{g}} C_\alpha)$ that $\mathscr{I}_8(t; C_{\bar{g}} C_\alpha) \le \frac{K}{t^6}.$\\
\noindent
\textbf{Case II: $C_{\bar{g}} C_\alpha = \frac{3}{2}.$}
In Equation \eqref{E:I-81-Second-der}, if $4C_{\bar{g}} C_\alpha-7=-1$, then
$\mathscr{I}_8(t; C_{\bar{g}} C_\alpha) \le \frac{K}{t^{6}} \int_{u=t^*}^t u^{-1} \thinspace du \le K \frac{\log t}{t^6}.$ \\
\noindent
\textbf{Case III: $\frac{1}{2} < C_{\bar{g}}C_\alpha <\frac{3}{2}.$}
If $1 < C_{\bar{g}} C_\alpha < \frac{3}{2}$, then Equation \eqref{E:I-82-Second-der} implies that $\mathscr{I}_8(t; C_{\bar{g}} C_\alpha) \le \frac{K}{t^{4C_{\bar{g}}C_\alpha}}.$
If $C_{\bar{g}} C_\alpha =1,$ then using $\frac{1}{r} \le \frac{1}{w} \le \frac{1}{s}$ yields
\begin{equation*}
\begin{aligned}
\mathscr{I}_8(t; C_{\bar{g}} C_\alpha) & \triangleq \frac{K}{t^4}  \int_{u=t^*}^t \int_{s=u}^t \int_{w=s}^t \int_{r=w}^t \frac{1}{rsw} e^{-C^*r+C^*u} \thinspace dr \thinspace dw \thinspace ds \thinspace du \\
& \le \frac{K}{t^4}  \int_{u=t^*}^t \int_{s=u}^t   \frac{1}{s^3} e^{2C^*u-2C^*s} \thinspace ds \thinspace du  \le \frac{C}{t^4}.
\end{aligned}
\end{equation*}
Finally, if $\frac{1}{2} < C_{\bar{g}} C_\alpha <1,$ then starting from the definition of $\mathscr{I}_8(t; C_{\bar{g}} C_\alpha)$, using $\frac{1}{r} \le \frac{1}{w} \le \frac{1}{s} \le \frac{1}{u}$ and basic integration estimates, we obtain
\begin{equation*}
\begin{aligned}
\mathscr{I}_8(t; C_{\bar{g}} C_\alpha) & \le \frac{K}{t^{4C_{\bar{g}} C_\alpha}} \int_{u=t^*}^t \int_{s=u}^t \int_{w=s}^t s^{C_{\bar{g}}C_\alpha-2} u^{C_{\bar{g}}C_\alpha-1} w^{2C_{\bar{g}}C_\alpha-4} e^{-C^*w + C^*u}\thinspace dw \thinspace ds \thinspace du \\
& \le \frac{K}{t^{4C_{\bar{g}}C_\alpha}} \int_{u=t^*}^t \int_{s=u}^t s^{3C_{\bar{g}}C_\alpha-6} u^{C_{\bar{g}} C_\alpha-1} e^{C^*u-C^*s} \thinspace ds \thinspace du \le \frac{K}{t^{4C_{\bar{g}}C_\alpha}} \int_{u=t^*}^t  u^{4C_{\bar{g}}C_\alpha-7} \thinspace du \\
& \le \frac{K}{t^{4C_{\bar{g}} C_\alpha}}\left[ \frac{1}{4C_{\bar{g}} C_\alpha-6} u^{4C_{\bar{g}}C_\alpha-6} \right]_{u=t^*}^t \le  \frac{K}{t^{4C_{\bar{g}}C_\alpha}}.
\end{aligned}
\end{equation*}
Combining the three cases yields the required bound.
\end{proof}

\subsection{Moment bound for the process $\theta$}\label{ss:MomentBoundTheta}
In this section, we show that, for sufficiently large $t^*$ (see Remark \ref{R:Rem-t*-choice} for the choice of $t^*$), the quantity $\sup_{t \ge t^* }\BE \left[ (\theta_t)^2 \right]$ is bounded by a constant independent of $C_\alpha.$ Following the arguments in \cite{siri_spilio_2020}, it suffices to control the term $\sup_{t \ge t^* }\BE \left[(\tilde{\theta}_t)^2 \right]$, where the process $\tilde{\theta}$ satisfies the {\sc sde}
\begin{equation}\label{E:tilde-theta}
d \tilde{\theta}_t = - \alpha_t \kappa(X_t) \tilde{\theta}_t \thinspace dt + \alpha_t f_{\theta}(X_t, \tilde{\theta}_t) \thinspace dW_t.
\end{equation}
The function $\kappa$ is defined in Assumption \ref{A:Moment-bounds} and satisfies $\kappa(x)> \gamma >0$.

\begin{lemma}\label{L:moment-bound}
Let the process $\tilde{\theta}$ be the solution of {\sc sde} \eqref{E:tilde-theta} and assume that $t^*$ is taken sufficiently large depending on the learning-rate magnitude $C_\alpha$. Then there exists a positive constant $\tilde{K}$, independent of $C_\alpha$ and uniform in time, such that
$$\sup_{t \ge t^* }\BE \left[(\tilde{\theta}_t)^2 \right] \le \tilde{K}.$$
\end{lemma}

\begin{remark}[Remark on the choice of $t^*$ in Lemma \ref{L:moment-bound}]\label{R:Rem-t*-choice}
Recall that the function $\kappa$ in Equation \eqref{E:tilde-theta} satisfies $\kappa(x)> \gamma >0.$ The choice of $t^*$ depends on the interplay between $\gamma$ and the learning-rate magnitude $C_\alpha.$ More precisely:
\begin{itemize}
\item  Case $0 <\gamma C_\alpha < \frac{1}{2}$. $t^*$ is chosen so that there exists a positive constant $\lambda_1$, independent of $C_\alpha$, such that
$$\max\left\{ \frac{1}{{t^*}^{2 \gamma C_\alpha}}, \thinspace \frac{1}{{(1-2\gamma C_\alpha) {t^*}^{2 \gamma C_\alpha}}}\right\} < \lambda_1.$$
\item Case $\gamma C_\alpha > 1$. $t^*$ is chosen so that there exists a positive constant $\lambda_2$, independent of $C_\alpha$, such that
$$\max\left\{ \frac{C_\alpha^2}{t^*}, \thinspace \frac{C_\alpha^2}{{(2\gamma C_\alpha -1)}t^*}, \thinspace \frac{C_\alpha^2}{t^* (2 \gamma C_\alpha -2)}\right\} < \lambda_2.$$
\item Case $\frac{1}{2} <\gamma C_\alpha < 1.$ $t^*$ is chosen so that there exists a positive constant $\lambda_3$, independent of $C_\alpha$, such that
$$ \max\left\{ \frac{C_\alpha^2}{t^*}, \thinspace \frac{C_\alpha^2}{{(2\gamma C_\alpha -1)}t^*}, \thinspace \frac{C_\alpha^2}{t^* \sqrt{ (2 \gamma C_\alpha -1) (2- 2 \gamma C_\alpha)}}\right\} < \lambda_3.$$
\end{itemize}
We also note that the borderline cases $\gamma C_\alpha =\frac{1}{2}, 1$ are straightforward, and in these cases the choice of $t^*$ does not depend on $C_\alpha.$
\end{remark}

\begin{proof}[Proof of Lemma \ref{L:moment-bound}]
We use It\^o's formula and the condition $\kappa(x) > \gamma >0$ to obtain
\begin{equation*}
\BE \left[(\tilde{\theta}_t)^2 \right] \le \BE \left[ (\tilde{\theta}_0)^2 \right] - 2 \gamma \BE \int_1^t \alpha_s  (\tilde{\theta}_s)^2 \thinspace ds + \BE \int_1^t \alpha_s^2 f_\theta(X_s, \tilde{\theta}_s)^2 \thinspace ds.
\end{equation*}
Solving this inequality for $\BE \left[(\tilde{\theta}_t)^2 \right]$ and then using the growth of $f_\theta$, together with the moment bounds for $X$, yields
\begin{equation}\label{E:E111}
\begin{aligned}
\BE \left[ (\tilde{\theta}_t)^2 \right] &  \le e^{-2 \gamma \int_1^t \alpha_s ds} \BE \left[ (\tilde{\theta}_0)^2 \right] + \int_1^t \alpha_s^2 e^{-2 \gamma \int_s^t \alpha_u du} \BE \left[  f_\theta(X_s, \tilde{\theta}_s)^2 \right]ds \\
& \le \BE \left[ (\tilde{\theta}_0)^2 \right] \left( \frac{1}{t} \right)^{2 \gamma C_\alpha} + \frac{C_\alpha^2}{t^{2 \gamma C_\alpha}}\int_1^t s^{2 \gamma C_\alpha-2} \left[ 1 + \BE |X_s|^q + \BE |\tilde{\theta}_s|^2\right]ds \\
& \le \BE \left[ (\tilde{\theta}_0)^2 \right] \left( \frac{1}{t} \right)^{2 \gamma C_\alpha} + \frac{K C_\alpha^2}{t^{2 \gamma C_\alpha}}\int_1^t s^{2 \gamma C_\alpha-2} \thinspace ds +  \frac{ C_\alpha^2}{t^{2 \gamma C_\alpha}}\int_1^t s^{2 \gamma C_\alpha-2} \BE |\tilde{\theta}_s|^2 \thinspace ds.
\end{aligned}
\end{equation}
Depending on the value of $\gamma C_\alpha$, we distinguish three cases.\\
\textbf{Case I}: If $2 \gamma C_\alpha -1< 0$, then Equation \eqref{E:E111} implies
 \begin{equation*}
\underbrace{\BE \left[(\tilde{\theta}_t)^2 \right]}_{x_t} \le \underbrace{\BE \left[ (\tilde{\theta}_0)^2 \right] \left( \frac{1}{t} \right)^{2 \gamma C_\alpha} + \frac{K C_\alpha^2}{t^{2 \gamma C_\alpha}} \frac{1}{(1- 2 \gamma C_\alpha )}}_{a_t}   +  \underbrace{\frac{ C_\alpha^2}{t^{2 \gamma C_\alpha}}}_{b_t}\int_1^t \underbrace{s^{2 \gamma C_\alpha-2}}_{\kappa_s} \BE |\tilde{\theta}_s|^2 \thinspace ds.
\end{equation*}
Fix $t^*$ sufficiently large so that $t \ge s \ge t^*$, and apply a generalized Gronwall inequality\footnote{For the inequality $x_t \le a_t + b_t \int_\alpha^t \kappa_s x_s \thinspace ds$, we have $x_t \le a_t + b_t \int_\alpha^t a_s \kappa_s e^{\int_s^t b_r \kappa_r dr}ds$.} with
\begin{equation*}
\begin{aligned}
e^{\int_s^t b_r \kappa_r \thinspace dr} & = e^{\int_s^t \frac{ C_\alpha^2}{r^{2 \gamma C_\alpha}} r^{2 \gamma C_\alpha-2} dr} = e^{\int_s^t \frac{C_\alpha^2}{r^2}dr}= e^{C_\alpha^2 \left[\frac{1}{s}- \frac{1}{t} \right]}, \qquad
a_t \kappa_t  = \frac{\BE (\tilde{\theta}_0)^2}{t^2}+ \frac{K C_\alpha^2}{(1- 2 \gamma C_\alpha)t^2}.
\end{aligned}
\end{equation*}
It follows that
\begin{equation*}
\begin{aligned}
\BE \left[ (\tilde{\theta}_t)^2 \right] & \le \BE \left[ (\tilde{\theta}_0)^2 \right] \left( \frac{1}{t} \right)^{2 \gamma C_\alpha} + \frac{K C_\alpha^2}{t^{2 \gamma C_\alpha}} \frac{1}{(1- 2 \gamma C_\alpha )} + \frac{ C_\alpha^2}{t^{2 \gamma C_\alpha}}\int_1^t \left\{ \frac{\BE (\tilde{\theta}_0)^2}{s^2}+ \frac{K C_\alpha^2}{(1- 2 \gamma C_\alpha)s^2} \right\} e^{C_\alpha^2 \left[\frac{1}{s}- \frac{1}{t} \right]} \thinspace ds \\
& \le \BE \left[ (\tilde{\theta}_0)^2 \right] \left( \frac{1}{t} \right)^{2 \gamma C_\alpha} + \frac{K C_\alpha^2}{t^{2 \gamma C_\alpha}} \frac{1}{(1- 2 \gamma C_\alpha )}  + \frac{ C_\alpha^2}{t^{2 \gamma C_\alpha}} \left\{ \BE (\tilde{\theta}_0)^2 +  \frac{K C_\alpha^2}{(1- 2 \gamma C_\alpha)}  \right\}\frac{1}{t^*} e^{\frac{C_\alpha^2}{t^*}} \\
& \le \BE \left[ (\tilde{\theta}_0)^2 \right] \left( \frac{1}{t^*} \right)^{2 \gamma C_\alpha} + \frac{K C_\alpha^2}{{t^*}^{2 \gamma C_\alpha}} \frac{1}{(1- 2 \gamma C_\alpha )}  + \frac{ C_\alpha^2}{{t^*}^{2 \gamma C_\alpha}} \left\{ \BE (\tilde{\theta}_0)^2 +  \frac{K C_\alpha^2}{(1- 2 \gamma C_\alpha)}  \right\} e^{\left(\frac{C_\alpha^2}{{t^*}^{2 \gamma C_\alpha}}\right)} \le  \tilde{K},
\end{aligned}
\end{equation*}
under the condition $\max\left\{ \frac{1}{{t^*}^{2 \gamma C_\alpha}}, \thinspace \frac{1}{{(1-2\gamma C_\alpha)}} \frac{1}{{t^*}^{2 \gamma C_\alpha}}\right\} < \lambda_1$. In particular, the constant $\tilde{K}$ is independent of $C_\alpha$.

\textbf{Case II}: If $2 \gamma C_\alpha -1>0$, then Equation \eqref{E:E111} implies
\begin{equation*}
\begin{aligned}
{\BE \left[(\tilde{\theta}_t)^2 \right]} &\le
 \BE \left[ (\tilde{\theta}_0)^2 \right] \left( \frac{1}{t} \right)^{2 \gamma C_\alpha} + \frac{K C_\alpha^2}{t^{2 \gamma C_\alpha}} \frac{1}{(2 \gamma C_\alpha -1)} \left[ t^{2 \gamma C_\alpha -1} -1 \right]  +  \frac{ C_\alpha^2}{t^{2 \gamma C_\alpha}}\int_1^t s^{2 \gamma C_\alpha-2} \BE |\tilde{\theta}_s|^2 ds \\
& \le  \BE \left[ (\tilde{\theta}_0)^2 \right] \left( \frac{1}{t} \right)^{2 \gamma C_\alpha} + \frac{K C_\alpha^2}{t} \frac{1}{(2 \gamma C_\alpha -1)}   +  {\frac{ C_\alpha^2}{t^{2 \gamma C_\alpha}}}\int_1^t {s^{2 \gamma C_\alpha-2}} \BE |\tilde{\theta}_s|^2 ds.
\end{aligned}
\end{equation*}
An application of Gronwall's inequality gives
\begin{equation}\label{E:E113}
\begin{aligned}
\BE \left[ (\tilde{\theta}_t)^2 \right] & \le
 \BE \left[ (\tilde{\theta}_0)^2 \right] \left( \frac{1}{t} \right)^{2 \gamma C_\alpha} + \frac{K C_\alpha^2}{t} \frac{1}{(2 \gamma C_\alpha -1)} + \frac{ C_\alpha^2}{t^{2 \gamma C_\alpha}} \int_1^t \frac{\BE (\tilde{\theta}_0)^2}{s^2} e^{C_\alpha^2 \left[\frac{1}{s}- \frac{1}{t} \right]} ds \\
& \qquad \qquad \qquad \qquad \qquad \qquad \qquad \qquad \qquad  + \frac{ C_\alpha^2}{t^{2 \gamma C_\alpha}} \int_1^t \frac{K C_\alpha^2}{(2 \gamma C_\alpha-1)}s^{2 \gamma C_\alpha -3}  e^{C_\alpha^2 \left[\frac{1}{s}- \frac{1}{t} \right]} ds.
\end{aligned}
\end{equation}
For $t \ge s > t^*$ with $t^*$ sufficiently large, it follows from \eqref{E:E113} that
\begin{equation}\label{E:E112}
\begin{aligned}
\BE \left[(\tilde{\theta}_t)^2 \right] & \le \BE \left[ (\tilde{\theta}_0)^2 \right] \left( \frac{1}{t} \right)^{2 \gamma C_\alpha} + \frac{K C_\alpha^2}{t} \frac{1}{(2 \gamma C_\alpha -1)} + \frac{ C_\alpha^2 \BE (\tilde{\theta}_0)^2}{t^{2 \gamma C_\alpha}} e^{\frac{C_\alpha^2}{t^*}} \\
& \qquad \qquad \qquad \qquad \qquad \qquad \qquad \quad +  \frac{ K C_\alpha^4}{(2 \gamma C_\alpha -1) (2 \gamma C_\alpha -2)t^{2 \gamma C_\alpha}}  e^{\frac{C_\alpha^2}{t^*}} \left[ t^{2 \gamma C_\alpha -2} -1 \right].
\end{aligned}
\end{equation}
If $\gamma C_\alpha > 1$, then using $\frac{1}{2 \gamma C_\alpha -1}< \frac{1}{2 \gamma C_\alpha -2}$ and $t^{2 \gamma C_\alpha}>t^2 >t >t^*,$ we obtain
\begin{equation*}
\begin{aligned}
\BE \left[ (\tilde{\theta}_t)^2 \right] & \le \BE \left[ (\tilde{\theta}_0)^2 \right] \frac{1}{t^*} + \frac{K C_\alpha^2}{t^*} \frac{1}{(2 \gamma C_\alpha -1)} + \frac{C_\alpha^2 \BE (\tilde{\theta}_0)^2}{t^{*}} e^{\frac{C_\alpha^2}{t^*}} +  \frac{K C_\alpha^4}{(2 \gamma C_\alpha -2)^2 t^{*}}  e^{\frac{C_\alpha^2}{t^*}} \le \tilde{K},
\end{aligned}
\end{equation*}
under the condition $\max\left\{ \frac{C_\alpha^2}{t^*}, \thinspace \frac{C_\alpha^2}{{(2\gamma C_\alpha -1)}t^*}, \thinspace \frac{C_\alpha^2}{t^* (2 \gamma C_\alpha -2)}\right\} < \lambda_2$. Again, the constant $\tilde{K}$ is independent of $C_\alpha$. If $\gamma C_\alpha =1$, then for all $t> t^*$, Equation \eqref{E:E113} gives
\begin{equation*}
\begin{aligned}
\BE \left[(\tilde{\theta}_t)^2 \right] & \le \BE \left[ (\tilde{\theta}_0)^2 \right] \left( \frac{1}{t} \right)^{2} + \frac{K C_\alpha^2}{t}  + \frac{ C_\alpha^2 \BE (\tilde{\theta}_0)^2}{t^{2}} e^{\frac{C_\alpha^2}{t^*}}\left[ 1- \frac{1}{t}\right] +  \frac{ K C_\alpha^4}{t^2}  e^{\frac{C_\alpha^2}{t^*}} \left[ \log t \right] \\
& \le \BE \left[ (\tilde{\theta}_0)^2 \right] \frac{1}{t^*} + \frac{K C_\alpha^2}{t^*} +  \frac{ C_\alpha^2 \BE (\tilde{\theta}_0)^2}{t^*} e^{\frac{C_\alpha^2}{t^*}}   +  \frac{ K C_\alpha^4}{t^*}  e^{\frac{C_\alpha^2}{t^*}} < \tilde{K}.
\end{aligned}
\end{equation*}
Finally, consider the regime $\frac{1}{2} <\gamma C_\alpha < 1$, that is, $-1 <2 \gamma C_\alpha -2 < 0$. From Equation \eqref{E:E112}, we obtain
\begin{equation*}
\begin{aligned}
\BE \left[ (\tilde{\theta}_t)^2  \right]
& \le \BE \left[ (\tilde{\theta}_0)^2 \right] \frac{1}{t^*} + \frac{K C_\alpha^2}{t^*} \frac{1}{(2 \gamma C_\alpha -1)} + \frac{ C_\alpha^2 \BE (\theta_0)^2}{t^{*}} e^{\frac{C_\alpha^2}{t^*}}  +  \frac{ K C_\alpha^4}{(2 \gamma C_\alpha -1) (2- 2 \gamma C_\alpha )t^{*}}  e^{\frac{C_\alpha^2}{t^*}} \le \tilde{K},
\end{aligned}
\end{equation*}
under the condition $\max\left\{ \frac{C_\alpha^2}{t^*}, \thinspace \frac{C_\alpha^2}{{(2\gamma C_\alpha -1)}t^*}, \thinspace \frac{C_\alpha^2}{t^* \sqrt{ (2 \gamma C_\alpha -1) (2- 2 \gamma C_\alpha)}}\right\} < \lambda_3$.\\
\textbf{Case III}: In Equation \eqref{E:E111}, if $2 \gamma C_\alpha -1=0$, i.e., $C_\alpha = \frac{1}{2 \gamma},$ then applying Gronwall's inequality for $t \ge t^*$ (with $t^*$ sufficiently large) yields
\begin{equation*}
\begin{aligned}
{\BE \left[ (\tilde{\theta}_t)^2 \right]} &  \le {\BE \left[ (\tilde{\theta}_0)^2 \right] \left( \frac{1}{t} \right) + \frac{K C_\alpha^2}{t} \log t}   +  {\frac{ C_\alpha^2}{t}}\int_1^t {s^{-1}} \BE |\tilde{\theta}_s|^2 ds \\
& \le \frac{\BE \left[(\tilde{\theta}_0)^2 \right]}{t^*} + \frac{K}{4 \gamma^2} + \frac{ C_\alpha^2}{t} \int_1^t  s^{-1} \BE |\tilde{\theta}_s|^2 ds
 \le \tilde{K} + \frac{ C_\alpha^2}{t}\tilde{K} \int_1^t  s^{-1} e^{C_\alpha^2 \left[\frac{1}{s}- \frac{1}{t} \right]} ds \\
 & \le \tilde{K}  + \frac{\tilde{K}}{t} e^{\frac{C_\alpha^2}{t^*}} \log t \le \tilde{K}.
\end{aligned}
\end{equation*}
This completes the proof.
\end{proof}

\color{black}

\subsection{Preliminaries}
In this section, we collect preliminaries on the Poisson equation on the whole space from \cite{pardoux2003poisson}, which we use repeatedly throughout the manuscript, and we also recall basic definitions and results from Malliavin calculus \cite{nualart2006malliavin} to keep the paper self-contained.

\subsubsection{Poisson equation}\label{S:Poisson-Equation}
Assume that Conditions \ref{A:f*-growth}, \ref{A:Well-Posedness}, and \ref{A:Growth-f-g} hold, and let $\mu$ denote the invariant measure of the process $X$. We further assume that the function $\mathsf{H} \in \mathscr{C}^{\alpha,2}(\mathscr{X}, \BR)$, where $\mathscr{X} \subseteq \BR$, satisfies the centering condition
\begin{equation}\label{E:Centring-condition}
\int_{\mathscr{X}} \mathsf{H} (x, \theta) \mu(dx)=0,
\end{equation}
and that there exist positive constants $K$, $\mathsf{p}_1$, $\mathsf{p}_2$, $\mathsf{p}_3$, and $q$ such that
\begin{equation}\label{E:Poisson-eq-H-growth}
\begin{aligned}
|\mathsf{H}(x,\theta)|  \le K (1+ |\theta|^{\mathsf{p}_1})(1+ |x|^q), \quad
|\mathsf{H}_{\theta}(x,\theta)| & \le K (1+ |\theta|^{\mathsf{p}_2})(1+ |x|^q), \\
|\mathsf{H}_{\theta \theta}(x,\theta)| & \le K (1+ |\theta|^{\mathsf{p}_3})(1+ |x|^q).
\end{aligned}
\end{equation}
Let $\mathscr{L}_x$ be the infinitesimal generator of $X$. Then the Poisson equation
$\mathscr{L}_x \mathsf{v}(x,\theta) =\mathsf{H} (x, \theta),\\ \int_{\mathscr{X}} \mathsf{v}(x,\theta) \mu(dx)=0 $
has a unique solution satisfying $\mathsf{v}(x, \cdot) \in \mathscr{C}^2$ for every $x\in \mathscr{X}$, $\mathsf{v}_{\theta \theta} \in \mathscr{C}(\mathscr{X} \times \BR)$, and there exist positive constants $\mathsf{K}$ and $m$ such that
\begin{equation}\label{E:Poisson-eq-Sol-growth}
\begin{aligned}
|\mathsf{v}(x,\theta)| + |\mathsf{v}_x(x,\theta)|  \le \mathsf{K} (1+ |\theta|^{\mathsf{p}_1})(1+ |x|^m), \thinspace
|\mathsf{v}_{\theta}(x,\theta)| + |\mathsf{v}_{x \theta}(x,\theta)|  & \le \mathsf{K} (1+ |\theta|^{\mathsf{p}_2})(1+ |x|^m), \\
|\mathsf{v}_{\theta \theta}(x,\theta)| + |\mathsf{v}_{x \theta \theta}(x,\theta)| & \le \mathsf{K} (1+ |\theta|^{\mathsf{p}_3})(1+ |x|^m).
\end{aligned}
\end{equation}
We emphasize that, in the growth bounds for $\mathsf{H}$ and $\mathsf{v}$ (and their derivatives), the exponents in $\theta$ coincide, whereas the powers in $x$ may differ. In our analysis, the function \(\mathsf{H}(x, \theta)\) always satisfies the centering condition \eqref{E:Centring-condition} and takes one of the following forms: \(\bar{g}_{\theta \theta}(\theta) - g_{\theta \theta}(x, \theta)\), \(\bar{g}_{\theta}(\theta) - g_{\theta}(x, \theta)\), or \(h(x, \theta) - \bar{h}(\theta)\), where \(h(x, \theta)\) is defined in Proposition \ref{P:SS20-Main-Result}. Since these functions satisfy \eqref{E:Centring-condition} and \eqref{E:Poisson-eq-H-growth}, the associated Poisson equation admits a solution satisfying the growth bounds in \eqref{E:Poisson-eq-Sol-growth}.

\subsubsection{Necessary notions from Malliavin Calculus}
\label{PrelimsOnMalliavin}
\noindent We outline the main tools from Malliavin calculus that are needed in this
paper. For a complete treatment, we refer the reader to
\cite{nualart2006malliavin}.

Broadly speaking, the Malliavin derivative of a random variable is the Fr{\'e}chet derivative taken along the directions in the Wiener space. It measures the sensitivity of the random variable to infinitesimal
perturbations of the underlying Brownian path at time $t$. To make it precise, let $\frak{H} = L^2
 \left(\mathbb{R}_+\right)$ and consider the isonormal Gaussian
 process $\left\{ W(h)
   \colon h \in \frak{H} \right\}$, that is, the collection of centered
 Gaussian random variables with covariance given by
\begin{equation*}
\mathbb{E}\left( W(h)W(g) \right) = \left\langle h,g \right\rangle_{\frak{H}}.
\end{equation*}
Denote by $\mathcal{S}$ the set of smooth cylindrical random variables
of the form $F = f
\left( W(\varphi_1), \cdots , W(\varphi_n) \right)$, where $n \geq 1$, $\{\varphi_i\}^n_{i=1} \subset
\mathfrak{H}$, and $f \in C_b^{\infty} \left(
  \mathbb{R}^n \right)$ ($f$ and all of its partial derivatives of all
orders are bounded functions). The Malliavin derivative of such a smooth cylindrical random variable
$F$ is defined as the $\mathfrak{H}$-valued random variable given by
\begin{equation*}
DF = \sum_{i=1}^n \frac{\partial f}{\partial x_i} \left( W(\varphi_1),
  \cdots, W(\varphi_n) \right)\varphi_i.
\end{equation*}
The derivative operator $D$ is closable from $L^2(\Omega)$
into $L^2(\Omega ; \mathfrak{H})$, and we continue to denote by $D$
its closure, the domain of which we denote by $\mathbb{D}^{1,2}$, and which is a Hilbert space in the Sobolev-type norm
\begin{equation*}
\left\lVert F \right\rVert_{1,2}^2 = \BE (F^2) + \BE \left( \left\lVert DF \right\rVert_{\mathfrak{H}}^2 \right).
\end{equation*}
Similarly, one can obtain a derivative operator $D
\colon\mathbb{D}^{1, 2}(\mathfrak{H}) \to L^2(\Omega; \mathfrak{H}
\otimes \mathfrak{H})$ as the closure of \\ $D \colon L^2(\Omega;
\mathfrak{H}) \to L^2(\Omega; \mathfrak{H} \otimes \mathfrak{H})$. We then set $D^2F = D(DF)$ and define the Sobolev-type norm
\begin{equation}\label{E:Sobolev-Norm}
\left\lVert F \right\rVert_{2,4}^4 = \BE (F^4) + \BE \left( \left\lVert DF \right\rVert_{\mathfrak{H}}^4 \right)+ \BE \left( \left\lVert D^2F \right\rVert_{\mathfrak{H}^{\otimes 2}}^4 \right).
\end{equation}
Note that more generally with $p,k > 1$ and , one can analogously obtain $\mathbb{D}^{k, p}$ as Banach spaces of Sobolev type by working with $L^p(\Omega)$.

At several places throughout this paper, we make use of the following consequence of the Malliavin differentiability theorem \cite[Theorem 2.2.1]{nualart2006malliavin}. Consider the {\sc sde}: $dX_t = b(t,X_t)\thinspace dt + \sigma(t, X_t)\thinspace dW_t$.
Then the Malliavin derivative $D_rX_t$ satisfies $D_rX_t = 0$, for $r>t$, while for $r \le t$,
$$D_rX_t = \sigma(r,X_r) + \int_r^t \partial_x b(s, X_s)\cdot D_rX_s \thinspace ds + \int_r^t \partial_x \sigma(s, X_s)\cdot D_rX_s \thinspace dW_s.$$
In particular, when the diffusion coefficient $\sigma$ is independent of the state variable, i.e., $\sigma(t,x)=\sigma(t)$, the stochastic integral vanishes. Consequently, the Malliavin derivative satisfies the deterministic integral equation
$D_rX_t = \sigma(r) + \int_r^t \partial_x b(s, X_s)\cdot D_rX_s \thinspace ds,$ for $r \le t$.

\bibliographystyle{alpha}
\bibliography{References}
\end{document}